\documentclass[article,11pt,a4paper,final]{amsart}

\usepackage[margin=1in]{geometry}
\usepackage{amsmath}
\usepackage{amssymb}
\usepackage{amsthm}
\usepackage{amscd}
\usepackage{mathtools}
\usepackage[T1]{fontenc}
\usepackage[utf8]{inputenc}
\usepackage{bbm}
\usepackage[overload]{textcase}
\usepackage[dvipsnames]{xcolor}
\usepackage{esint}
\usepackage{MnSymbol,wasysym}
\usepackage{dutchcal}
\usepackage{xcolor}
\usepackage{titlesec}
\usepackage{parskip}
\usepackage{biblatex}

\titleformat*{\section}{\Huge\bfseries}
\titleformat*{\subsection}{\LARGE\bfseries}
\titleformat*{\subsubsection}{\LARGE\bfseries}
\titlespacing*{\section}{0cm}{3cm}{3cm}
\titlespacing*{\subsection}{0cm}{2cm}{2cm}

\titlelabel{\thetitle.\ }
\usepackage{tocloft}

\usepackage{graphicx}
\graphicspath{ {./images/} } 

\newcommand{\R}{\mathbb{R}}

\newcommand{\T}{\mathbb{T}}

\newcommand{\norm}[1]{\|#1\|}
\newcommand{\eps}{\epsilon}
	
\newcommand{\loc}{\mathrm{loc}}

\newcommand{\dt}{ dt}

\theoremstyle{definition}

\newtheorem*{kom}{Comment}
\newtheorem{tw}{Theorem}

\newtheorem{defi}{Definition}
\newtheorem*{dow}{Proof}
\newtheorem{uw}{Remark}

\newtheorem{lem}{Lemma}
\newtheorem{prop}{Proposition}
\newtheorem{przyk}{Example}

\newtheorem*{dzieki}{Acknowledgement}
\newtheorem*{dow2}{Proof of Theorem \ref{main}}
\newtheorem*{dow1}{Proof of Theorem \ref{existence}}
\newtheorem*{tw*}{Theorem 1}

\DeclareMathOperator{\esssup}{ess\,sup}

\DeclareMathOperator{\diverg}{div}

\DeclareMathOperator{\supp}{supp}

\DeclareMathOperator{\tr}{tr}
\DeclareMathOperator{\intel}{int}
\DeclareMathOperator{\capac}{Cap}

\renewcommand{\epsilon}{\varepsilon}
\renewcommand{\phi}{\varphi}

\begin{document}
\begin{titlepage}
   \begin{center}
University of Warsaw \\
Faculty of Mathematics, Informatics and Mechanics\\
Institute of Applied Mathematics and Mechanics
       \vspace{6cm}
       
       \Huge{\textbf{Partial Differential Equations with a Singular Measure or Weight}}

       \vspace{3cm}

       \large{\textbf{Łukasz Chomienia}}

\vspace{1cm}
     
Supervisor:\\
dr hab. Anna Zatorska-Goldstein, prof. UW

       \vfill
            
       April 2024
    \end{center}
\end{titlepage}

\newpage
\
\newpage

\thispagestyle{plain}
\begin{center}
    \large
    \textbf{Abstract}
    \vspace{0.2cm}
\end{center}

This thesis pertains to the study of elliptic and parabolic partial differential equations on "thin" structures. These structures are closed sets $S$ embedded in Euclidean space $\mathbb{R}^3$ of the form $S = \bigcup_{i=1}^m S_i$, where $S_i$ are smooth manifolds with dimensions $\dim S_i \in \{1,2\}$. 

In this setting, a singular measure $\mu$ with support on $S$ is defined and equipped with a tangent bundle. To do this, $N_{\mu}$ is defined as the set of compactly supported vector fields $w \in C^{\infty}_c(\mathbb{R}^3;\mathbb{R}^3)$ such that $w=\nabla v$ on $\text{supp} \ \mu$ for some $v \in C^{\infty}_c(\mathbb{R}^3)$ that vanishes on $\text{supp} \ \mu$. The $\mu$-tangent space $T_{\mu}(x)$ at $x$ is defined as the $\R^3$-orthogonal completion of the set $ \{w(x) \in \mathbb{R}^3: w \in N_{\mu}\}$. 

The $\mu$-related gradient $\nabla_{\mu}$ is established by the orthogonal projection $P_{\mu}$ on the tangent bundle $T_{\mu},$ namely $\nabla_{\mu}\phi=P_{\mu}\nabla\phi$.
The focus is on the space $H^1_{\mu}$, which is a completeness of $C^{\infty}_c(\mathbb{R}^3)$ in the norm $\| \cdot \|_{\mu}:=\left( \| \cdot \|_{L^2_{\mu}}^2 + \| \nabla_{\mu} \cdot \|^2_{L^2_{\mu}}\right)^{\frac{1}{2}}$. 

Low-dimensional elliptic-type problems of the form $\int_{\Omega}A_{\mu}\nabla_{\mu}u \cdot \nabla_{\mu}\phi d\mu = \int_{\Omega}f\phi d\mu$ are considered within this setting. 

For strong-form parabolic problems, a non-local setting based on the Bouchitt{\'e}-Fragl{\`a} framework \cite{Bou02} is proposed. A space of pairs $(u,b)$ is introduced, where $u\in H^1_{\mu}$ and $b$ is a Cosserat vector field playing the role of an artificial normal component of the classical gradient. Specifically, it is demanded that $\nabla_{\mu} u+b \in H^1_{\mu}$. In this space, low-dimensional counterparts of classical second-order operators are introduced.

The theory of partial differential equations consistent with the low-dimensional variational calculus and previous results related to weak elliptic problems is developed. The thesis consists of two papers authored by the researcher: "Parabolic PDEs on low-dimensional structures"\cite{Cho24} and "Higher regularity of solutions to elliptic equations on low-dimensional structures"\cite{Cho23}.

The first main objective of the thesis is to establish the strong and weak low-dimensional counterparts of the parabolic problem $\partial_t u - \diverg(B\nabla u) = 0$ with the Neumann boundary condition and initial datum $u = g,$ defined on the low-dimensional structure $S$. The suitable measure-theoretic second-order framework is developed by analogy with the first-order setting sketched above. Fundamental outcomes regarding the existence and uniqueness of solutions are established by combining the two crucial facts: first, the low-dimensional counterpart $L_{\mu}$ of the operator $\diverg(B\nabla u)$ is closed, and second, the operator $L_{\mu}$ generates a contraction semigroup.

The main technical result is achieving the closedness of the low-dimensional second-order $L_{\mu}$. This is done by applying special geometric extensions of functions defined on $S$ with the convergence results of the second-order operators, and new characterisations of the involved spaces of functions. 

To construct a semigroup generated by $L_{\mu}$, a variant of Magyar \cite{Mag89} of the Hille-Yosida Theorem for non-invertible operators is adapted. The idea is to construct the semigroup by "forward" iterations $\frac{1}{k!}L_{\mu}^{(k)} = \frac{1}{k!}(L_{\mu} \circ... \circ L_{\mu}).$ The proposed method only accesses highly regular initial data $g \in \bigcap_{k=1}^{\infty}D(L^{(k)}_{\mu}).$ 

An alternative direction of study is presented to extend the class of accessible initial data. Weak-type parabolic problems are defined, and the existence of solutions is obtained by the application of the Lions version of the Lax-Milgram Lemma. It is also shown that the obtained weak solutions are regularized by demonstrating that they belong to the space of solutions to a stronger type of problem. 

The asymptotic behaviour of parabolic solutions is studied to establish a connection between parabolic and elliptic problems.

The second aspect of the thesis is to examine the higher regularity of weak solutions to the abovementioned elliptic problems. A componentwise $H^2(S_i)$ regularity is proven as the most elementary regularity result. 
Due to the non-standard geometry of the set $S$ the problems with shifts $u(\cdot + hv),\; v\in S$ of weak solutions $u\in H^1_{\mu}$ occur. To circumvent the related issues suitable extensions of solutions are constructed. 

Our proposed method involves extending the solution to the entire space using a formula that can be formally represented as:
"$\widetilde{u}(x,y,z) := u(0,y,z) - (\tr^\Sigma u)(0,y,0) + u(x,y,0)$".
This approach has been combined with componentwise estimates, trace estimates, and other tools to conclude that the $H^1_{\mu}$ space is closed with respect to the generalised shifts of functions. Next, the uniform upper bound for the generalised difference quotients is obtained. While this result confirms that the solution satisfies $u\in H^2(S_i)$, it falls short of our expectations as the correspondence in the second-order behaviour between various component manifolds is not controlled. 

The essential regularity result is established by rectifying the obtained outcome. This theorem proves that for any weak elliptic solution $u\in H^1_{\mu}$ exists some Cosserat vector field $b$ witnessing the membership of $u$ to the domain of the non-local second-order operator $L_{\mu}.$ 

The continuity of weak solutions in neighbourhoods of junction sets $S_i\cap S_j$ is investigated by connecting the established results with the Sobolev-capacity theory and facts related to the global behaviour of the partial traces. As a result, it is concluded that the solutions are indeed continuous.

\vspace{3cm}

\textbf{2020 Mathematics Subject Classification:} 35K10, 35K65, 28A25, 47D06 35J15, 35B65, 35R06, 35D30.\\

\textbf{Key words and phrases:} non-standard domains, rectifiable sets, generating semigroup, second-order parabolic equation, existence and uniqueness of solutions, weak solutions, regularity, singular measures.\\

\begin{dzieki}
I would like to thank my supervisor -- Professor Anna Zatorska-Goldstein, for her continuous support and help, and to Professor Piotr Rybka for suggesting many interesting ideas. \\
The author was in part supported by the National Science Centre, Poland, by the Grant: 2019/33/B/ST1/00535. 
\end{dzieki}

\newpage

\thispagestyle{plain}
\begin{center}
    \large
    \textbf{Streszczenie}
    \vspace{0.2cm}
\end{center}

Niniesza rozprawa odnosi się do badania eliptycznych i parabolicznych równań różniczkowych cząstkowych na "cienkich" strukturach. Struktury te są domkniętymi podzbiorami $S$ przestrzeni euklidesowej $\mathbb{R}^3$ o postaci $S = \bigcup_{i=1}^m S_i,$ gdzie $S_i$ są gładkimi rozmaitościami wymiaru \hbox{$\dim S_i \in \{1,2\}.$}

Na $S$ rozważamy miarę $\mu$ oraz stowarzyszoną z nią wiązkę styczną. W tym celu wprowadzamy zbiór $N_{\mu}$ gładkich pól wektorowych o zwartych nośnikach $w \in C^{\infty}_c(\mathbb{R}^3;\mathbb{R}^3)$ takich, że $w=\nabla v$ na $\text{supp} \ \mu$ dla pewnej funkcji $v \in C^{\infty}_c(\mathbb{R}^3),$ która znika na $\text{supp} \ \mu$. Przestrzeń styczna $T_{\mu}(x)$ do miary $\mu$ w punkcie $x$ jest zadana jako dopełnienie ortogonalne w przestrzeni $\R^3$ zbioru $ \{w(x) \in \mathbb{R}^3: w \in N_{\mu}\}$.   

Gradient odpowiadający mierze $\mu$, oznaczany $\nabla_\mu,$ jest zdefiniowany przez rzut ortogonalny $P_{\mu}$ na wiązkę styczną $T_{\mu}.$ 
W tym kontekście rozważamy zarówno zagadnienia eliptyczne 
jak również silne zagadnienia paraboliczne. W tym celu wprowadzamy przestrzeń Sobolewa $H^1_{\mu}$, zdefiniowana jako uzupełnienie $C^{\infty}_c(\mathbb{R}^3)$ w normie $\| \cdot \|_{\mu}:=\left( \| \cdot \|_{L^2_{\mu}}^2 + \| \nabla_{\mu} \cdot \|^2_{L^2_{\mu}}\right)^{\frac{1}{2}}$. Wykorzystujemy także nielokalną teorię Bouchitt{\'e}'a-Frag\`alii wprowadzoną w pracy \cite{Bou02}. Rozważamy mianowicie przestrzeń par $(u,b),$ gdzie $u\in H^1_{\mu}$ oraz $b$ jest polem wektorowym Cosserata odgrywającym rolę sztucznej składowej normalnej klasycznego gradientu. W szczególności, wymagane jest aby $\nabla_{\mu} u+b \in H^1_{\mu}$. W przestrzeni tej zostają wprowadzone niskowymiarowe odpowiedniki klasycznych operatorów drugiego rzędu.
Rozważana teoria równań różniczkowych cząstkowych jest zgodna z niskowymiarowymi zagadnieniami wariacyjnymi oraz niskowymiarowymi problemami eliptycznymi słabej postaci. 

Rozprawa składa się z wyników otrzymanych przez autora w dwóch pracach: 
\begin{itemize}
\item \textit{Parabolic PDEs on low-dimensional structures}\cite{Cho24} 
\item \textit{Higher regularity of solutions to elliptic equations on low-dimensional structures}\cite{Cho23}. 
\end{itemize}

Pierwszym z głównych zagadnień poruszanych w rozprawie jest zdefiniowanie na strukturze $S$ silnego i słabego niskowymiarowego odpowiednika parabolicznego zagadnienia Neumanna postaci \hbox{$\partial_t u - \diverg(B\nabla u) = 0$} z danymi początkowymi $u = g.$ 
Zasadniczym rezultatem jest wykazanie istnienia i jednoznaczności rozwiązań. Dowód tego faktu bazuje na udowodnieniu dwóch zasadniczych faktów: operator $L_{\mu}$ będący niskowymiarowym odpowiednikiem klasycznego operatora $\diverg(B\nabla u)$ jest operatorem domkniętym, oraz operator $L_{\mu}$ generuje półgrupę zwężającą. 

Kluczowym wynikiem technicznym jest wykazanie domkniętości niskowymiarowego operatora drugiego rzędu $L_{\mu}.$ Dowód polega na konstrukcji specjalnych rozszerzeń funkcji zadanych na strukturze $S,$ zastosowaniu wyników dotyczących zbieżności operatorów drugiego rzędu, oraz użyciu nowych charakteryzacji przestrzeni funkcyjnych.

Konstrukcja półgrupy generowanej przez operator $L_{\mu}$ bazuje na zaadaptowaniu wariantu Magyara \cite{Mag89} Twierdzenia Hille'a-Yosidy. Idea zastosowanej metody polega na konstrukcji półgrupy za pomocą iteracji "w przód": $\frac{1}{k!}L_{\mu}^{(k)} = \frac{1}{k!}(L_{\mu} \circ... \circ L_{\mu}).$ Zaproponowane podejście wymaga bardzo regularnych danych początowych $g \in \bigcap_{k=1}^{\infty}D(L^{(k)}_{\mu}).$ 

W rozprawie badane jest także alteratywne podejście, którego celem jest otrzymanie rezultatów dla szerszej klasy danych początkowych. W tym celu wprowadzamy zagadnienia paraboliczne słabego typu. Istnienie rozwiązań otrzymujemy przez zastosowanie Lematu Laxa-Milgrama w wersji Lions'a . Wykazujemy także wyższą regularność słabych rozwiązań oraz 
ich asymptotykę. 

Drugim kluczowym zagadnieniem wchodzącym w skład rozprawy są badania dotyczące wyższej regularności wspomnianych wcześniej, problemów elitycznych. Podstawowy wynik dotyczy wyższej regularności typu $H^2(S_i)$ na wszystkich rozmaitościach składowych struktury $S.$ Niestandardowa geometria zbioru $S$ powoduje przeszkody w prawidłowym zdefiniowaniu przesunięcia słabego rozwiązania $u\in H^1_{\mu}$, tzn. zdefiniowaniu \hbox{$u(\cdot + hv)$}, dla $v\in S$.  Aby ominąć te trudności  konstruujemy odpowiednie rozszerzenia rozwiązań.

Nasza metoda polega na zastosowaniu rozszerzeń słabych rozwiązań do całej przestrzeni euklidesowej przy użyciu formuły, która w sposób nieformalny może zostać wyrażona następującym wzorem: 
\[
"\widetilde{u}(x,y,z) := u(0,y,z) - (\tr^\Sigma u)(0,y,0) + u(x,y,0)".
\]
Podejście to łączy wykorzystanie odpowiednich oszacowań na rozmaitościach składowych, estymacji śladów, oraz innych narzędzi w celu wykazania, że przestrzeń $H^1_{\mu}$ jest zamknięta względem operacji brania uogólnionych przesunięć funkcji. Wynikiem przeprowadzonego rozumowania jest konkluzja, iż rozwiązania posiadają regularność typu $u\in H^2(S_i).$ 

Dzięki zastosowaniu wyżej wspomnianego wyniku otrzymujemy główny rezultaty dotyczący regularności słabego rozwiązania $u\in H^1_{\mu}$ niskowymiarowego zagadnienia eliptycznego: istnieje pewne pole wektorowe Cosserata $b,$ które świadczy o tym, że funkcja $u$ należy do dziedziny nielokalnego operatora drugiego rzędu $L_{\mu}.$ 

Otrzymane rezultaty dotyczą również zachowania słabych rozwiązań w otoczeniu zbiorów styku rozmaitości składowych, tj. w otoczeniu zbiorów $S_i\cap S_j$. Dzięki zastosowaniu omówionych wyników wraz z użyciem pewnych faktów z teorii pojemności Sobolewa, oraz pewnych uzyskanych faktów dotyczących globalnego zachowania się śladów możemy wnioskować o ciągłości rozwiązań.

\newpage
\
\newpage

\newpage
{\LARGE\textbf{Contribution}}
\bigskip

As a part of my thesis research, I have referred to the paper titled "Higher regularity of solutions to elliptic equations on low-dimensional structures" \cite{Chom23}. This paper was co-authored with Michał Fabisiak, who is a doctoral candidate at the University of Warsaw. Therefore, I need to clarify my contribution to the paper. I assert that all the proposed ideas, methods, and their implementation included in the final version of paper \cite{Chom23} were developed and prepared by me. 
The second author proposed some minor ideas and refinements to the final version of the paper. 
\vspace{5cm}

Sign of Supervisor:
\vspace{3cm}

Sign of Co-author:

\newpage

\tableofcontents

\newpage

\newpage
\
\newpage

\section{Introduction}
\subsection{Brief discussion of the setting and the results}{$ $}

To give a proper meaning to the heat problem $\diverg(B\nabla u) = f$ on closed subsets $S \subset \R^3$ such that $S = \bigcup_{i=1}^m S_i,$ where $S_i$ are smooth manifolds of $\dim S_i \in \{1,2\},$ the measure-theoretic framework is combined with analytical tools. Let $\mu$ be a Radon measure supported on $S.$ The set \hbox{$N_{\mu}\subset C^{\infty}_c(\R^3;\R^3)$} is defined as $w\in N_{\mu}$ if and only if  $w=\nabla v$ on $\supp \mu$ for some $v\in C^{\infty}_c(\R^3)$ such that $v^{-1}(\{0\})\supset S.$ For a fixed $x \in S$ the subspace $N_{\mu}(x)$ is given as $\bigcup_{w\in N_{\mu}} w(x).$ The $\mu$-tangent space $T_{\mu}(x)$ at $\mu$-almost every $x \in S$ is the $\R^3$-orthogonal complement of $N_{\mu}(x).$ The $\mu$-tangent gradient $\nabla_{\mu}$ of $\phi \in C^{\infty}_c(\R^3)$ is the orthogonal projection $P_{\mu}$ of the classical gradient $\nabla \phi$ onto the bundle $T_{\mu}$, namely $\nabla_{\mu}\phi = P_{\mu}\nabla \phi.$  
The $\mu$-Sobolev space $H^1_{\mu}$ is a completion of $C^{\infty}_c(\mathbb{R}^3)$ in the norm $\| \cdot \|_{\mu}:=\left( \| \cdot \|_{L^2_{\mu}}^2 + \| \nabla_{\mu} \cdot \|^2_{L^2_{\mu}}\right)^{\frac{1}{2}}$.

In the proposed setting the weak variant of the heat equation is $\int_{\Omega}B_{\mu}\nabla_{\mu}u \cdot \nabla_{\mu}\phi d\mu=\int_{\Omega}f \phi d\mu,$ for any $\phi \in C^{\infty}_c(\R^3).$ Here $B_{\mu}$ is a suitably modified coefficient matrix. The existence of solutions to this problem is known \cite{Ryb20}. We are interested in regularity aspects. Before we discuss our related results more precisely, we need to clarify what kind of regularity is expected in this framework.

By analogy with the classical problems, we might look for higher regularity $u|_{S_i} \in H^2(S_i)$ on each component $S_i$ of the structure $S$. It turns out that this kind of regularity is necessary for the well-posedness of strong-form second-order problems, but far from being sufficient as it does not encode higher-order correspondence between behaviour on various components $S_i$ \hbox{of the set $S$.}

Let $Z$ be the set of pairs $(\phi, \nabla^{\perp} \phi),$ $\phi \in C^{\infty}_c(\R^3).$ The low-dimensional second derivative operator $A_{\mu}:Z \to (L^2_{\mu})^{3 \times 3}$ is given as $A_{\mu}(\phi, \nabla^{\perp} \phi)= Q_{\mu}(\nabla^2 \phi).$ Here $Q_{\mu}$ is the projection on the space of "tangent matrices". 

Taking the $L^2_{\mu}$-closure, the operator $A_{\mu}$ is extended, and its domain is denoted by $D(A_{\mu}).$ Intuitively speaking, one can think about the set $D(A_{\mu})$ as the space of pairs $(u,b),$ where $u\in H^1_{\mu}$ and $b\in L^2_{\mu}(\R^3;N_{\mu})$ is the Cosserat vector field - a normal to $S$ vector field satisfying \hbox{$\nabla_{\mu}(\nabla_{\mu}u +b)\in (H^1_{\mu})^{3\times 3}.$} Using the operator $A_{\mu}$ the second-order counterpart $L_{\mu}$ of the $\diverg(B\nabla u)$ can be properly introduced. The operator $L_{\mu}$ is defined on the subspace of $u\in H^1_{\mu}$ for which there exists the Cosserat field $b$ being the witness of $(u,b)\in D(A_{\mu}).$

In this setting, we establish our main elliptic regularity theorem of \cite{Cho23}. We prove that for any function $u\in H^1_{\mu}$ satisfying  $\int_{\Omega}B_{\mu}\nabla_{\mu}u \cdot \nabla_{\mu}\phi d\mu=\int_{\Omega}f \phi d\mu$ for any $\phi \in C^{\infty}_c(\R^3),$ exists the Cosserat vector field $b$ implying that $u \in D(L_{\mu}).$ The proof is based on the generalisation of the difference quotients method and is divided into two key steps. We begin by showing that the solution $u$ can be extended to the $\R^3$ in a way that the global behaviour of the extension $\widetilde{u}$ is controlled by the behaviour on the structure $S$. To this end we propose the extension relied on the abstract variant of the inclusion-exclusion formula: "$\widetilde{u}(x,y,z) := u(0,y,z) - (\tr^\Sigma u)(0,y,0) + u(x,y,0)$". This allows us to define the generalisations of difference quotients. By proving additional Sobolev-type regularity of traces of a solution on the sets of intersections of components $S_i,$ and by standard results concerning the higher regularity in regions separated from the junction sets $S_i\cap S_j,$ we conclude that the space $H^1_{\mu}$ is closed with respect to the generalised difference quotients. By establishing the uniform upper bounds to the generalized difference quotients, we conclude that $u\in H^2(S_i).$ 

In the next step we refine the obtained result. Applying new facts regarding the operator $L_{\mu}$ we conclude existence of the auxiliary vector field $b$ implying that $u \in D(L_{\mu}).$ Besides that we use certain tools of the Sobolev capacity theory to conclude the global continuity of weak solutions.

The second main subject of our research \cite{Cho24} is focused on the parabolic problem $\partial_t u = L_{\mu}u,$ being the low-dimensional counterpart of $\partial_t u=\diverg(B\nabla u).$ The approach to the existence of solutions is based on the operator semigroup theory - we construct the contraction semigroup generated by the operator $L_{\mu}.$ To this end we invoke the Magyar \cite{Mag89} variant of the Hille-Yosida Theorem. The crucial technical component of the proof is in verifying that the operator $L_{\mu}$ is closed in the sense of the $L^2_{\mu}$-convergence. Consider the sequence $u_n \in H^1_{\mu}$ converging in the $L^2_{\mu}$-norm such that $L_{\mu}u_n$ also converge in the $L^2_{\mu}$-sense. We can not ensure that the sequence of corresponding Cosserat vector fields $b_n$ will converge. To circumvent this difficulty, the modification of such a sequence is proposed. We construct $\widetilde{b_n}$ which convergence is controlled in terms of $u_n$ and $L_{\mu}u_n.$ The application of the mentioned Magyar result allows us to construct the semigroup in the process of "forward" iterations: $\frac{1}{k!}L_{\mu}^{(k)} = \frac{1}{k!}(L_{\mu} \circ... \circ L_{\mu}).$ The proposed method only accesses highly regular initial data of the class $\bigcap_{k=1}^{\infty}D(L^{(k)}_{\mu}).$

To extend the class of accessible initial data, we study weak versions of the parabolic problem $\partial_t u = L_{\mu}u.$ The existence and uniqueness of solutions is by adapting the Lions-Lax-Milgramm Lemma. Further, the extra regularity of weak solutions is examined by studying membership in the space of solutions of the more regular problem. The correspondence between parabolic and elliptic equations is examined. It is shown that weak parabolic solutions converge, as time goes to infinity, to solutions of the weak elliptic issue.

\subsection{Overview of the theory and state of the art}{$ $}

The field of engineering sciences poses several fundamental challenges that are difficult to overcome due to the unique geometry involved. One of such significant challenges is determining the optimal value of a quantity that is an outcome of a specific process on a "thin" fragment of the domain. This is particularly challenging when the distribution of the quantity or equation describing the process is only known in the entire domain. The problem arises due to the difference in scales between the space in which the law is given and the "thin" subset of interest.

Another issue that has gained wide attention in the field is the description of the distribution of physical values, such as the field of tension in "thin" objects embedded within a medium. In engineering sciences, this problem appears as the limit of a sequence of discrete issues used in numerical computations. Additionally, there is the classic physical problem of modelling the process of heat transfer in a conductor made up of thin parts that may have components of varying dimensions (see Figure \ref{fig1} below). The complexity of the underlying geometry makes it hard to address such problems using standard analytical methods. It is also important to note that similar challenges arise in the context of problems related to the optimal design \cite{Lew23}, the elasticity theory and the field of strength of materials \cite{Bol20}.
\begin{figure}[h]\label{fig1}
\centering
\includegraphics[scale=0.5]{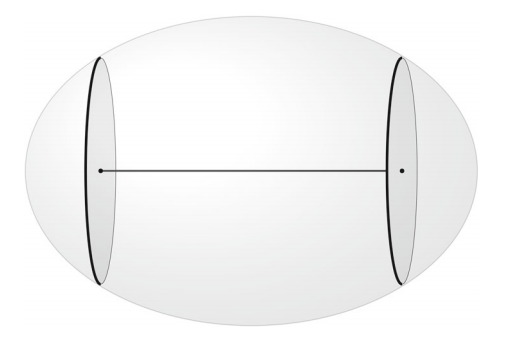}
\caption{Two discs glued to the interval - a simple model of a "thin" conductor.\\ The figure taken from \cite{Ryb20}.}
\end{figure}

There were several attempts at the analysis on "thin" subsets.  For example, in paper \cite{Ace91}, the authors use the method of fattening the lower dimensional manifold to avoid working with the sets of the zero Lebesgue measure. 
Unfortunately, all such attempts do not work properly in the case of irregular structures, for instance, when the structure has cones or consists of junctions of parts of different dimensions. 

If the structure that is described as "thin" exhibits a degree of smoothness that is suitable, it is possible to obtain accurate results through the employment of the fattening method, as, for instance, demonstrated in \cite{Fon98}. Generally speaking, this method involves considering the higher dimensional version of a variational problem in an $\epsilon$-neighborhood $Z_{\eps}$ of the aforementioned "thin" structure $Z$ and then taking the limit of $\epsilon$-solutions as $\epsilon$ approaches zero. A study of such convergence is usually done with the help of $\Gamma$-convergence. Definitions and facts related to the $\Gamma$-convergence theory of functionals are explained, for example, in book \cite{Mas93}. 

To clarify, a consideration of the minimization problem of the functional $J(u)=\int_Zf(\nabla u,u,x)$ is replaced with a study of the sequence of Euclidean problems $J_{\eps}(u)=\frac{1}{\eps}\int_{Z_{\eps}}f(\nabla u,u,x)$ in the $\epsilon$-neighborhood $Z_{\eps}.$ 

However, it should be noted that this approach has a limitation in that it cannot cover structures such as the one illustrated in Figure \ref{fig1} due to the required level of smoothness.  

Geometric structures that have cusps, corners, some kind of discontinuity, or are composed of parts with varying dimensions of the tangent space cannot be accurately approximated using a sequence of auxiliary problems defined on shrinking $\eps$-neighbourhoods. In such cases, any sequence of functions defined in these neighbourhoods and converging in any reasonable sense will produce a singularity in the limit. However, if the considered structure is at least twice continuously differentiable, the existence of suitable neighbourhoods is guaranteed by the existence of smooth normal vector fields. This follows from the Tubular Neighbourhood Theorem, well-known in the mathematical analysis.
The existence of $\eps$-neighbourhoods in cases where the dimension of the tangent structure varies cannot be addressed by single-scale convergence of neighbourhoods. This results in irregularities at the points of intersection of components of different dimensions. The singularities that arise in the limit produce a blow-up of the $\eps$-limit of solutions. 

A more modern approach has been designed to circumvent the abovementioned difficulties and covers the general case of irregular, lower-dimensional, closed subsets of the Euclidean space $\R^n$. This class of structures contains a variety of geometrical objects originating from various fields, including frames, CW-complexes, graphs, stratified manifolds, and many others. The central idea of this new approach is to interpret the "thin" subset of the Euclidean space as a particular measure that is singular with respect to the Lebesgue measure of the domain in which it is embedded. Further in the thesis, we will name such "thin" subsets or corresponding singular measures as low-dimensional structures. By equipping the related measure with the tangent bundle, various notions of differential geometry can be reformulated in the measure-theoretical setting. 
With this in hand, it is possible to transfer the analytical or geometrical framework into the introduced measure-related setting. Indeed, the first- and second-order variational theories have certain counterparts there. An advantageous feature of the new problem presentation method is its capability to define a reasonable differential calculus with a tangent structure given only at almost every point of the domain. This "inner" approach does not require the smoothness of structure or the existence of any well-converging sequence of open neighbourhoods. At first glance, it may seem that the measure-theoretic method described here cannot address second-order problems due to the discontinuities of the tangent bundle. 

Let us assume that $S\subset \R^2$ is a closed subset defined as 
$$S:=S_1 \cup S_2,\;\; S_1:=\{(x,0): x\in [-1,1]\},\;\; S_2:=\{(0,y): y\in [-1,1]\}.$$ Consider an arbitrary function $u:\R^2\to \R$ that is smooth as a function on $\R^2$ and in addition let assume that $\partial_xu(0,0)\neq \partial_yu(0,0).$ 

As the function $u|_S$ is continuous, differentiable in every point of $S$ and is obtained as a restriction of some $\R^2$-smooth function, it seems not controversial to expect that such function $u|_S$ is an element of any properly defined space of lower dimensional first-order functions on $S$. A discontinuity at the level of gradients can present challenges in managing higher-order derivatives. A natural way to circumvent such difficulties is to replace the "inner" gradient of the function $u|_{S}$ with a restriction of a classical gradient of the open set extension. However, using such restrictions may prove problematic in capturing the underlying geometry of a structure. Besides that, we are interested in considering less regular functions on the set $S$ that do not need to be restrictions of some regular function posed on an open set containing $S$. As such, it is imperative to explore alternative approaches to ensure a comprehensive understanding of the structure.
This indicates that when thinking of a higher-order regularity of functions defined on $S,$ we need to deal somehow with this kind of discontinuity in points of transversal intersection of components $S_1$ and $S_2.$

Although the jump type discontinuity of the gradient of functions supported on "thin" subsets of weak regularity may look like an obstacle to establishing a second-order theory, it has been shown that the problem can be addressed in the measure-related framework by introducing an artificial regularizing component of the "inner" gradient.  

The low-dimensional analysis based on a measure-theoretic framework is a well-established approach with firm foundations that has been validated positively by numerous applications. The underlying idea of this method is to represent a given geometric structure as a Radon measure equipped with a tangent bundle. Apart from its real-life applications, this approach has successfully resolved mathematical problems related to the minimization of functionals involving differential operators on subsets that are singular with respect to the Lebesgue measure. 

The simplest example is a question of minimizing the energy functional 
\begin{equation}\label{funkcjonal1}
F(\phi)=\int_{A} |\nabla \phi|^2
\end{equation}
on the "thin" set $A \subset \R^n,$ in the class of smooth functions $\phi \in C^{\infty}_c(\R^n)$ satisfying some kind of boundary conditions.  

The basics of this new theory were introduced by Bouchitté,  Buttazzo, and Seppecher in paper \cite{Bou97}. The authors study the issues of finding minimizers of functionals posed on sets which are singular with respect to the underlying Lebesgue measure, like in \eqref{funkcjonal1}. The direct approach to such issues demands certain assumptions from considered functionals, especially that it needs to be lower semicontinuous. However, the functional $F$ as stated in \eqref{funkcjonal1} is not lower semicontinuous in the given class of functions. As is common in related issues, the authors constructed its relaxation - the greatest lower semicontinuous envelope.

The relaxed functional $\widetilde{F}$ is expressed as 
\begin{equation}\label{funkcjonal2}
\widetilde{F}(\phi) = \int_{\R^n}|\nabla_{\mu}\phi|^2 d\mu.
\end{equation}
Here appeared new objects: the measure $\mu$ corresponding to the subset $A,$ and the $A$-essential part of gradient denoted by $\nabla_{\mu}.$   

In the measure-theoretic approach, a "thin" domain $S=\bigcup_{i=1}^mS_i \subset \R^n$ is paired with the singular measure $\mu$, which is supported in the set $S$ and encodes its geometry. Very often in applications, it is enough to use $\mu = \Sigma_{i=1}^m\mathcal{H}^{\dim S_i}\lfloor_{S_i},$ where $\mathcal{H}^k$ is the $k$-dimensional Hausdorff measure. This type of measures are sometimes named as the multijunction measures, \cite{Bou02}. The proper notion of the tangent structure $T_{\mu}$ is introduced to the measure $\mu$ to define the "thin" counterparts of classical differential operators. 

In this framework it is typically possible to give an explicit representation of the domain in which the relaxed functional is finite.  For instance, in the mentioned case \eqref{funkcjonal1} of the relaxation of $F$, to determine the domain of finiteness, it is enough to give an explicit form of the part of the gradient that acts on the low-dimensional, denoted as $\nabla_{\mu}.$ Actually, the operator $\nabla_{\mu}$ can be constructed as a measure-theoretical generalization of the gradient tangent to a smooth manifold, that is $\nabla_{\mu} u = P\nabla u,$ where $P$ stands for the orthogonal projection on the tangent bundle of the measure $\mu$. 
The gradient $\nabla_{\mu}u$ expresses the whole "first-order information" governed by the singular measure $\mu,$ or expressing it differently, the rest of the classical gradient, that is $\nabla u - \nabla_{\mu}u$ is not determined by the underlying measure $\mu.$ 
Let us consider the following trivial example.

Define $S=[-1,1] \times \{0\} \subset \R^2,\; \mu=\mathcal{H}^1\lfloor_{S}$ and consider $\phi_1,\; \phi_2:\R^2\to \R$ given by $\phi_1\equiv0$ and $\phi_2(x,y) = y.$ It is clear that $\phi_1=\phi_2$ on $S,$ but $\partial_y\phi_1(x,0) \neq \partial_y\phi_2(x,0).$

The measure-theoretical framework was later utilized to establish a "thin" version of the first-order Sobolev spaces denoted by $H^1_{\mu}$.
The concept of Sobolev spaces related to a singular measure was independently introduced by many authors using different approaches, see \cite{Lou14} for more discussion.  This kind of framework was first introduced in \cite{Bou97}. It was later reestablished in \cite{Haf03} in a slightly different setting relying on concepts from \cite{Fra99}.
The authors of \cite{Fra99} obtained a notion of the tangent bundle to a measure in a less complicated, and more adequate to apply in our further applications way. By analogy with a standard smooth setting, the class $N_{\mu}$ of $w\in C^{\infty}_c(\R^3;\R^3)$ satisfying $w = \nabla v,$ for some $v\in C^{\infty}_c(\R^3),$ where the $0$-level set of $v$ contains $S$ is considered. The space $T_{\mu}(x)$ is defined as the $\R^3$-orthogonal complement of the evaluation $N_{\mu}(x).$ This construction allowed for introducing the Sobolev-like space $H^1_{\mu}$ as a completion of smooth functions $C^{\infty}_c(\R^n)$ in the norm $(\int_{\Omega}|u|^2d\mu + \int_{\Omega}|\nabla_{\mu}u|^2 d\mu)^{\frac{1}{2}}.$ For a detailed discussion and comparison of different formulations of the setting, please refer to \cite{Fra99}. 

Problems on "thin" sets and the introduced framework have been studied by several authors, leading to further developments in the variational calculus on low-dimensional sets. For example, in paper \cite{Man05} the author refined the relaxation results given in \cite{Bou97}. He proposed a method of constructing the lower-semicontinuous envelope for functionals on a low-dimensional set $A$ of the form $\int_Af(\nabla u),$ where the integrand $f$ need not be convex. It was a significant improvement of the results of \cite{Bou97} as they worked only under the assumption of convexity of an integrand. In papers \cite{Bou01, Bou021} the low-dimensional analysis is connected with the $\Gamma$-convergence approach to study convergence in various scales. Moreover, in paper \cite{Bou01} weak partial differential equations on low-dimensional structures were considered for the first time. 

Recent studies conducted in \cite{But23} use the low-dimensional setting in order to consider a specific variant of the so-called mass optimization problem - the problem of minimization of $\min\{-E(\mu)+C(\mu)\},$ where $E$ is the energy functional, $C$ is the cost functional and the minimization is over the space of positive scalar measures $\mu.$ Later, the $\mu$-related low-dimensional analysis was applied in paper \cite{Bol22}, where the authors study links between the free material design problem, the mass optimization problem and the theory of Monge-Kantorovich. The mentioned results indicate that the low-dimensional theory is an active research topic and serves as a very useful tool with many applications.

The results listed above are characterized by the fact that functionals under examination are of the first-order; that is, the integrand depends only on the gradient and lower-order terms.
It was initially unclear how to approach the handling of higher-order problems. However, Bouchitt{\'e} and Fragal{\`a} made significant advancements in this area in their work \cite{Bou02}. Their paper focuses on variational problems of the second-order. They adapt the technicalities used to establish the low-dimensional first-order framework into a more complex second-order setting. One of the main challenges encountered in this process was the inability to decouple second-order derivatives from first-order derivatives. The low-dimensional second-order setting is the fundamental framework used in our research to study strong-type parabolic problems on the low-dimensional structures.

Let us focus on the structure $S=S_1\cup S_2 \subset \R^2$ such that $S_1=\{(x,0)\in \R^2: x\in [-1,1]\}$ and $S_2=\{(0,y)\in \R^2: y\in [-1,1]\}.$  
Moreover, let $v:S \to \R^2$ be a fixed tangent vector field on $S.$

As discussed earlier, the $\mu$-tangent gradient defined as $\nabla_{\mu}u=P\nabla u$ is the right low-dimensional counterpart of the classical gradient. This means that to establish the second-order differential operator on $S,$ we first need to focus on how $\nabla_{\mu}$ acts on a tangent field $v.$ A first trivial observation is that $v$ does not necessary belong to the domain of $\nabla_{\mu}$ due to a potential discontinuity in $S_1\cap S_2.$ Moreover, we should note that to consider $\partial_x v|_{S_1}$ we immanently need to take into account also $\partial_x v|_{S_2}$ in the $S_2$-neighbourhood of the junction point $S_1\cap S_2.$ Thirdly, the value of $\partial_x v|_{S_2}$ depends on the choice of the extension $\widetilde{v}:\R^2 \to \R$ of $v.$   
These observations suggest that the operator associated with the measure-related second-order derivative should not solely act on a single function supported on a lower-dimensional structure, but instead, it should also depend on the normal (to the structure $S$) component of the (full) gradient. Moreover, this is still not enough to ensure that the second derivative operator will be a well-defined single-valued operator. We should also consider not the standard Hessian but a projection onto the space of $\mu$-essential part. Otherwise, the second-order operator would be multivalued, which is an outcome that is rather not expected. This stands in a direct analogy with the first-order setting; above, we have shown in a simple example that the normal component of the classical gradient is not determined by the behaviour of functions on the structure, and thus, it is necessary to consider projection on the tangent bundle. 

In \cite{Bou02}, the authors resolve the mentioned problems by assuming that the measure-related second-order operator acts on pairs $(u,b)$, which consists of a function $u$ supported on a low-dimensional structure and an additional normal to the low-dimensional structure vector field $b,$ called later the Cosserat vector field. This vector field plays the role of the artificial normal component of the $\mu$-tangent gradient $\nabla_{\mu}$. Taking a sequence of smooth functions $w_n$ defined in the whole space, such that $w_n \to u$ and $\nabla^{\perp}w_n \to b$ in a suitable sense, we are able to define the low-dimensional second derivative operator $A_{\mu}(u,b)$ acting on pairs $(u,b),$ as the limit of $Q(\nabla^2 w_n),$ where the operator $Q$ is a certain kind of projection on the space of "tangent matrices".

The fact that the introduced second derivative involves acting on the mentioned Cosserat vector fields is the source of some original phenomena that appear in this theory and are not present in the classical (Euclidean) cases. For instance, it is worth noting that finding a suitable relaxation of the second-order functional in this context is not a local problem. This fact is clearly illustrated in the example given in \cite{Bou02} which we briefly sketch here.

Let $S \subset \R^2$ be an equilateral triangle with the prescribed measure $\mu = \mathcal{H}^1\lfloor_{S}.$ We consider a problem of finding a relaxation of the simple functional $G(u)=\int_S |\nabla u|^2.$ It is shown that the relaxed form of $G$ on the structure $S$ is expressed by the formula $\widetilde{G}(u) =\inf \{\int |u''|^2+2|b'|^2d\mu\}, $ where the infimum is taken over the set of pairs $(u,b)$ satisfying a certain compatibility condition. On the other hand, if we remove one side of the triangle, and denote the restriction of the measure $\mu$ to this structure by $\nu,$ then the relaxation of $G$ is given by $\overline{G}(u)=\int |u''|^2d\nu.$

The recalled non-locality is a unique feature of the low-dimensional variational setting.

The development of variational theories in the context of singular subsets has raised questions about differential equations in this area. In the Euclidean setting, there is a well-known relationship between the calculus of variations and the theory of partial differential equations. Therefore, it is natural to ask about a similar relationship in the low-dimensional framework. Surprisingly, this aspect has not been widely studied. Partial differential equations appeared for the first time in the low-dimensional context in paper \cite{Bou01}, but in a form connected with the measure-theoretic two-scale convergence method used there; not as an independent object of studies. The first publication focused on this subject was \cite{Ryb20}, published in 2020, aiming to provide a rigorous framework for the problem of heat dissipation in a "thin" conductor. The authors of the article studied the class of weak elliptic Neumann problems in the low-dimensional framework and sought to establish the existence of solutions in the Sobolev-type space $H^1_{\mu}$. More precisely, they examine if there exist a function $u \in H^1_{\mu}$ satisfying for all $\phi \in C^{\infty}(\R^d)$ 
\begin{equation}\label{bla1}
\int_{\Omega}A_{\mu}\nabla_{\mu}u \cdot \nabla_{\mu} \phi d\mu = \int_{\Omega}f\phi d\mu, 
\end{equation}
where $\nabla_{\mu}$ is a $\mu$-tangent gradient, $A_{\mu}$ is a suitable relaxation of the matrix of coefficients, and $\mu$ is a measure corresponding to the given low-dimensional structure.

The proof is based on the fact that this issue can be treated as the Euler-Lagrange equation for some low-dimensional variational problem. The existence and uniqueness of solutions are obtained by solving a minimization problem of the corresponding energy functional. An important technical contribution of the paper is in establishing a new variant of the Poincar{\'e}-type inequality. Problems with the classical Poincar{\'e} inequality appear if a point of junction is small in the sense of the Sobolev capacity (for basic facts and definitions related to the Sobolev capacity theory, we refer to \cite{Eva15}). 

Paper \cite{Ryb20} proposes a weaker variant of the Poincare inequality valid on the general class of low-dimensional structures, see Subsection 2.2.

However, it is important to note that further properties of such solutions have not been discussed, nor have other types of equations been examined. This lack of further investigation, coupled with the intriguing and unique phenomena that emerge in low-dimensional problems, serves as a key motivator for our current research.   

\subsection{Results of our research}{$ $}
\vspace{0.5cm}

The research presented in the thesis is focused on various types of partial differential equations defined on low-dimensional structures. We can distinguish two primary objectives of our studies.

The first part is devoted to establishing the theory of strong and weak parabolic equations in the low-dimensional setting. This encompasses various types of time-dependent equations and their relation with other variational theories. This type of result is the outcome of our paper "Parabolic PDEs on low-dimensional structures" \cite{Chom24}.

The second objective of our studies is to develop the regularity theory for low-dimensional weak elliptic problems. The goal of the research program related to this aspect is to provide correspondence between different notions of equations and variational issues. The results connected to this subject are based on our paper "Higher regularity of solutions to elliptic equations on low-dimensional structures" \cite{Chom23}.

As both directions of our work are new, we have had to develop a suitable framework and extend the existing one. This includes in-depth studies of related spaces of functions, constructing appropriate generalizations of operators, and studying their functional-theoretic properties. The obtained results are useful in themselves and can be applied in future studies in the low-dimensional setting.

\subsubsection{Low-dimesional parabolic problems}{$ $}
\vspace{0.5cm}

The results established in this part come from our paper "Parabolic
PDEs on low-dimensional structures" \cite{Cho24}.

Our research presents a novel contribution to the field of partial differential equations by developing a theory applicable to irregular low-dimensional structures consistent with the variational and elliptic setting. Our work yields new results related to both strong-form second-order issues as well as weak-formulations. We establish key properties of second-order differential operators and prove the main theorem regarding the existence and uniqueness of solutions to second-order problems with the Neumann initial data. In order to broaden the class of accessible initial data, we also investigate parabolic problems of weak-type and examine the regularity of weak solutions. Additionally, we explore the asymptotic behaviour of parabolic solutions, establishing a connection between solutions to low-dimensional parabolic and elliptic equations. Our work relies on the development of new characterizations of Sobolev-type spaces on low-dimensional structures and the establishment of new functional-theoretic properties of low-dimensional differential operators.   

Let $S = \bigcup_{i=1}^mS_i \subset\R^3,$ where $S_i$ are compact, smooth submanifolds (with boundary) of $\R^3,$ which are pairwise transversal and $\dim S_i \in \{1,2\}.$
The goal of this work is to establish the right meaning of the following formal Neumann boundary problem on the geometric structure $S,$   
\begin{equation}\label{formal}
\begin{matrix}
u_t - \diverg(B\nabla u) = 0 & \text{ in } & S \times [0,T], \\ \\
B\nabla u \cdot \nu = 0 & \text{ on } & \partial S \times [0,T], \\ \\
u = g & \text{ on } & S \times \{0\},
\end{matrix}
\end{equation}
where the matrix of coefficients $B$ satisfies a suitable ellipticity condition, $(B\nabla u \cdot \nu)\lfloor_{\partial S}$ is an appropriate normal derivative and $g$ is a given initial data. 

To translate the problem into a suitable framework, we associate with the structure $S$ a corresponding singular measure $\mu,$ which encodes the geometry of the set $S.$ By an application of the second-order setting of \cite{Bou02}, we define a counterpart of the operator $\partial_t  -L,$ where $Lu=\diverg(B\nabla u)$ is related with the underlying measure $\mu.$ 

The main result deals with the existence of solutions:
\begin{itemize}
    \item[$\bullet$] The low-dimensional conterpart of Problem \eqref{formal} has a unique solution $u,$ assuming that the initial data $g$ is sufficiently regular. This is the statement of Theorem \ref{existence} in Section 4.
\end{itemize}
For definitions of the involved operators (Definition \ref{secondop} and Definition \ref{lap}) we refer to Section 2. Reformulated Problem \eqref{parabolic} and Definition \ref{parabolic2} of a solution also can be found in Section 2 of the thesis.
The proof applies the general semigroup theory of differential operators (for example, see \cite{Paz83}) and a specific variant of the Hille-Yosida Theorem about operators generating contraction semigroups established in \cite{Mag89}. A core of this approach is in avoiding using a resolvent operator, which turns out to be not well-defined due to the lack of invertibility of the considered operator. 

A crucial technical component of the proof is verifying the following property:
\begin{itemize}
    \item [$\bullet$] The low-dimensional realisation of the operator $\diverg(B\nabla u)$ is a closed operator. This is the content of Theorem \ref{main} in Section 4.
\end{itemize}
Definitions of the mentioned differential operator and its domain are located in Section 2 (Definition \ref{secondop} and Definition \ref{0domain}, respectively).

In our reasoning, we use a second-order functional space consisting of pairs $(u,b),$ where $u$ is a function that belongs to a proper subspace of the Sobolev-type space $H^1_{\mu}$ related to $\mu$, and $b$ is a Cosserat vector field normal to the low-dimensional structure $S.$ The Cosserat vector field $b$ is a type of normal field that imitates the normal component of the classical gradient $\nabla u.$ We can use the framework established in \cite{Bou02} to create the low-dimensional counterpart of \hbox{Problem \eqref{formal}.} By making use of the $\mu$-related second-order operators defined there, we can construct the operator $L_{\mu},$ which is a low-dimensional generalization of $\diverg(B \nabla u).$ 
This implies that the second-order operator need also to be defined on the space of pairs $(u,b).$ On the other hand, as our aim is to generalise the classical parabolic problems posed in Euclidean domains, on smooth manifolds, and we aim for consistency with the low-dimensional stationary problems defined in \cite{Ryb20}, the operator $L_{\mu}$ cannot depend on the choice of the particular Cosserat vector field $b.$ Indeed, the proposed construction of the second-order equation operator $L_{\mu},$ see Definition \ref{secondop} in Section 2, satisfies all of the expected properties.  

The initial step needed in our construction of solutions to the low-dimensional counterpart of the Neumann problem, see Definition \ref{parabolic2}, is in proving the closedness of the operator $L_{\mu}.$ It turns out that this issue is much more challenging than the closedness of the second derivative operator introduced in \cite{Bou02}.

It is worth noting that for a given value of $u$, the vector field $b$ is not uniquely determined, and there is insufficient information regarding its behaviour. Thus, issues arise in controlling the convergence of sequences of tuples $(u_n,b_n)$. Even in scenarios where the sequence $u_n$ converges in a suitable strong sense and it is established that $L_{\mu}u_n$ also converges, concluding that the limit of the sequence $u_n$ also has a corresponding Cosserat field is unfeasible. 

To address these challenges, we propose a procedure for modifying the vector field $b_n$ for a given pair $(u_n,b_n)$. As the process is based on special geometrical constructions, we localise it to reduce the initial structure to a set of generic parts. This can be done with the help of our new characterisations of the low-dimensional Sobolev-type spaces. The method of constructing the new sequence of Cosserat vector fields generates a new normal vector field $\widetilde{b}_n$, which possesses properties comparable to the original field $b_n$. The rationale behind introducing the modified sequence of vector fields $\widetilde{b}_n$ is that its convergence can be entirely controlled in terms of the corresponding functions $u_n$. Combining these conclusions with the closedness results given in \cite{Bou02}, Lemma \ref{comp} in Section 3, characterisations of membership in the domain of the second-order derivative operator and some additional facts related to well-posedness of Whitney-type extensions, we are able to derive that the limit term $u$ ($u_n\to u$) is also equipped with the proper Cosserat vector field $b.$ Finally, this implies that the equation operator $L_{\mu}$ is closed in a considered sense. 

The subsequent task of our programme of establishing the existence of strong parabolic solutions relies on adapting the Magyar variant of the Hille-Yosida theorem proposed in paper \cite{Mag89}. The essential obstacle that we encounter is lack of the well-defined resolvent operator corresponding to the equation operator $L_{\mu}.$ Indeed, the operator $L_{\mu}$ is not surjective in the expected sense, which implies that the resolvent can not be defined. Actually, the approach presented in \cite{Mag89} circumvents the use of the resolvent operator. Applying the results of \cite{Mag89} to the considered framework, we construct the contraction semigroup generated by the operator $L_{\mu}.$ The method of construction of the generated semigroup bases on a "forward" iteration process, namely the semigroup is obtained as a limit $k\to \infty$ of the subsequent iterations $$\frac{1}{k!}L_{\mu}^{(k)} = \frac{1}{k!}(L_{\mu} \circ... \circ L_{\mu}).$$
This method of construction of solutions works well in our problem, but it has a drawback that is inseparable from its nature.    

The applied method forces narrowing the set of accessible initial data to the class $\bigcap_{k=1}^{\infty} D(L_{\mu}^{(k)})$ of functions that are smooth with respect to the operator $L_{\mu}.$
To broaden the class of accessible initial data, we explore various versions of weak formulations of the parabolic problem (Section 5) and study their regularity. Proving the existence of solutions to weak variants of the considered issue is less complicated than that of the strong formulation. To this end, we adapt the Lions version of the Lax-Milgram theorem to the considered setting. Further, our focus is on verifying the membership of the obtained weak solutions in the space of solutions to equations of a more regular form (yet less regular than the initially considered strong-type setting). We derive important information about the regularity of weak solutions and demonstrate the connection between parabolic and elliptic problems.

We examine the asymptotic behaviour of weak-type solutions. We show that in the long time weak parabolic solutions converge to weak solutions to the low-dimensional elliptic problem (Subsection 5.3). The results are obtained by the use of the methods proposed in \cite{Gol08} for the classical $p$-Laplace equations.

\subsubsection{Higher regularity of solutions to elliptic problems.}{$ $}
\vspace{0.5cm}

The discussed results are established in our paper "Higher regularity of solutions to elliptic
equations on low-dimensional structures" \cite{Cho23}.

The second major objective of our project is focused on validating the higher regularity of low-dimensional solutions of elliptic problems. This is a fundamental and preliminary step towards establishing the relationship between this category of equations and the second-order variational theory of Bouchitt{\'e} and Fragal{\`a}, as introduced in their publication \cite{Bou02}, or between the strong-form parabolic problems examined by the author in \cite{Cho24}.

In \cite{Ryb20}, the authors prove the existence of weak solutions to elliptic problems considered on a general class of glued manifolds of potentially different dimensions (Equation \eqref{bla1}). This naturally leads to the question of the higher regularity of said solutions. In the low-dimensional setting the most fundamental notion of regularity is the higher Sobolev regularity on the component manifolds of the given low-dimensional structure. This kind of regularity is required for further studies but is far from sufficient from the perspective of the low-dimensional structures theory since it does not capture any additional connection between higher-order behaviour on different component manifolds of the structure. It turns out that a membership of a weak solution in the domain of the second-order operator introduced in \cite{Bou02} is the most adequate kind of differential regularity in this framework. 
Indeed, the mentioned framework is rich enough to encode the geometry of the structure, provides correspondence between regularity on various components and arises naturally in low-dimensional variational problems.     
   
Firstly, we address the aforementioned problem of upgrading the elliptic regularity in the classical Sobolev sense on components. This is the starting point for further improvements.   
Further, we examine other regularity-related properties of weak solutions. Combining the local regularity result with some facts from the capacity theory, we conclude that weak solutions are continuous. A different critical aspect of our study uses technical properties of the low-dimensional second-order operator to prove the membership of weak solutions in the proper higher-order space of functions.

We are especially interested in the elliptic problems on structures with one- or two-dimensional parts embedded in $\R^3$, and throughout the research, we restrict our attention to this setting. Such restriction is motivated by physical applications (see, for example, \cite{Lew23}) as well as the application of our earlier results about parabolic problems, which use methods designed to work in the case of structures consisting of at most two-dimensional component manifolds. It is expected that our research may be generalised to analogous problems in an arbitrary $n$-dimensional Euclidean space, but this needs further non-trivial improvements. To avoid technical complications, we also impose additional restrictions on the class of considered structures. Nonetheless, we believe that the presented methods are likely to carry over to more general geometric structures. Potential generalisations are discussed in the next section.
 
The following facts concerning an additional regularity of low-dimensional weak solutions to elliptic equations constitute the main results of our studies:
\begin{itemize}
\item On each component manifold $S_i$ of the low-dimensional structure $S$ a weak solution $u$ of the elliptic issue \eqref{bla1} has the extra regularity $u \in H^2(S_i).$ This is the statement of Theorem \ref{globreg} in Section 6.$ $\\

\item A weak solution $u$ is globally continuous on the given low-dimensional structure, that is $u \in C(S).$ This fact is precisely expressed in Theorem \ref{ciaglosc_ogolne} in Section 6.$ $\\

\item If $u$ solves weak problem \eqref{bla1}, then $u$ is a member of the domain of the low-dimensional second-order derivative operator $L_{\mu}.$ This is the result of Theorem \ref{nalezenie} in Section 6.
\end{itemize}

The proof of the result of the first point is technical and combines new constructions and methods dedicated to the considered low-dimensional framework as well as modifications of the classical facts. As our approach is closely related to the geometry of the examined structure, before we proceed to the proper consideration, we prove that the study of the structure of a general form might be replaced by examining a set of generic types of substructures of the form $S_i\cap S_j,\;  S_i\cap S_j\neq \emptyset.$   The main obstacle in showing componentwise higher regularity lies in defining a proper notion of translation $u(\cdot + hv),\; h>0,\; v\in S$ of the function $u$ supported on the "thin" subset $S$. Problems arise due to the non-trivial geometry of a domain.
Let us discuss a more concrete case. 

Suppose that $u$ is a function supported on $S \subset \R^3,\; S=D_1\cup D_2,$ with $D_1$ being the unit disc with origin at zero in variables $x,y$ and $D_2$ analogous disc but in variables $y,z.$ Moreover, let the support of $u$ be contained within the set $(1-\eps)S$ for some small $\eps >0.$ 

It is worth noting that determining the shift of the function $u$ in the direction parallel to the set of the intersection of two component manifolds, that is, as $\Sigma=S_1\cap S_2 = \{(0,y,0): y \in [-1,1]\},$ is a straightforward task. Precisely, defining the function $u(\cdot + he_y)$, where $|h|$ is sufficiently small, poses no problems. However, determining the shift in the direction of the variable $x$ or $z$ is not immediately apparent as we do not know the value of the shift on the component orthogonal to the direction of shifting.

As we aim to generalise the difference quotients method (see for instance \cite{Eva10}), it is necessary to address this question. We propose a construction based on a properly chosen sequence of well-behaved extensions of weak solutions. 

For instance, in the setting of the structure $S$ considered in the above-mentioned example, we would intuitively like to extend the function $u:S \to \R$ to $\widetilde{u}:\R^3 \to \R$ by the formula 
\begin{equation}\label{rozszerzenie}
\text{"}\widetilde{u}(x,y,z) := u(0,y,z) - (\tr^\Sigma u)(0,y,0) + u(x,y,0)\text{"},
\end{equation}
where $\tr^{\Sigma}$ denotes the trace on the intersection $\Sigma.$
Above, we used quotation marks because the presented formula expresses a general idea standing behind the proposed extension. The formal construction of the extension is quite involved and technical, and it is the crucial part of Theorem \ref{globreg} in Section 6. It is important to mention that in the general case (as considered in Theorem \ref{globreg}), the function to which we apply this formula is a solution to a low-dimensional elliptic problem. It turns out that it is not valid for an arbitrary function of the class $H^1_{\mu}.$ 

In the next step, we construct a sequence of approximations $\alpha_n$ of the extension $\widetilde{u}$ sharing two special properties. Firstly, we can extend each term $\alpha_n$ from the considered structure to the whole Euclidean space and obtain a sufficiently regular function. Secondly, the global behaviour of extensions is controlled in terms of the original function posed on the low-dimensional structure. Now, each term of the sequence $\alpha_n$ can be shifted giving the sequence $\alpha_n(\cdot + hv),\; v\in S.$ To prove that the shifted sequence converges in a suitable sense and the limit is regular enough, we examine the regularity of the trace $\tr^{\Sigma}u.$ Establishing that $\tr^{\Sigma}u \in H^2(\Sigma)$ and combining this result with the local regularity in regions separated from the intersection set $\Sigma$ we are able to conclude that the function space is closed under shifts and the generalised difference quotients. This allows us to prove a higher regularity of solutions by controlling uniform bounds of generalised difference quotients.

Latter results are implications or refinements of the main theorem. We establish the global continuity of low-dimensional weak solutions by utilising some facts from the Sobolev capacity theory and conclusions provided by our main result. An interesting fact is that the continuity on the whole structure depends on the dimensions of the components.

Furthermore, still relying on the main theorem, we provide a membership of a weak solution in the domain of the measure-related second-order operator $L_{\mu}$. To show that a solution $u$ is in $D(L_{\mu})$ we need to prove that there exists the Cosserat vector field corresponding to $u.$ That is for $u$ exists some normal to $S$ vector field $b$ such that the pair $(u,b)$ belongs to the domain of the low-dimensional second-order derivative. The main tool used to establish this is the closedness result for the operator $L_{\mu}$ stated in \cite{Cho24}.

\subsubsection{Further discussion of the results}{$ $}
\vspace{0.5cm}

Let us briefly discuss some other potentially related results and highlight differences. Equipped with the naturally induced metric, the low-dimensional structure $S$ is a metric measure space. In the category of metric measure spaces, there are several known candidates for generalisations of the Euclidean gradient, amongst them two are probably best known: the upper gradient discussed in \cite{Amb05, Hei15} and the Cheeger gradient, see \cite{Che99}. It turns out that interpreting the low-dimensional setting as metric measure spaces, the upper gradient corresponds to $|\nabla_{S_i}|,$ where $\nabla_{S_i}$ is a component of the gradient tangent to $S_i$. In other words, such an approach loses a significant amount of information required by our goals. Therefore, we start with the $\mu$-related gradient, which is the most natural generalisation of a gradient tangent to a manifold. It should be mentioned that our regularity results are strongly based on a form of the gradient and the fact that it represents the classical gradient locally. Potential generalisations of the obtained regularity results to a less geometrical, more abstract setting seems to be difficult.

It should be noted that the considered setting of measures equipped with a tangent structure is a special case of the very general notion of a varifold, or a variational measure as named in book \cite{Fed96}. The difference is that in the case of the theory examined in this project, tangent structures to the considered objects are defined intrinsically by the inherited embedding into the Euclidean space, and 
varifolds may possess very general tangent structures. 

We would like to draw attention to a recent paper \cite{Cap22} that addresses the Neumann problem for the fractional Laplacian. The authors of this paper consider this problem in the context of a general setting of doubling metric measure spaces $X$ which are equipped with the Poincaré inequality. Through the development of Caccioppoli-type estimates in this context, the authors establish the H{\"o}lder continuity of solutions to this problem. Moreover, the authors investigate the existence of solutions to a similar fractional Laplacian problem on those spaces that are equivalent to the "boundary" $\overline{X}\setminus X.$ Significantly, they are able to relax the assumptions and expand the class of considered spaces by eliminating the Poincaré inequality assumption. Our low-dimensional setting, in contrast to the general theory of \cite{Cap22}, offers a distinct advantage in that it provides a well-defined notion of the second-order derivative, which enables us to consider second-order problems.

There is a rich theory of partial differential equations on graphs. For an introduction to the theory, we recommend referring to \cite{Lag04}, and for an overview of some results, we suggest \cite{Meh01}. There are two ways to describe analytic problems on graphs. In the continuous approach, we consider functions defined on the edges of the graph with some transmission conditions posed on vertices, for instance see \cite{Kra20}. In the discrete approach, we consider classes of functions defined on the vertices of the graph, and the integration or differential operators are in the discrete form, see \cite{Gri16}. From our perspective, the former type of setting is more important as it is closer to the framework we consider. Therefore, we limit our discussion to it. The extensive research in this theory includes studies of elliptic or parabolic problems \cite{Kra20}, or variational problems \cite{Lin21}. Moreover, this type of formalism has a wide range of practical applications, for instance, in biological sciences \cite{Kra20} or image processing \cite{Haf16}. Comparing the theory of partial differential equations on a graph with the setting we examine, we can point out two differences. First, graphs are naturally one-dimensional structures, and in our considerations, we study a more general class of objects that also have their own geometry. Second, and more importantly, in the analysis on graphs the functions are usually initially defined on the edges $e_i,\; i=1,...,m$ of graphs, that is spaces like $\Pi_{i=1}^m H^1([0,1])$ are considered, and the interplay between edges is introduced as set of conditions posed in nodes of the graph.  In our approach, we define functional spaces on the given structure $S$, which is a subset of $\R^3$, as a completion of smooth functions $u:\R^3\to\R$ in a corresponding norm. This is crucial as the problems we consider are related to variational problems obtained as relaxations of the classical issues of the form $F:C^{\infty}(\R^3)\to\R$. Besides that, in the low-dimensional setting, counterparts of transmission conditions used in graph theories are naturally included in domains of finiteness of low-dimensional differential operators.

Let us note that, in general, the solutions to low-dimensional problems are not simple gluings of classical solutions on component manifolds.
To present that the low-dimensional solutions might differ in an essential way from solutions of classical problems, we recall the conclusion given by Example \ref{ex2}. For a more detailed discussion, see Example \ref{ex2} and \ref{ex3} in Subsection 5.2 showing that quite unexpected phenomena appear even in the simple case of the stationary heat equation.

Assume that $\Omega \subset \mathbb{R}^2$ is a 2-dimensional unit ball $B(0,1),$ and set
\[
E_1:=\{(y,0): y\in [-1,1]\}, \quad E_2:=\{(0,z):z \in [-1,1]\}
\]
and $\mu := \mathcal{H}^1|_{E_1}+\mathcal{H}^1|_{E_2}.$ Let us choose
	$f:=\begin{cases}
	y \text{ on } E_1,\\
	0 \text{ on } E_2
	\end{cases}$  
	and consider the stationary heat problem (see \cite{Ryb20} for the existence and uniqueness result)
	\begin{equation}\label{slabe}
	\int_{\Omega} \nabla_{\mu}u\cdot \nabla_{\mu} \phi d\mu = \int_{\Omega} f \phi d\mu
	\end{equation}
	for $\phi \in C^{\infty}_c(\R^2).$
 
	On the component $E_1$ let us take $u_1(y):= -\frac{21}{1080}-\frac{y^4}{12}+\frac{y^3}{6}+\frac{y^2}{6}-\frac{y}{2}$ and on the component $E_2$ the constant function $u_2(z):=-\frac{21}{1080}.$ It can be observed (see Example \ref{ex2}) that neither $u_1$ nor $u_2$ satisfies the weak equation separately on components, because $\int_{E_i} u_i dx \neq 0.$ On the other hand, the function 
	$u := \begin{cases}
	u_1 \text{ on } E_1,\\
	u_2 \text{ on } E_2
	\end{cases}$
	belongs to $H^1_{\mu}$ with $\int_{\Omega}ud\mu =0,$ and is a solution to weak problem \eqref{slabe}.$ $

We would like to discuss the potential applications of our results and the future prospects of development in the related theory. By establishing the right parabolic setting, we can describe time-dependent phenomena whose static counterparts were examined previously. This includes a long list of various types of problems mentioned earlier. We also note that the strong-form second-order framework was not widely studied before in the analysis of singular measures, and the obtained results provide a suitable setting for expressing issues studied earlier in a more regular way. Rephrasing the weak problems in this way is important because it allows us to capture more of the specific geometrical properties of the singular subsets.

The higher regularity results given for the elliptic type equations are not only important in applications to weak problems as considered in \cite{Ryb20}, but also have a wide range of other possible applications. Specifically, due to the Euler-Lagrange correspondence of weak elliptic problems with the first-order variational theory established in \cite{Ryb20} our results imply higher regularity of minimizers. Given the richness of the variational theory on low-dimensional structures, this has far-reaching consequences in many areas.

There are several ways in which the obtained results can be generalized. First and foremost, it is worth noting that our proofs of main theorems included certain additional restrictions on the class of considered structures. Specifically, we assumed that all structures are embedded in $\R^3$ and we excluded the possibility of the common intersection of three components of the structure. This was done because, in the main technical result of our paper \cite{Cho24}, which is the proof of the closedness of the operator $L_{\mu},$ we proposed a suitable explicit geometrical construction of extensions of functions defined on low-dimensional structures that works in this setting. However, it does not appear to be immediately generalizable to higher dimensions. We utilized this result in our paper devoted to regularity theory, where we conclude that weak solutions belong to the domain of higher-order operators. Although we have not encountered any counterexamples, we believe it may be possible to replicate the proposed construction with a more abstract one that will work in higher dimensions. 

Regarding the elimination of the common intersections of multiple components, a similar situation arises in the paper dealing with parabolic problems. We believe that such generalization in the regularity-related paper can be established by the following observation: the smooth approximation formula for functions supported in low-dimensional domains, as described in \cite{Cho23}, can be interpreted as a variant of the inclusion-exclusion principle. By analogy, if we extend this formula to the case of $k$-many structures, the further complications that arise appear to be only of a technical nature and likely can be eliminated.

Another aspect connected to the low-dimensional theory that seems very interesting to us is the problem of the existence of solutions for equations in the class of functions taking values into the fixed low-dimensional structure. The focus is given to the following issue. 
Let $S\subset \R^3$ be a low-dimensional structure, and consider the problem of minimization of the energy $E(u)=\int|\nabla u|^2$ in some class of functions $u:\Omega \to S,$ where $\Omega$ is some open subset of the Euclidean space. We ask if it is possible to perturb the given functional $E$ in a way that will give us the corresponding differential equation which solutions are minimizers. The variational problems and corresponding harmonic mappings were studied in this setting by Gromov and Schoen in \cite{Gro92}. They needed such a framework for specific problems. Unfortunately, their work is based on the internal approach, and they do not provide suitable tools to define and analyse partial differential equations in this setting. Precisely, the authors consider the perturbations $E(u+t\psi),$ where $\psi:\Omega \to S.$ In this way the existence of the unique minimizer of $E$ can be shown, but it is not enough to obtain the corresponding PDE. Indeed, we need to consider the perturbations $E(\Pi(u+t\phi)),$ where $\Pi$ denotes the nearest point projection on $S,$ in a class of suitably regular $\phi:\Omega \to \R^3.$ The key problem is in the fact that the low-dimensional structure $S$ is not convex, which implies that the nearest point projection $\Pi$ produces irregularities of a jump-type on the gradient term $\nabla\Pi(u+t\phi).$ Nevertheless, it might be worth considering the approach to this problem based on reformulating this problem into the BV theory. Such a move would allow us to derive an elliptic equation from the properly modified energy functional and to transfer the results of \cite{Gro92} in the "exterior" framework in the sense that the related second fundamental form would be obtained. Developing a second fundamental form on low-dimensional structures is crucial for the later study of harmonic map flow and other partial differential equations in this setting. Further we plan to use the obtained results to derive the proper meaning of weak counterpart of the harmonic map flow and studying regularity of solutions. 

The thesis is organised as follows. 

In Section 2 we evoke basic notions of the new framework. Among others, we present here: definitions of spaces of functions and the adequate low-dimensional framework, preliminary properties of introduced operators, or notions of solutions to parabolic problems.

Section 3 contains elementary and basic results describing the properties of newly introduced objects. Subsection 3.3 includes new and significant characterisations of first- and second-order spaces of functions.

Section 4 is devoted to the main theorems of the parabolic setting -- the proof of the existence of a semigroup generated by the operator $L_{\mu}$ and the result showing the closedness of the operator $L_{\mu}.$       

Section 5 is devoted to weak variants of the parabolic equations and the regularity of solutions. We prove that when passing $t \to \infty,$ parabolic solutions converge to a solution of the stationary equation. This section also contains examples of low-dimensional issues, in particular Example \ref{ex3} shows it is not always possible to obtain low-dimensional weak solutions with the expected regularity by adding up weak solutions component-wise. 

Section 6 deals with the higher regularity of solutions to the elliptic problems. A various types of regularity-related statements are proven there. Finally, we establish connections between the strong-form second-order operator and weak solutions of elliptic equations.

\newpage

\section{Preliminaries}
In this section, we present concepts and theories that form the foundation of our research. We formulate the low-dimensional framework and provide an overview of the associated analytical setting. Following this, we delve into differential problems, which encompass various forms of low-dimensional counterparts of elliptic and parabolic problems.  

In order to conduct an analysis in the considered setting, it is essential to establish a notion of a tangent bundle to a closed subset of the Euclidean space. Various methods can be used to define such objects, including a variant presented in \cite{Bou97}, utilizing the general theory of variational manifolds as in \cite{Fed96}, or generalizing the definition of the notion of tangent space in the case of a smooth manifold as presented in \cite{Haf03}. The different methods produce slightly different spaces of functions, and a detailed discussion of this subject can be found in paper \cite{Fra99}. 

Each of the concurrent methods of defining such objects is better suited to certain problems. For this purpose, we choose the method based on the analogy with the classical way of constructing a tangent structure to a smooth manifold, as presented in \cite{Fra99}. Our decision was influenced by two factors. Firstly, as presented in \cite{Ryb20}, this definition works well with partial differential equations. Secondly, it suits second-order problems considered in this scenario. Indeed, by utilizing this definition, we immediately obtain the existence of a smooth approximating sequence of functions posed on the whole space. This is crucial in constructing suitable extensions of functions defined on a low-dimensional structure.   
 
The definitions pertaining to the first-order functional setting have been sourced from \cite{Bou97}, \cite{Haf03}, and \cite{Fra99}. The majority of the definitions in the second-order framework have been extracted from \cite{Bou02}. The operator related to the second-order equation is the concept defined in \cite{Cho24}. The theory of weak elliptic equations has been presented in a format outlined in \cite{Ryb20}. The spaces of time-dependent functions and appropriate framework used later in the weak parabolic equations are taken from \cite{Sho97}.  
The parts related to the higher regularity with respect to the abstract operator and semigroup theory are taken from \cite{Mag89}.

Throughout the rest of the section, if not specified more precisely, the letter $\mu$ denotes a positive real-valued Radon measure on $\R^3.$

\subsection{Low-dimensional structures}{$ $}
\vspace{0.5cm}

We introduce a primary object of interest -- the class of low-dimensional structures. This class consists of specific Radon measures that represent geometrical structures.

The presented class of measures, introduced in \cite{Cho24}, is a modification of that defined in papers \cite{Bou97}, \cite{Haf03}.

Let $\Omega \subset \R^3$ be a non-empty, open and connected set. For a fixed $m \in \mathbb{N}$ and $0\leqslant k \leqslant m+1,$ let $S_i,\; 1 \leqslant i \leqslant m$ be a:
\begin{itemize}
    \item $2$--dimensional closed manifold with boundary for $i\leqslant k$;
    \item $1$--dimensional closed manifold with boundary for $i > k$.
\end{itemize}
Assume further that for each $1 \leqslant i,j \leqslant m, \; i \neq j$:
\begin{itemize}
\item[a)] $\partial \Omega \cap S_i = \partial S_i$;
\item[b)] $S_i$ is transversal to $S_j$ and $\partial S_i \cap \partial S_j =\emptyset$;
\item[c)] for any $1\leqslant i,j,k \leqslant m, \; S_i\cap S_j \cap S_k = \emptyset.$
\end{itemize}   
With each $S_i$ we associate the pair $(\mathcal{H}^{\dim S_i}\lfloor_{S_i}, \theta_i),$ where $\mathcal{H}^{\dim S_i}\lfloor_{S_i}$ is the $\dim S_i$-dimensional Hausdorff measure restricted to $S_i$ and $\theta_i \in L^{\infty}(S_i)$, where $\theta_i \geqslant c > 0,$ for some constant $c.$
We say that a positive Radon measure $\mu$ belongs to the class $\widehat{\mathcal{S}}$ if it is of the form $$\mu = \sum_{i=1}^m \theta_i \mathcal{H}^{\dim S_i}\lfloor_{S_i}.$$

A positive Radon measure $\mu$ belongs to the class $\widetilde{\mathcal{S}}$ if it is of the form $$\mu = \sum_{i=1}^m \mathcal{H}^{\dim S_i}\lfloor_{S_i}.$$

Let us briefly explain the restrictions imposed on the $\widehat{\mathcal{S}}$ class. First of all, we
should understand that each structure $S$ is dependent on the choice of the region $\Omega$, maybe
not in a topological, but at least in a geometrical sense. This is forced by the assumption
a). The transversality condition b) excludes those kinds of domains that do not really need
the presented framework and similar results can be proved with the help of standard methods.
Finally, condition c) disallows for more than two component manifolds to intersect at a single
point. This assumption is introduced to simplify technicalities, and we believe it can be relaxed
for a price of longer computations. 

Throughout the thesis, we name measures belonging to the
class $\widehat{\mathcal{S}}$ as low-dimensional structures. We also use the same name to refer to sets
on which the singular measures are supported. This ambiguity does not produce confusion; the
type of object we refer to should be clear from a context. 

By $\mathcal{S}$ we denote the subclass of the class $\widetilde{\mathcal{S}}$ consisting of low-dimensional structures whose component manifolds have fixed dimension, that is, if $\mu \in \mathcal{S}$ and $\supp\mu = \bigcup_{i=1}^mS_i,$ then for all $i\in \{1,...,m\}$ $\dim S_i = k,$ for $k=1$ or $k=2.$
The subclass $\mathcal{S}$ will be especially important for us when considering problems related to the strong-form operators (see, Section 2.3).
In some cases, it is more appropriate to refer directly to the set $S$ in which the low-dimensional structure $\mu$ is supported and write $S \in \mathcal{S}.$ Indeed, we use this from time to time if such disambiguation does not cause problems. 

The introduced above class $\widehat{\mathcal{S}}$ of lower dimensional structures considered here is a proper subset of the one used, for example, in \cite{Ryb20}. The difference is in two aspects. First, we assume that a boundary of each component manifold $S_i$ is exactly the set of intersection of $S_i$ with the boundary of the ambient set $\Omega$ (condition a)). In \cite{Ryb20}, the authors allow a situation, where $\partial S_i \subset S_j,\; i\neq j.$ A second difference is in an additional restriction on possible types of intersections of components. This is expressed in point c). This restriction rules out common intersections of three or more component manifolds. The above-mentioned conditions are not crucial from our perspective and were introduced to simplify further reasoning and shorten the exposition.

Let $\mu \in \widehat{\mathcal{S}}$ and $S = \supp \mu.$ Let $S = \bigcup_{i=1}^m S_i,$ where $S_i$ are component manifolds. We denote 
$$\partial S := \bigcup_{i=1}^m \partial S_i.$$
For any $i \in \{1,...,m\}$ let $n\lfloor_{\partial S_i}:\partial S_i \to \R^3$ denote the uniquely determined outward normal unit vector field to $\partial S_i$ in the sense of Spivak \cite{Spi65}. 
By the outward normal unit vector field to the low-dimensional structure $\mu$ we call the function $n: \partial S \to \R^3$ defined as
$$n:= \bigcup_{i=1}^m n\lfloor_{\partial S_i}.$$  

We list some examples of low-dimensional structures here. For instance, this class includes the following structures.
\begin{itemize}
    \item Any smooth manifold of a dimension $1$ or $2$ with a boundary embedded in $\R^3$ can be considered as a low-dimensional structure with exactly one component and the corresponding Hausdorff measure.
    \item CW-complex with smooth cells is an example of a low-dimensional structure. In fact, this generalizes the first point as any smooth manifold has a CW-complex structure.
    \item Certain finite graphs belong to the class $\mathcal{S}.$ 
    \item An isometric embedding into the Euclidean space of a singular metric space $X$, as considered by Gromov and Schoen in \cite{Gro92}.  
    \item Smooth immersed submanifolds that have transversal self-intersections.  
\end{itemize}

Let us discuss the following two examples.
\begin{przyk}\cite{Cho24}
Let the set $A \subset \R^3$ be defined as $A=A_1 \cup A_2,$ with component manifolds $A_1 = \{(0,0,z)\in \R^3: z \in [-1,1]\}$ and $A_2=\{(x,y,0)\in \R^3: x^2+y^2\leqslant 1\}.$ Let $\mu = \mathcal{H}^1\lvert_{A_1}+\mathcal{H}^2\lvert_{A_2}.$ Clearly, the measure $\mu$ is a low-dimensional structure with the obvious partition to component manifolds.

The second example presents a structure that is not contained in our definition.
\begin{przyk}\cite{Cho24}
Let us consider the disc $B_1=\{(x,y,0) \in \R^3: x^2+y^2 \leqslant 1\}$ and the manifold $B_2=\{(x,y,z)\in \R^3: x,y\in [-1,1], z=x^2 \}.$ We define a subset $B = B_1 \cup B_2$ with the corresponding singular measure $\nu = \mathcal{H}^2\lvert_{B_1}+\mathcal{H}^2\lvert_{B_2}.$\\
The tangent space to the component $B_2$ in the point $(0,0,0)$ is a plane $\R^2.$ Obviously at any point of $B_1$ the tangent is also $\R^2.$ This means that the condition $b)$ of the definition of the low-dimensional structure is violated.
\end{przyk}
\end{przyk}

\subsection{First-order space of functions}{$ $}
\vspace{0.5cm}

To define a derivative of a function given on some (possibly irregular) low-dimensional structure (or in general on a support of a Radon measure), we need to develop a notion of a tangent bundle of a measure. Before formalising this object, let us recall the well-known measure-related Lebesgue space.

	 For $p \in [1,\infty]$ the Lebesgue space $L^p_{\mu}$ is defined as a subspace of the space of $\mu$-measurable functions such that 
	 \begin{itemize}
	 \item[a)] for $p\in [1,\infty)$ the norm $$\norm{f}_{L^p_{\mu}}:=\left(\int_{\Omega} |f|^p d\mu\right)^{\frac{1}{p}},$$
	 \item[b)] for $p=\infty$ the norm
	 $$\norm{f}_{L^{\infty}_{\mu}}:=\esssup_{\mu} |f|,$$ 
	\end{itemize}
	is finite.

A tangent space to a Radon measure $\mu$ can be defined in many non-equivalent ways. An extensive discussion of this subject with a comparison of various definitions can be found in \cite{Fra99}. We follow the method of construction proposed in \cite{Bou97} and in \cite{Haf03}. The main advantage of this approach is its simplicity and direct analogy with a classical tangent space to a smooth manifold.

	Let us introduce a set $\mathcal{t}^{\perp}_{\mu}$ of all smooth functions vanishing on $\supp \mu,$ i.e.,
	$$\mathcal{t}^{\perp}_{\mu}:=\left\{v \in C^{\infty}_c(\R^3): v=0 \text{ on } \supp \mu \right\}$$ 
	and the set $\mathcal{T}^{\perp}_{\mu}$ of gradients of such functions: 
	$$\mathcal{T}^{\perp}_{\mu}:= \left\{ w \in C^{\infty}_c(\R^3;\R^3): w=\nabla v \text{ on } \supp \mu \text{ for some } v \in \mathcal{t}^{\perp}_{\mu} \right\}.$$
	It is easy to observe that a set-valued mapping $T_{\mu}^{\perp}: \R^3 \to P(\R^3),$ 
	$$T_{\mu}^{\perp} (x):= \left\{w(x)\in \R^3: w \in \mathcal{T}^{\perp}_{\mu}  \right\}$$ assigns to each $x \in \R^3$ certain linear subspace of $\R^3.$  
	The space $T_{\mu}(x)$ tangent to the measure $\mu$ at a point $x \in \R^3$ is defined as 
	$$T_{\mu}(x):= \left\{ w \in \R^3: w\perp T_{\mu}^{\perp} (x)  \right\},$$  
	the symbol $\perp$ denotes the orthogonality in the Euclidean scalar product in $\R^3.$

With a notion of the tangent structure to $\mu$, we might project the classical gradient onto it to obtain its measure-related counterpart. Strictly speaking,
for $\mu$ almost every $x \in \R^3$ let $P_{\mu}(x):\R^3 \to T_{\mu}(x)$ denote the orthogonal projection onto the tangent space $T_{\mu}.$ Then the tangent gradient of a function $u \in C^{\infty}_c(\R^3)$ is defined for a.e. $x \in \R^3$ as 
$$\nabla_{\mu} u(x):= P_{\mu}(x) \nabla u(x).$$

We are ready to introduce the basic Sobolev-like space $H^1_{\mu}.$
The Sobolev space $H^1_{\mu}$ is defined as a completion of the space $C^{\infty}_c(\R^3)$ in the Sobolev norm 
$$\norm{\cdot}_{\mu}:=\left( \norm{\cdot}^2_{L^2_{\mu}} + \norm{\nabla_{\mu} \cdot}^2_{L^2_{\mu}}\right)^{\frac{1}{2}}.$$

Combining a useful characterisation of the space $H^1_{\mu}$ established in \cite{Bou02, Lem. 2.2} with the fact that multiplying a measure by a bounded and separated from zero density does not change a tangent structure \cite{Ryb20}, we derive following observations expressing relations between classical tangent spaces and spaces tangent to a measure $\mu \in \widehat{\mathcal{S}}.$

Assume, that $\mu \in \widehat{\mathcal{S}},$ $\supp \mu = \bigcup_{i=1}^m S_i.$ A classical tangent structure on the component manifold $S_i$ is denoted by $T_{S_i},$ and a classical tangent gradient on the component $S_i$ is denoted by $\nabla_{S_i}.$ Then 
	\begin{itemize}
		\item[a)] $T_{\mu}(x) = \sum_{i=1}^mT_{S_i}(x)$ for $\mu$-a.e. $x \in \R^3,$
		\item[b)] if $u \in H^1_{\mu},$ then for $i \in \left\{1,...,m\right\}$ $u\lfloor_{S_i} \in H^1(S_i),$
		\item[c)] if $u \in H^1_{\mu},$ then for $i \in \left\{1,...,m\right\}$ $(\nabla_{\mu}u)\lfloor_{S_i} = \nabla_{S_i}u,$
		\item[d)] for a fixed $i\in\left\{1,...,m\right\},$ let $\phi \in C^{\infty}(\R^3),\; \phi=0$ on $\bigcup_{j\neq i}S_j,$ $u \in L^2_{\mu}$ and $u\lfloor_{S_i} \in H^1(S_i),$ then $\phi u \in H^1_{\mu}.$  
	\end{itemize}

The basic tool in the classical Sobolev spaces are the Poincar{\'e}-type inequalities like, for example, $\int |u-c_u|^2dx \leqslant C\int|\nabla u|^2dx,$ where $c_m = \strokedint u dx$ and $C>0$ is a positive constant. 

If the considered low-dimensional structure has all component manifolds of the same dimension, then the kernel of the operator $\nabla_{\mu}$ is one-dimensional, and as a result, the standard Poinar{\'e} inequality is valid. See \cite{Ryb20} for details.

It is also shown in \cite{Ryb20, Sec. 2.2} that in a case in which there exist $i,j\in \{1,...,m\}$ such that $\dim S_i \neq \dim S_j,$ where $S_i, S_j$ are component manifolds of the low-dimensional structure $\mu \in \widetilde{\mathcal{S}},$ the classical Poincar{\'e} inequality for functions of the class $H^1_{\mu}$ is not valid. In such a situation, we can prove that the kernel of $\nabla_{\mu}$ is two-dimensional, which clearly proves that the classical Poincar{\'e} inequality is false. This last fact is implied by the observation that the function that on each component is a different constant belongs to the space $H^1_{\mu}.$ The fact that the kernel of $\nabla_{\mu}$ is two-dimensional is related to the null Sobolev capacity of the intersection of component manifolds. 

Due to this fact, the authors of the mentioned paper introduce a weaker version of the Poincar{\'e} inequality \cite{Ryb20, Thm. 2.1}.
 
 Consider a partition $I_1,...,I_d$ of the set of indexes $\left\{1,...,m\right\}=I_1\cup...\cup I_d,$ where sets \hbox{$I_k,\; k\in \{1,...,d\}$} are pairwise disjoint and the following conditions are satisfied:
\begin{itemize}
		\item[a)] for any $k\in \{1,...,d\}$ the characteristic function $\chi_{\left\{\bigcup_{i\in I_k}S_i\right\}}$ of a sum of component manifolds with indexes in $I_k$ belongs to the kernel of the tangent gradient operator, that is 
		\begin{equation*}
		\chi_{\left\{\bigcup_{i\in I_k}S_i\right\}} \in \ker \nabla_{\mu}
		\end{equation*}
		\item[b)] the set $I_k,\;$ $k\in \{1,...,d\}$ is the largest set of indices with the property a): if $\alpha \in I_k$ and $\widetilde{I}_k:= I_k \setminus \{\alpha\},$ then 
		\begin{equation*}
		\chi_{\left\{\bigcup_{i\in \widetilde{I}_k}S_i\right\}} \notin \ker \nabla_{\mu}.
		\end{equation*}
	\end{itemize}  
	With any element $I_k$ of the partition, we associate the projection 
	\begin{equation*}
	P_ku:=\chi_{\left\{\bigcup_{i\in I_k}S_i\right\}} \strokedint_{\bigcup_{i\in I_k}S_i} u d \mu,
	\end{equation*}
	where $\strokedint_Xfd\mu := \frac{1}{\mu(X)}\int_Xfd\mu.$
	Then for $\mu \in \widetilde{\mathcal{S}}$ exists a set of positive constants $\left\{C_1,...,C_d\right\}$ such that for all $k\in\{1,...,d\}$ and any $u \in H^1_{\mu}$ the following version of the Poincar{\'e} inequality is satisfied
	\begin{equation}\label{weakpoincare}
	\sum_{j\in I_k} \int_{\Omega}|u-P_ku|^2d \mu_j \leqslant C_k \sum_{j\in I_k}\int_{\Omega}|\nabla_{\mu}u|^2d\mu_j, 
	\end{equation}
	for $\mu_j:=\mu\lfloor_{S_j}.$

Let us remark that the idea of introducing this kind of partition of component manifolds is to divide the low-dimensional structure into maximal groups of components on which the classical Poincar{\'e} inequality is valid.

\subsection{Second-order framework}{$ $}
\vspace{0.5cm}

This part is focused on defining the operator of a second derivative and corresponding spaces of functions. A general idea is close to the one used in the previous section to construct first-order Sobolev spaces, but now a formal realisation is much more complicated and technical. The first problem is that, in general, even in the case of the smooth function $u,$ the $\mu$-tangent gradient does not belong to $H^1_{\mu}$ space. Thus we cannot apply the operator $\nabla_{\mu}$ to it. A second significant thing is related to a choice of a smooth approximating sequence of a function defined on some low-dimensional structure. It turns out that a choice of an approximating sequence affects a value of a second derivative of the limiting function. To solve this problematic issue, we introduce a certain kind of a normal vector field $b,$ consider the differentiability of $\nabla_{\mu}u+b,$ and properly project the matrix of second-order derivatives.

The presented here exposition of the second-order framework follows paper \cite{Bou02}. Such notions were originally introduced to study the calculus of variations on low-dimensional structures in a measure-oriented setting.

For each $u \in C^{\infty}_c(\R^3)$ by $\nabla^{\perp}u$ we denote a $\mu$-a.e. normal component of a gradient, that is\\ ${\nabla u = \nabla_{\mu} u + \nabla^{\perp} u.}$ The symbol $\R^{3\times 3}_{\text{sym}}$ stands for the space of $3 \times 3$ symmetric matrices, and $\nabla^2u$ is a matrix of second partial derivatives of a function $u.$ 

We introduce a set of triples 
$$\mathcal{g}_{\mu}:= \left\{(u,\nabla^{\perp}u,\nabla^2u): u \in C^{\infty}_c(\R^3)\right\} \subset L^2_{\mu} \times L^2_{\mu}\left(\R^3; T_{\mu}^{\perp}\right) \times L^2_{\mu}\left(\R^3;\R^{3\times 3}_{\text{sym}}\right).$$

In an analogy with the method of defining the space $T_{\mu}$ tangent to a measure $\mu,$ we construct the space $M_{\mu},$ which intuitively can be seen as a second-order counterpart of the space $T_{\mu}.$

Let $\overline{\mathcal{g}_{\mu}}$ denote the closure of the set $\mathcal{g}_{\mu}$ in the space $L^2_{\mu} \times L^2_{\mu}(\R^3; T_{\mu}^{\perp}) \times L^2_{\mu}(\R^3;\R^{3\times 3}_{\text{sym}})$ equipped with the natural product norm. 

We define the space 
$$\mathcal{m}_{\mu}:= \left\{z \in L^2_{\mu}(\R^3;\R^{3\times 3}_{\text{sym}}): (0,0,z)\in \overline{\mathcal{g}_{\mu}}\right\}.$$ 

By Proposition 3.3 (ii) in \cite{Bou02}, for $\mu$-almost every $x\in \R^3$ there exists a $\mu$-measurable multifunction $M^{\perp}_{\mu}:\R^3 \to P(\R^3),$ where $P(\R^3)$ is the power set of $\R^3$ such that 
$$\mathcal{m}_{\mu} = \left\{z \in L^2_{\mu}(\R^3;\R^{3\times 3}_{\text{sym}}): z(x) \in M^{\perp}_{\mu}(x) \text{ for } \mu-a.e.\; x\in \R^3 \right\}.$$ 

For $\mu$-a.e. $x \in \R^3,$ the space of $\mu$-tangent matrices $M_{\mu}$ is defined as 
$$M_{\mu} := \left\{z \in L^2_{\mu}(\R^3;\R^{3\times 3}_{\text{sym}}): z(x) \perp  M^{\perp}_{\mu}(x) \text{ for } \mu-a.e.\; x\in \R^3\right\}.$$
Here the symbol $\perp$ stands for orthogonality in the sense of the standard Euclidean scalar product in $\R^{3\times 3}.$

For a detailed discussion of properties of $M^{\perp}_{\mu},$ $M_{\mu}$ and the space $\mathcal{g}_{\mu},$ see Lemma 3.2. and Proposition 3.3. in \cite{Bou02}.

For $\mu$-a.e. $x\in \R^3,$ let $Q_{\mu}(x):\R^{3\times 3}_{\text{sym}} \to M_{\mu}(x)$ be the orthogonal projection. Consider the set $$D(A_{\mu}):=\left\{(u,b)\in L^2_{\mu} \times L^2_{\mu}(\R^3;T_{\mu}^{\perp}): \exists z\in L^2_{\mu}(\R^3;\R^{3\times 3}_{\text{sym}}) \text{ such that } (u,b,z)\in \overline{\mathcal{g}_{\mu}}\right\}.$$\label{dommu}
A normal vector field $b$ such that $(u,b) \in D(A_{\mu})$ is called the Cosserat vector field of $u \in L^2_{\mu}.$

On the domain $D(A_{\mu})$ we introduce the operator $A_{\mu}$ as 
\begin{equation}\label{drop}
A_{\mu}(u,b) := Q_{\mu}(z). 
\end{equation}
This operator can be understood as a $\mu$-related matrix of second-order derivatives of a function $u.$

In the next step, we propose generalising a classical tangent Hessian of a function defined on a smooth manifold.

	For $\mu$-a.e. $x \in \R^3$ let $P_{\mu}^{\perp}(x):\R^3 \to T_{\mu}^{\perp}(x)$ denotes an orthogonal projection onto the normal space $T_{\mu}^{\perp}(x).$  
	Put 
	$$D(\nabla^2_{\mu}):= \{u \in H^1_{\mu}: \exists b : (u,b)\in D(A_{\mu})\}$$ 
	and 
	$$D(C) := \{v \in L^2_{\mu}(\R^3;\R^3): P_{\mu}(x)v(x) \in (H^1_{\mu})^3,\; P_{\mu}^{\perp}(x)v(x) \in (H^1_{\mu})^3 \text{ for } \mu-a.e. \; x\in \R^3\}.$$
 
 Let $C: D(C) \to L^2_{\mu},$ be an operator defined as 
	$$Cv := P_{\mu}^{\perp}\nabla_{\mu}(P_{\mu}v) +  P_{\mu}\nabla_{\mu}(P_{\mu}^{\perp}v),$$ 
	and let $T_C: \mathbb{R}^3 \to \mathbb{R}^{3\times 3}$ be a tensor field 
	$$T_C(x)v(x) := (Cv)(x).$$ 
 
	A right generalisation of a tangent Hessian to the class of lower dimensional structures is an operator $\nabla^2_{\mu}: D(\nabla^2_{\mu}) \to (L^2_{\mu})^{3\times 3}$ given by a formula 
	$$\nabla^2_{\mu}u := P_{\mu}A_{\mu}(u,b)P_{\mu} - T_Cb.$$

The operator $C: D(C)\to L^2_{\mu}$ might be extended to a continuous operator on the whole $L^2_{\mu}$ space (this can be done by the Hahn-Banach Theorem). This fact implies, that the expression $T_Cb$ makes sense for any Cosserat vector field $b$ related to $u \in D(\nabla^2_{\mu}).$

To verify that the operator $\nabla_{\mu}^2$ is properly defined on the space $D(\nabla_{\mu}^2),$ it is necessary to check that its value is independent of a choice of a Cosserat vector field $b.$ This fact was proven in \cite{Bou02} in Proposition 3.15.

Let us recall an observation (Prop. 3.10 in \cite{Bou02}) describing more straightforwardly a structure of the operator $A_{\mu}$ if $\mu$ is an element of the class $\mathcal{S}.$ 
Firstly please notice that in the considered case, the following inclusion is true 
\begin{equation}\label{prop1}
D(A_{\mu}) \subset \left\{(u,b)\in H^1_{\mu} \times L^2_{\mu}(\R^3;T_{\mu}^{\perp}): \nabla_{\mu}u+b \in H^1_{\mu}\right\}.
\end{equation}
The above inclusion is strict for a generic measure being a member of $\mathcal{S}.$
A second worth-to-observe thing is that if $\mu \in \mathcal{S},$ then the operator $A_{\mu}$ takes a rather expected form and can be expressed as
\begin{equation}\label{prop2}
A_{\mu}(u,b)P_{\mu} = \nabla_{\mu}(\nabla_{\mu}u+b).
\end{equation}

We introduce an operator that can be interpreted as an implementation of the operator $\diverg(B \nabla u)$ in the setting of the second-order low-dimensional theory. As the formula is given in terms of the previously proposed operators, it is consistent both with the second-order theory of Bouchitt{\'e} and Fragal{\`a} \cite{Bou02} and with the theory of first-order differential equations established in \cite{Ryb20}.

\begin{defi}(Second-order differential operator $L_{\mu}$)\cite{Cho24}\label{secondop}\\
	Put $D(L_{\mu}) := D(\nabla^2_{\mu})$ and let $B$ be a $3 \times 3$ smooth, elliptic and symmetric matrix, that is, $B=(b_{ij})_{i,j\in \{1,2,3\}},\; b_{ij} \in C^{\infty}(\mathbb{R}^3),\; b_{ij} \geqslant c>0,\; B=B^{T},$ and there exists a constant $C>0$ such that for every $\xi \in \mathbb{R}^3,\; C|\xi|^2\leqslant \sum_{i,j=1}^3b_{ij}\xi_i\xi_j.$ 
 We introduce the tangent divergence operator $\diverg_{\mu}:(H^1_{\mu})^3 \to L^2_{\mu},$ $$\diverg_{\mu} ((v_1,v_2,v_3)) := \tr\nabla_{\mu}(v_1,v_2,v_3).$$ The second-order differential operator $L_{\mu} : D(L_{\mu}) \to L^2_{\mu}$ is defined as $$L_{\mu}u := \sum_{i,j=1}^3 b_{ij}(\nabla^2_{\mu}u)_{ij} + \sum_{i=1}^3 (\nabla_{\mu}u)_i \diverg_{\mu}(b_{i1},b_{i2},b_{i3}).$$ 
\end{defi}

We are especially interested in the case where the matrix of coefficients $B$ is the identity matrix.

\begin{defi}(Operator $\Delta_{\mu}$)\cite{Cho24}\label{lap}\\
Let the operator $L_{\mu}$ be as defined above and assume that the matrix $B$ is the identity matrix, i.e. $B=Id.$ In such case the operator $L_{\mu}$ will be denoted by $\Delta_{\mu}$ and its domain by $D(\Delta_{\mu}).$ The operator $\Delta_{\mu}$ can be expressed as  $$\Delta_{\mu}u = \tr \nabla^2_{\mu}u.$$
\end{defi}

A slightly modified version of the domain $D(L_{\mu}),$ which includes information about Neumann boundary conditions, will also be needed. The definition we give might be without any problems expressed in case of any regular enough flow through the boundary $\partial S$, but for clarity of further proceedings, we pose it in the special case of zero flow through the boundary.

\begin{defi}($D(L_{\mu})$ with Neumann boundary conditions)\cite{Cho24}\label{0domain}\\
	Assume that a matrix $B$ corresponds to $L_{\mu}$ as in Definition \ref{secondop}. By $D(L_{\mu})_N \subset D(L_{\mu})$ we denote a subspace that consists of all $u\in D(L_{\mu})$ such that for all component manifolds $S_i,\; i\in \{1,...,m\}$ we have 
	$$(B\nabla_{S_i} u)\lfloor_{\partial S_i} \cdot n\lfloor_{S_i} = 0,$$ 
	almost everywhere with respect to the measure $\mathcal{H}^{\dim S_i - 1}$ on $\partial S_i.$ Here $n$ is the outward normal unit vector to $\partial S.$
\end{defi}

Directly from the definition of $D(L_{\mu})$ it follows, that if $u\in D(L_{\mu}),$ then for all $i\in \{1,...,m\}$ $u|_{S_i}\in H^2(S_i).$ This implies, that a trace of $\nabla_{S_i} u$ for $u \in D(L_{\mu})$ is well-defined. Moreover, please note, that $D(L_{\mu})_N$ is a Banach subspace of the space $D(L_{\mu})$ in the inherited norm.

Analogously the domain of the operator $A_{\mu}$ can be equipped with Neumann boundary conditions.

\begin{defi}($D(A_{\mu})$ with Neumann boundary conditions)\cite{Cho24}\\
The subspace $D(A_{\mu})_N \subset D(A_{\mu})$ is defined as 
$$D(A_{\mu})_N:=\left\{(u,b)\in D(A_{\mu}): u\in D(L_{\mu})_N\right\}.$$
\end{defi}

From a perspective of the research conducted in this thesis, an important role will be played by semigroups of linear operators whose evolution is governed by the operator $L_{\mu}.$
 
We say that the operator $L_{\mu}$ generates a semigroup $S:\left[0,+\infty\right)\times L^2_{\mu} \to L^2_{\mu}$ if for all $u \in D(L_{\mu})$ we have $$L_{\mu}u = \lim_{t \to 0_+} \frac{S(t)u - u}{t},$$ where the limit is taken in the $L^2_{\mu}$-norm.

Let $\mu \in \mathcal{S}$ be a low-dimensional structure. Denote $E := \supp \mu;$ $E \subset \Omega.$  
Let $E=\bigcup_{i=1}^mE_i,$ where $E_i$ is a component manifold and $n$ be a normal unit vector field on $\partial E$ directed outward. Let $l^1$ denotes the one-dimensional Lebesgue measure and assume that $g \in L^2_{\mu}.$ Moreover, let $B$ be a matrix of coefficients of the operator $L_{\mu}$ as in Definition \ref{secondop}.

A low-dimensional counterpart of a parabolic problem with Neumann boundary conditions is defined as 
\begin{equation}\label{parabolic}
\begin{matrix}
\partial_tu - L_{\mu}u = 0 &\; \mu\times l^1-\text{a.e. in } E \times [0,T],\\ \\

B\nabla_{\mu}u\cdot n = 0 &\; \text{ on } \partial E \times [0,T],\\ \\

u = g &\; \text{ on } E \times \{0\}.
\end{matrix}
\end{equation}

\begin{defi}(Solution to a parabolic problem)\label{parabolic2}\cite{Cho24} \\
	A solution to the low-dimensional parabolic problem is a function $u$ determined by a {\color{blue}strongly continuous semigroup} $S: \left[0,T\right] \times L^2_{\mu} \to L^2_{\mu}$ generated by the operator $L_{\mu}.$ Precisely, $u:\left[0,T\right] \to L^2_{\mu}$ is a solution of parabolic Neumann problem \eqref{parabolic} if 
	\begin{itemize}
	\item[a)] for all $t \in \left[0,T\right],$ 
	$$u(t) = S(t)g,$$
	where $S: \left[0,T\right] \times L^2_{\mu} \to L^2_{\mu}$ is a semigroup generated by the operator $L_{\mu}$ and the function $g$ is a given initial data, 
\item[b)]the function $u$ satisfies the equation $$\partial_tu - L_{\mu}u = 0,$$  $\mu\times l^1-\text{almost everywhere in } E \times [0,T],$
 \item[c)] $u$ satisfies zero Neumann boundary condition in a sense formulated below.
	\end{itemize}
	For every $t \in \left[0,T\right]$ and every $i\in \{1,...,m\}$ we have $B\nabla_{\mu}u\cdot n = 0$ almost everywhere on $\partial E_i$ with respect to $\mathcal{H}^1\lfloor_{\partial E_i},$ where $B\nabla_{\mu}u\cdot n:D(L_{\mu}) \to L^2_{\mathcal{H}^1\lfloor_{\partial E_i}} ,$ 
	$$(B\nabla_{\mu}u\cdot n)(x) := B(x)(\nabla_{\mu}u)\lfloor_{\partial E_i}(x) \cdot n\lfloor_{\partial E_i}(x).$$ 
	$(\nabla_{\mu}u)\lfloor_{\partial E_i}$ is a trace of a Sobolev function on the boundary of $E_i$ and $n(x)$ is the outer normal vector (of length one) at a point $x \in \partial E.$   
\end{defi}
 
\subsection{Weak formulations of elliptic and parabolic problems}{$ $}
\vspace{0.5cm}

Throughout this section, we do not limit our attention to structures consisting of component manifolds of fixed dimension one or two, but we consider general low-dimensional structures $\mu \in \widehat{\mathcal{S}}.$ Let us remind that on such structures, the Poincar{\'e} inequality is satisfied in the generalised sense as stated in Equation \eqref{weakpoincare}.

In order to introduce an appropriate elliptic setting, we begin with a notion of a suitable relaxation of a given matrix of coefficients.
\begin{prop}\label{relax}(Relaxed matrix of coefficients); Proposition 3.1 of \cite{Ryb20}$ $\\
Let $B \in \left(L^{\infty}_{\mu}\right)^{3 \times 3}$ be such that for $\mu$-almost every $x \in \R^3$:
\begin{itemize}
\item[a)] $B(x)$ is a symmetric matrix;
\item[b)] $T_{\mu}(x) \subset  \text{Im}B(x)$;
\item[c)] $(B(x)\xi,\xi)\geqslant \lambda |\xi|^2$ for all $\xi \in \text{Im}B(x)$ and some $\lambda>0$. 
\end{itemize}
The relaxation $B_{\mu}$ of the matrix $B$ is given by the formula 
$$
B_{\mu}(x) := B(x) - \sum_{i=1}^l \frac{B(x)e_i(x)\otimes B(x)e_i(x)}{(B(x)e_i(x),e_i(x))},
$$ 
where $l = 1,2,$ $e_i(x),$ $1 \leqslant i \leqslant l$ are linearly independent, $\mu$-measurable, span $T_{\mu}(x)^{\perp} \cap ImB(x)$ and for $1 \leqslant i,j \leqslant l $, $(B(x)e_i(x),e_j(x)) = \delta_{ij},$ where $\delta_{ij}$ is the Kronecker's delta. 
\qed
\end{prop}

We now refer to the low-dimensional weak equations and the theory first presented in \cite{Ryb20}.

\begin{defi}(Weak formulation of the Neumann problem)\label{defin}\cite{Cho24}$ $\\
Let $f \in L^2_{\mu}$ and $\int_{\Omega} fg d\mu = 0$ for all $g \in \ker \nabla_{\mu}.$ We say that $u \in H^1_{\mu},$ such that $\int_{\Omega} ug d\mu = 0$ for all $g \in \ker \nabla_{\mu},$ is a weak solution to the elliptic Neumann problem 

$$\diverg(B\nabla u)=f,\quad B\nabla_{\mu} u \cdot \nu\lvert_{\partial S}=0$$ 
provided that

\begin{gather}\label{weak}
\int_{\Omega}B_{\mu}\nabla_{\mu} u \cdot \nabla_{\mu}\phi d\mu = \int_{\Omega} f\phi d\mu,
\end{gather}
holds for all $\phi \in C^1_c(\R^3).$  
\begin{uw}\cite{Cho24}\label{H1 jako testowe}
The concept of the elliptic problem as formulated in Definition \ref{defin} comes from \cite{Ryb20}, but the authors do not formulate it as a definition. Due to this fact, we credit paper \cite{Cho24} where the definition is formulated.

By the definition of the space $H^1_{\mu}$, smooth functions $C^{\infty}_c(\R^3)$ are dense in $H^1_{\mu}.$ This means that the class of test functions in the definition above may be extended to $H^1_{\mu}.$\\ 
Moreover, by the fundamental theorem of the calculus of variations it can be checked that condition $\int_{\Omega}u \phi d\mu =0 $ for all $\phi\in \ker\nabla_{\mu}$ implies $\int_{\Omega} u d\mu=0.$

As it is stated in paper \cite{Ryb20} it is possible to give a meaning to the zero Neumann condition using the theory of measures with divergences, see \cite{Che01}.
\end{uw} 
\end{defi}

Our considerations will also be held in the following spaces of functions.
 \begin{defi}\cite{Cho24, Sho97}
Let $\mathring{L^2_{\mu}}:=\{u \in L^2_{\mu}:\int_{\Omega}u d\mu=0\},$ $\mathring{H}^1_{\mu}:= \{u \in H^1_{\mu}:\int_{\Omega}u d\mu=0\}.$ We define 
$$\mathcal{H}:= L^2(0,T;{H}^1_{\mu}),\; \mathring{\mathcal{H}}:= L^2(0,T;\mathring{{H}^1_{\mu}})$$ and 
$$\mathcal{T}:= \{v \in \mathcal{H}: v \in W^{1,2}(0,T;L^2_{\mu}), v(T)=0\},$$ with the zero mean counterpart $$\mathring{\mathcal{T}}:= \{v \in \mathring{\mathcal{H}}: v \in W^{1,2}(0,T;L^2_{\mu}), v(T)=0\}.$$
The spaces $\mathcal{T}, \mathring{\mathcal{T}}$ are equipped with the norm
$$\norm{u}_{\mathcal{T}}:= \left(\norm{u}^2_{L^2H^1_{\mu}}+\norm{u(0)}_{L^2_{\mu}}\right)^{\frac{1}{2}}.$$
Moreover, let $$\mathcal{C}^{\infty}_0:=\{w\in C^{\infty}([0,T];C^{\infty}_c(\R^N)): w(T)=0\},$$ and $$\mathring{\mathcal{C}^{\infty}_0}:=\{w\in C^{\infty}([0,T];C^{\infty}_c(\R^N)): w(T)=0, \int_{\Omega} w d\mu=0\}.$$
\end{defi}

In the next definition, we define operators needed for establishing the weak counterpart of the parabolic problem. Definition \ref{weak3} gives the precise meaning of the considered issue.  
\begin{defi}\cite{Cho24, Sho97}
Let $B=(b_{ij})_{i,j\in \{1,2,3\}},\; b_{ij} \in C^{\infty}(\mathbb{R}^3),\; b_{ij} \geqslant c>0,\; B=B^{T},$ and assume there exists a constant $C>0$ satisfying for every $\xi \in \mathbb{R}^3,\; C|\xi|^2\leqslant \sum_{i,j=1}^3b_{ij}\xi_i\xi_j.$
We introduce a bilinear form $E:\mathring{\mathcal{H}} \times \mathring{\mathcal{T}} \to \mathbb{R}$ by a formula $$E(u,v):= \int_0^T\int_{\Omega}(B\nabla_{\mu}u(t), \nabla_{\mu}v(t)) - u(t)v'(t)d\mu dt,$$ and a functional $F:\mathring{\mathcal{T}} \to \mathbb{R}$ defined as 
$$F(v):=\int_0^T\int_{\Omega}f(t)v(t)d\mu dt + \int_{\Omega}u_0 v(0)d\mu.$$ Here we assume that $f \in L^2(0,T;{L^2_{\mu}}),\; u_0 \in \mathring{L^2_{\mu}}$ and $v':=\frac{d}{dt}v.$ 
\end{defi}	        	

\begin{defi}\cite{Cho24, Sho97}\label{weak3}
Let $f \in L^2(0,T;{L^2_{\mu}})$ and $u_0 \in \mathring{L^2_{\mu}}.$ By a weak parabolic problem with the zero Neumann boundary condition we name an issue of finding a function $u\in \mathring{\mathcal{H}}$ satisfying
\begin{equation}\label{parabolicweak}
E(u,\phi) = F(\phi)
\end{equation}	
for all $\phi \in \mathring{\mathcal{C}^{\infty}_0}.$
\end{defi}

\begin{uw}\cite{Cho24}
Upon initial observation, it might appear that the weak solutions are simply glueings of weak solutions to classical problems that have been set on manifolds $S_i$. However, this is not always the case, even in elliptic issues, as demonstrated in examples of Subsection 5.2. It is not always possible to obtain a low-dimensional weak solution with the expected regularity by simply adding up weak solutions component-wise.
\end{uw}

\subsection{Smooth functions with respect to operators}{$ $}
\vspace{0.5cm}

Throughout this subsection, we introduce notions used to construct the semigroup generated by the low-dimensional second-order operator. For our applications it is enough to consider the operator $\Delta_{\mu},$ but it seems possible to generalise the presented method to cover the case of a more general form of the operator $L_{\mu}.$  

In further reasonings, the space $A_s(\Delta_{\mu})_N$ plays a role of a part of the domain $D(\Delta_{\mu})_N$ in which all needed operations make sense. By this fact the set $A_s(\Delta_{\mu})_N$ will replace the set $D(\Delta_{\mu})_N.$

{\color{blue}The definitions given here are taken from paper \cite{Mag89} and adapted to the low-dimensional framework considered here.}

\begin{defi}\label{veryreg}(Domain of iterations, "very regular" functions){\color{blue}\cite{Mag89}}\\
Denote $D(\Delta^1_{\mu})_N:= D(\Delta_{\mu})_N.$ For $n \in \{2,3,...\}$ a domain of the $n$-th iteration of the operator $\Delta_{\mu}$ is defined as 
$$D(\Delta^n_{\mu})_N:= \bigl\{u \in L^2_{\mu}: \Delta^{n-1}_{\mu}u \in D(\Delta_{\mu}),\; u \in D(\Delta^{n-1}_{\mu})_N\bigr\}.$$ 

Put $D_N:= \bigcap_{n=1}^{+\infty} D(\Delta^n_{\mu})_N.$ 

We introduce a space of "very regular" functions as 
$$A_s(\Delta_{\mu})_N:=\Bigl\{u \in D_N: \exists M>0\; \forall n \in \mathbb{N}\; \forall t \in (0,s)\; \frac{t^n}{n!}\norm{\Delta_{\mu}^n u}_{L^2_{\mu}}<M\Bigr\}.$$ Moreover, let 
$L^2_{\mu, s}:= \overline{A_s(\Delta_{\mu})_N}^{\norm{\cdot}_{L^2_{\mu}}}.$ 
It will be useful to put $\mathcal{E}(\Delta_{\mu})^T_N:= \bigcap_{s\in [0,T]}A_s(\Delta_{\mu})_N.$
\end{defi}  

We will need to introduce a notion of smoothness with respect to an operator that generates a semigroup. In these terms, the elements of the set $A_s$ defined above can be interpreted as functions analytic with respect to a given differential operator.  

\begin{defi}(Smooth functions with respect to a generator){\color{blue}\cite{Mag89}}\\
Let $V$ be a semigroup on the space $L^2_{\mu},$ let $G$ be a generator of $V.$ We define the space $$C^{\infty}(V):=\{u \in \bigcap_{n=1}^{\infty}D(G^n): \norm{G^n u}_{L^2_{\mu}}<+\infty\; \forall n\in \{0,1,2,3,...\}\}.$$ 
The family of seminorms $\{\norm{G^n \cdot}_{L^2_{\mu}}\}_{n\in \{0,1,2,...\}}$ will be denoted for short as $\norm{\cdot}_{C^{\infty}(V)}.$
In the case of the operator $\Delta_{\mu}$ we introduce the space
$$C^{\infty}(\Delta_{\mu})_N:=\{u \in {D}_N: \forall n \in \{0,1,2,...\},\; \norm{\Delta_{\mu}^n u}_{L^2_{\mu}}<+\infty\}.$$ Similarly, we denote the related family of seminorms by $\norm{\cdot}_{C^{\infty}(\Delta_{\mu})_N}.$
\end{defi}

{\color{blue} Theorem 2 of Magyar paper \cite{Mag89} proposes a variant of the Hille-Yosida theorem that will be crucial in our considerations of Subsection 4.2. The original formulation of the mentioned theorem is stated for general locally convex Hausdorff spaces, but we recall it in a special case of the space $L^2_{\mu}$ which is Hilbert.
\begin{tw}(Generation of a semigroup)\cite{Mag89}\\
Assume that
\begin{itemize}
\item For any neighbourhood of zero $W \subset L^2_{\mu}$ exists a neighbourhood of zero $U \subset L^2_{\mu}$ such that if for all $k\in \{1,2,...\},$ for all $u \in D(\Delta_{\mu}^k)_N$ and some constants $\alpha,C>0$ we have $(1-\alpha C)^{-k}(Id-\alpha \Delta_{\mu})^k u\in U,$ then $u \in W,$
\item $\mathcal{E}(\Delta_{\mu})^T_N$ is dense in $L^2_{\mu}.$  
\end{itemize}
Then the operator $\Delta_{\mu}$ is closable in $L^2_{\mu}$ and its closure $\overline{\Delta_{\mu}}$ generates continuous locally equicontinuous semigroup $V$ in $L^2_{\mu}$ such that $e^{-Ct}V(t)$ is equicontinuous.
\end{tw}
}

\newpage

\section{Basic results}

This section contains facts and observations, which will be applied in further reasoning.

Subsection 3.1 deals with new properties of the $H^1_{\mu}$ Sobolev-type space. The results presented here will be applied in our considerations related to the higher elliptic regularity of weak solutions. Subsection 3.2 is devoted to the results regarding the $\mu$-related second derivative operator $A_{\mu}.$ The goal of Subsection 3.3 is to introduce new characterisations of spaces of functions. These results will be helpful in reducing the abstract problem to a set of generic cases of a geometric nature.

Results exposed in Subsection 3.1 are taken from \cite{Cho23}. The content of Subsections 3.2 and 3.3 is established in \cite{Cho24}.

\subsection{Properties of the space $H^1_{\mu}$}{$ $}
\vspace{0.5cm}

If a given function $u$ belongs to the Sobolev-type space $H^1_{\mu}$, it is immediately that $u \in H^1(S_i)$ on each component manifold $S_i.$ We propose two remarkable results showing how the classical Sobolev regularity interplays between various component manifolds.
  
The following observation shows a correspondence between componentwise behaviour in terms of the equality of traces. This is an important tool applied in further studies of the regularity of elliptic solutions.
\begin{prop}\cite{Cho23}\label{slady}
Let $S \subset \R^3$ be a low-dimensional structure of the form $S = S_1 \cup S_2,$ with $\dim S_1 = \dim S_2,$ and denote $\Sigma = S_1 \cap S_2.$ 
Let $u \in H^1_{\mu}.$ Then $\tr^{\Sigma}u_1=\tr^{\Sigma}u_2$, where $\tr^{\Sigma}$ denotes the usual trace on $\Sigma$.

\begin{dow}
Let $\alpha_n \in C^{\infty}(\R^3),\; \alpha_n \rightarrow u$ in $H^1_\mu$ and let $\alpha_n^i:= \alpha_n \lvert_{S_i}, \; i=1,2$. Since $\alpha_n$ is smooth, we know that $\tr^{\Sigma}\alpha^i_n = \alpha^i_n \lvert_{\Sigma}$. Since $\alpha_n^1\lvert_{\Sigma} = \alpha_n^2\lvert_{\Sigma}$, it follows that $\tr^{\Sigma}\alpha^1_n = \tr^{\Sigma}\alpha^2_n$. Since $\alpha^i_n$ converges to $u_i$ in $H^1(S_i)$, the continuity of the trace operator implies 
$$
\tr^{\Sigma}\alpha^i_n \xrightarrow{L^2(\Sigma)}\tr^\Sigma u_i, \; i=1,2.
$$
However, $\tr^{\Sigma}\alpha^1_n = \tr^\Sigma \alpha^2_n$ and thus $\tr^{\Sigma}u_1 = \tr^{\Sigma}u_2$. 
\qed
\end{dow}
\end{prop}

It remains to address the case of $\dim S_1 \neq \dim S_2$. Let us assume that $\dim S_1 < \dim S_2$; as the components $S_1, \, S_2$ of the low-dimensional structure are transversal (condition b) in the definition of a low-dimensional structure), this implies $\dim(S_1 \cap S_2) < \dim S_2 - 1,$ and the $S_1\cap S_2$-trace operator of a function defined on $S_2$ is ill-posed, at least in a classical sense. Even under some extra assumptions (like, for instance, componentwise continuity) we cannot expect any agreement of traces. Indeed, we have the following result.

\begin{prop}\cite{Cho23}\label{zero_extension}
Assume that $\dim S_1 \neq \dim S_2.$ Let $w \in H^1(S_i),\; i=1,2.$ Then the functions 
$$
u:= 
\begin{cases}
w &\text{ on }S_1,\\
0 &\text{ on }S_2
\end{cases}
\quad
\text{and}
\quad 
v:= 
\begin{cases}
0 &\text{ on }S_1,\\
w &\text{ on }S_2
\end{cases}
$$
belong to $H^1_{\mu}.$
\end{prop}
\begin{dow}
Without the loss of generality, we can assume that $\dim S_1 > \dim S_2$. Let us observe that $\capac_2(S_1\cap S_2, S_1)=0$ (the $\capac$ stands for the Sobolev capacity -- for its definition see Section 4.7 in \cite{Eva15}). This means that there exists a sequence $\phi_n \in C^{\infty}(\R^2)$ which proves the mentioned fact.  

For the function $w\in H^1(S_1)$ we find the sequence $\alpha_n \in C^{\infty}(\R^2)$ approximating in the $H^1$-norm, and for each term we introduce the $\R^3$-extension $\widetilde{\alpha_n}\in C^{\infty}(\R^3)$ by the formula\\ $\widetilde{\alpha_n}(x,y,z):= \alpha_n(x,y).$ 

Similarly, we extend the $\phi_n$ sequence to the sequence $\widetilde{\phi_n} \in C^{\infty}(\R^3),\; \widetilde{\phi_n}(x,y,z):=\phi_n(x,y).$

Now we define another sequence $\Psi_n\in C^{\infty}(\R^3),\; \Psi_n:= 1-\widetilde{\phi_n}\widetilde{\alpha_n}.$ It can be easily verified that $\Psi_n \xrightarrow{H^1_{\mu}} u,$ what gives $u \in H^1_{\mu}.$   
This provides that $u \in H^1_{\mu}.$ The fact that $v \in H^1_{\mu}$ is established analogously.  
\qed
\end{dow}

\begin{uw}\cite{Cho23}
From the Proposition \ref{zero_extension} we can derive a short proof of the fact that in the case of a low-dimensional structure with component manifolds of various dimensions, the $\mu$-tangent gradient operator $\nabla_{\mu}$ has more than one-dimensional kernel. Indeed, let $S$ be a low-dimensional structure. For shortening, let us assume that $S=S_1 \cup S_2$ and $\dim S_1 > \dim S_2.$ By Proposition \ref{zero_extension} we obtain that the functions $u_1=\begin{cases}
    1 & \text{ on } S_1,\\
    0 & \text{ on } S_2
\end{cases},$ $u_2=\begin{cases}
    0 & \text{ on } S_1,\\
    1 & \text{ on } S_2
\end{cases}$ are members of the $H^1_{\mu}$ space. Of course we have $u_1,u_2 \in \ker \nabla_{\mu},$ and they span $\ker \nabla_{\mu}.$ 
\end{uw}

\subsection{Some properties of $A_{\mu}$}{$ $}
\vspace{0.5cm}

The construction of $\mu$-essential second derivative operator $A_{\mu}$ is given in Subsection 2.3; see Equation \ref{drop}.  

The operator $A_{\mu},$ despite its non-standard form, shares properties expected from a reasonable differential operator. From our research perspective, the next observation is essential for further reasoning.

\begin{lem}\cite{Bou02, Cho23}\label{comp}\\
Assume that $\mu \in \mathcal{S}.$
	\begin{itemize}
		\item[a)] The operator $A_{\mu}: D(A_{\mu}) \to (L^2_{\mu})^{3 \times 3}$ is closed in the sense of the $L^2_{\mu}$-convergence, that is, if $$(u_n,b_n) \in D(A_{\mu}),\;\; u_n \xrightarrow{L^2_{\mu}} u \in L^2_{\mu},\;\; b_n \xrightarrow{L^2_{\mu}} b \in L^2_{\mu}(\R^3;T^{\perp}_{\mu})$$ and $$A_{\mu}(u_n,b_n) \xrightarrow{L^2_{\mu}} D \in (L^2_{\mu})^{3 \times 3},$$ then $$(u,b) \in D(A_{\mu})\;  \text{  and  }\;  D = A_{\mu}(u,b).$$
		\item[b)] The same is true if the domain $D(A_{\mu})$ is replaced with $D(A_{\mu})_N.$
	\end{itemize}
	\begin{dow}
		\begin{itemize}
			\item[a)] For the proof, see \cite{Bou02, Prop. 3.5 (i)}.
			\item[b)] By the result of point a), the inclusion $D(A_{\mu})_N \subset D(A_{\mu}),$ by applying the classical Poincar{\'e} inequality on each component manifold $E_i,\; (\supp \mu = \bigcup_{i=1}^mE_i),$ and due to continuity of the trace operator it can be checked that the space $D(A_{\mu})_N$ is closed with respect to the convergence described in point a). Thus the proposed result follows. 
            \qed
		\end{itemize} 
	\end{dow}
\end{lem}

Next lemma shows that a global continuity of a function $u \in D(A_{\mu})$ can be derived not only in the case of component manifolds of a fixed dimension but also if the dimensions of components vary.

\begin{lem}\cite{Cho23}\label{conti}
    Let $u \in D(A_{\mu}).$ Assume that $\mu \in \mathcal{S},\; P:= \supp \mu = E_i \cup E_j,$ \hbox{$E_i\cap E_j\neq \emptyset,$} $\dim E_i=1 \text{ and } \dim E_j=2.$
	Then the function $u$ is continuous on the set $P.$   
	\begin{dow}
	 To avoid discussion of technical aspects let us assume that $u\lfloor_{E_j}$ is compactly supported in the intel of the component $E_j.$	 
   By inclusion \eqref{prop1} we know that $u\in H^1_{\mu}$ and $u\lfloor_{E_k} \in H^2(E_k)$ for $k=i,j.$ We justify continuity of $u$ proceeding by a contradiction. 
   
   Suppose our claim is false, i.e. $u$ is discontinuous. As $u\lfloor_{E_k} \in H^2(E_k)$ for $k=i,j$ provides continuity on each component manifold $E_k,\; k=i,j,$ a discontinuity have to appear in the junction set $E_i \cap E_j.$ Let $p \in E_i \cap E_j$ be a fixed point of discontinuity. After multiplying the function $u$ by a constant, we can assume that a jump at the discontinuity point $p$ is greater than $2.$ Let $\{\phi_n\}_{n\in \mathbb{N}},\; \phi_n \in C^{\infty}_c(\R^3)$ be an approximating sequence of $u,$ which justifies the membership in domain $D(A_{\mu}).$ 
   
		A convergence in the sense of $D(A_{\mu})$ implies convergence in the $H^1$-norm on the 1-dimensional component $E_i$ (of course $D(A_{\mu})$-convergence implies $H^2(E_i)$-convergence, but the weaker type is enough at this moment of reasoning). This implies that it is possible to extract a subsequence converging uniformly on $E_i$ (relabelling is omitted). 
  
	Let us consider a sequence $w_n:E_j \to \R,$ $w_n:= u\lfloor_{E_j}-\psi_n.$ The sequence $w_n$ satisfies \hbox{$w_n \in C_c(\intel E_j),$} $w_n \in H^2(E_j)$ and $w_n \to 0$ in $H^2(E_j).$ Moreover, $|w_n(p)|>1.$ This means that the sequence $w_n$ can be used as a witness showing that the $H^2$-capacity of the point $p$ in the disc $E_j$ is zero. This is a contradiction with the general theory of Sobolev capacity.	
 A general situation where $u\lfloor_{E_j}$ is not necessarily compactly supported in the intel of $E_j$ can be reduced to the discussed above case by modifying the function $u\lfloor_{E_j}$ properly.
		\qed
	\end{dow}
\end{lem}

\subsection{Local characterisation of the functional spaces}$ $
\vspace{0.5cm}

We introduce a local characterisation of low-dimensional Sobolev-type spaces in terms of the behaviour of functions in a neighbourhood of junction sets. This will allow us to divide our study into a certain number of "generic" local cases. 

In this paragraph we assume that $\mu \in \mathcal{S}.$

\begin{defi}(Partition of unity)\cite{Cho23}\\
	Let $\mu$ be an arbitrary measure of the class $\mathcal{S}.$ Denote $S:=\supp \mu \subset \mathbb{R}^3,$ where $S=\bigcup_{i=1}^m S_i$ and $S_i$ are component manifolds. 
	We introduce the set
	$$\widetilde{J}:= \left\{p\in S: \exists i,j\in \{1,...,m\},\; i\neq j,\; p\in S_i\cap S_j\right\}.$$ 
	On the set $\widetilde{J}$ we define a relation $\sim_r$ as
	$$p\sim_r q \iff \exists i,j \in \{1,...,m\},\; i\neq j,\; \exists w\in C^{\infty}\left([0,1];S_i \cap S_j\right):\; w(0)=p,\; w(1)=q.$$
	It is easy to check that $\sim_r$ is an equivalence relation and the number of all equivalence classes is finite. In each equivalence class $[p]_r$ we arbitrarily choose one representative $p.$ Let $J$ stands for the set of all such representatives. Every element of the set $J$ represents a different junction set of the structure $S.$ 
	
        A standard result of the mathematical analysis determines the existence of a finite set $K$ and a partition of unity that consists of functions $\alpha_p \in C^{\infty}_c(\R^3),\; p\in I:=J\cup K$ and two sequences of sets (open in the topology of $\R^3$): 
	$\left\{O_p\right\}_{p \in I},\; \left\{U_p\right\}_{p \in I} \subset \R^3$ satisfying stated below properties.
 
	For each $p\in J,\; p \in S_i\cap S_j$ exists an open set $O_p \supset S_i\cap S_j$ and exists an open set $U_p$ such that $S_i \cap S_j \subset U_p\subset O_p,$ $U_p \cap O_q = \emptyset,$ for $q\in I,\; q\neq p.$  
 
 The family $\left\{O_p\right\}_{p \in I}$ cover the set $S,$ that is $\bigcup_{p \in I} O_p \supset S.$ For each $p\in I$ the functions $\alpha_p$ satisfy $\alpha_p \subset \overline{O_p}$ and $\sum_{p\in I}\alpha_p  \equiv 1$ on $S.$ We also assume that the low-dimensional structure $\mu \in \mathcal{S}$ satisfies two additional facts. For each $p\in J$ there exists $x \in S_i \cap S_j$ and a ball $B(x,1)$ such that $B(x,1) \subset S$ and $O_p = B(x,1).$ Moreover, the family $\{O_p\}_{p\in I}$ can be chosen in a way that for each $O_p$ we have $\partial O_p \in C^{2}.$
\end{defi}

The following propositions show that being a member of the spaces $H^1_{\mu}$ or $D(A_{\mu})$ can be completely characterised in terms of behaviour near points of junctions.

\begin{prop}("from local to global" characterisation of $H^1_{\mu}$)\cite{Cho23}\label{loc1}\\
	Denote $\mu_p := \mu\lfloor_{O_p}, p\in I.$ The space $H^1_{\mu}$ can be characterised as 
	$$H^1_{\mu} = \left\{u \in L^2_{\mu}: \forall p\in I\; \forall i\in \{1,...,m\}\; \exists u_p \in H^1_{\mu_p},\; u = \sum_{p \in I}\alpha_pu_p\right\}.$$
	\begin{dow}
		$(\implies)$ Assume that $u \in H^1_{\mu},$ put $u_p:= u\lfloor_{\supp \alpha_p}$ for $p\in I.$
		By standard properties of a partition of unity, it follows that $u = \sum_{p \in I}\alpha_pu_p.$
		
		$(\Longleftarrow)$ Let $u \in \left\{w \in L^2_{\mu}: \forall p\in I\;  \exists w_p \in H^1_{\mu_p}\; w = \sum_{p \in I}\alpha_pw_p\right\}.$ 
		For any $u_p \in H^1_{\mu_p},$ the functions $\alpha_p u_p$ can be extended by zero to obtain $\alpha_p u_p\in H^1_{\mu}.$ Let $\phi^n_p \in C^{\infty}(\R^3)$ be a sequence witnessing belonging of $u_p$ to $H^1_{\mu_p}.$
		Taking the sequence 
		$$\sum_{p \in I}\alpha_p \phi^n_p \in C^{\infty}_c(\R^3)$$ 
		and passing to the limit $n \to \infty$ we conclude that $u \in H^1_{\mu}.$
		\qed
	\end{dow}
\end{prop}
 
An analogous fact can be established in the second-order framework.

\begin{prop}("from local to global" characterisation of $D(A_{\mu})$)\cite{Cho23}\label{propp}\\
	 Denote $\mu_p := \mu\lfloor_{O_p}, p \in I.$  Then
	\begin{equation}\label{loc2}
	D(A_{\mu}) = \left\{(u,b)\in L^2_{\mu}\times L^2_{\mu}(\R^3;T^{\perp}_{\mu}): \forall p\in I\; \exists (u_p,b_p)\in D(A_{\mu_p}),\; u = \sum_{p \in I}\alpha_p u_p\right\}.
	\end{equation}
	\begin{dow}
		$(\implies)$ Let $(u,b) \in D(A_{\mu}).$ For each $p \in I,$ take $u_p:= \alpha_p u.$ By a definition of the domain $D(A_{\mu})$ there exists a sequence $\psi_n \in C^{\infty}(\R^3),$ such that $(\psi_n,\nabla^{\perp}\psi_n) \to (u,b)$ in the sense of the $L^2_{\mu}\times L^2_{\mu}(\R^3;T^{\perp}_{\mu})$-convergence and $A_{\mu}(\psi_n,\nabla^{\perp}\psi_n) \to A_{\mu}(u,b)$ as $n \to \infty.$ 
  
  For any $p \in I$ consider sequence $\alpha_p\psi_n.$ 
  
  It is easy to see, that there exist $b_p \in L^2_{\mu_p}(\R^3;T^{\perp}_{\mu_p})$ such that
		$$A_{\mu_p}(\alpha_p \psi_n,\nabla^{\perp}(\alpha_p \psi_n)) \to A_{\mu_p}(\alpha_p u, b_p)$$
		in the norm of $\left(L^2_{\mu}\right)^{3 \times 3},$ when we pass with $n \to \infty.$
		This shows that 
		$$(\alpha_p u, b_p) \in D(A_{\mu_p}) \text{ for all } p\in I$$ 
		and proves the considered implication.
		
		$(\Longleftarrow)$ Assume that $(u,b)$ belongs to the right hand side of equality \eqref{loc2}.
		Let $\psi^n_p \in C^{\infty}_c(\R^3)$ be a sequence justifying membership $(u_p,b_p) \in D(A_{\mu_p}).$ Let take $\alpha_p$ and consider $\alpha_p \psi^n_p$ for $p \in I.$
		The functions $\alpha_p u_p,$ can be extended by zero to the whole set $\supp \mu,$ thus $(\alpha_p u_p,\nabla^{\perp}(\alpha_p u_p))\in D(A_{\mu}).$
		Each term of the sum 
		$$\sum_{p\in I}\alpha_p \psi^n_p \in C^{\infty}(\R^3)$$  
		converges in the sense of $D(A_{\mu}),$ thus $(u,b)\in D(A_{\mu}).$
		\qed
	\end{dow}
\end{prop}
 
The case of the second derivative operator equipped with the Neumann boundary condition can be immediately covered too.

\begin{prop}("from local to global" characterisation of $D(A_{\mu})_N$)\cite{Cho23}\label{proppp}\\
	Put $\mu_p := \mu\lfloor_{O_p}, p \in I.$ Then
	
	\begin{equation}\label{loc3}
	D(A_{\mu})_N = \left\{(u,b)\in L^2_{\mu}\times L^2_{\mu}(\R^3;T^{\perp}_{\mu}): \forall p\in I\;  \exists (u_p,b_p)\in D(A_{\mu_p})_N,\; u = \sum_{p \in I}\alpha_p u_p \right\}.
	\end{equation}
	
	\begin{dow}
		The proof follows exactly the same lines as the proof of Proposition \ref{propp}. The needed modification is to change for all $p \in I$ the domains $D(A_{\mu_p})$ to its counterparts equipped with the zero Neumann boundary conditions -- the spaces $D(A_{\mu_p})_N.$
		\qed
	\end{dow}
\end{prop}

\newpage

\blankpage

\newpage

\

\newpage
\newpage
\section{Strong-form parabolic problems}
The content presented in this section consists of the main results established in our paper \cite{Cho24}.

This section aims to establish the existence and uniqueness of solutions to strong-form parabolic equations posed on lower dimensional structures.

The section has the following structure. In 4.1 we prove the main technical result of \cite{Cho24}, showing that the second-order equation operator $L_{\mu}$ is a closed operator in the sense of the $L^2_{\mu}$-convergence. The heart of the proof of this result is based on a suitable geometrical extension of functions from the low-dimensional structure to the whole space in which it is embedded. Then in 4.2 we proceed to the main result which is a proof of the existence of parabolic solutions in the sense of Definition \ref{parabolic2}. To this end, we make use of the results of \cite{Mag89}, where the author study the approach to the Hille-Yosida Theorem based on "forward" iterations of an operator. This allows us to circumvent the use of the resolvent operator and construct the demanded strong solution.

In order to establish the existence of strong-type solutions, we consider a method that does not require additional studies of a higher-order regularity of solutions. The advantage of utilizing the semigroup approach is that we can perform all operations in the space of functions with an assigned Cosserat vector field $b,$ as opposed to methods that rely on studying weak variants of the problem. This is important because obtaining the existence of the Cosserat field $b$ requires additional information about the behaviour near the intersection set, which is typically difficult to control when dealing with weak problems.

The main existential result in the parabolic setting is expressed in the following theorem:
\begin{tw}\label{existence}(Existence and uniqueness of solutions)\cite{Cho24}\\
	Assume that $L_{\mu}=\Delta_{\mu}.$ For any given initial data $g \in A_T(\Delta_{\mu})_N$ (see, Definition \ref{veryreg} in Section 4.1) of the low-dimensional parabolic problem (equation \ref{parabolic}) exists a unique solution $u:[0,T]\to L^2_{\mu}$ in the sense of Definition \ref{parabolic2}. 
\end{tw}

Throughout Section 4 we assume that $\mu \in \mathcal{S}.$ Further we will discuss in more detailed way the need to pose such demanding. Let us only mention that it will be crucial for us to adapt some results discussed in \cite{Bou02}, where the authors pose this kind of assumption on the considered class of measures.

\subsection{Closedness of the operator $L_{\mu}$}{$ $}
\vspace{0.5cm}

To prove this result, we will adapt the semigroup theory, and a conclusion will be derived by applying a certain variant of the Hille-Yosida Theorem, treating semigroups generated by differential operators. Before we proceed to the proof of Theorem \ref{existence}, we need to prove \hbox{Theorem \ref{main},} which is the main technical result of paper \cite{Cho24}.

\begin{tw}\label{main}\cite{Cho24}
	Let $\mu \in \mathcal{S}$ be a low-dimensional structure. The operator \hbox{$L_{\mu}:D(L_{\mu})_N \to L^2_{\mu}$} is closed in the $L^2_{\mu}$-convergence.
\end{tw}

It turns out that the proof of this property in the case of the operator $L_{\mu}$ is more demanding than showing the closedness of $A_{\mu}.$ This phenomenon is related to the fact that a convergence of functions in the sense of $D(L_{\mu})$ space does not provide control of Cosserat vector fields.

\begin{dow2}
Let $\{u_n\}_{n \in \mathbb{N}} \subset D(L_{\mu})_N$ be a sequence such that $u_n \xrightarrow{L^2_{\mu}} u \in L^2_{\mu}$ and $L_{\mu}u_n \xrightarrow{L^2_{\mu}} B \in L^2_{\mu}.$ We need to show that $u \in D(L_{\mu})_N$ and $B = L_{\mu}u.$ A proposed strategy of the proof is based on the fact that the operator $A_{\mu}$ is closed (see Lemma \ref{comp}). 

For the sequence $\{u_n\}_{n \in \mathbb{N}} \subset D(L_{\mu})_N$ we modify the corresponding Cosserat sequence $\{b_n\}_{n \in \mathbb{N}}$ and construct a new sequence of normal vector fields $\{\widetilde{b}_n\}_{n \in \mathbb{N}}$ convergent in the $L^2_{\mu}$-norm and for which the result of Lemma \ref{comp} is valid.

As $D(L_{\mu})_N \subset D(L_{\mu})$ for any element $u_n \in D(L_{\mu})_N$ exists a sequence $\{u^m_n\}_{m \in \mathbb{N}} \subset C^{\infty}_c(\mathbb{R}^3)$ such that $u^m_n \xrightarrow{L^2_{\mu}} u_n,\; \nabla^{\perp}u^m_n \xrightarrow{L^2_{\mu}} b_n$ and $A(u^m_n, \nabla^{\perp}u^m_n) \xrightarrow{L^2_{\mu}} A(u_n, b_n).$ Existence of such sequences is ensured by the definition of the domain $D(A_{\mu})$ (Definition \ref{dommu}). 

After extracting, if it is necessary, a subsequence from the sequence $\{u^m_n\}_{m \in \mathbb{N}}$ (we omit to relabel of the chosen subsequence) we may assume that  
\begin{equation}\label{esty2}
\norm{u^m_n-u_n}_{L^2_{\mu}}\leqslant \frac{1}{m^2},\;\;\;\; \norm{\nabla^{\perp} u^m_n-b_n}_{L^2_{\mu}}\leqslant \frac{1}{m^2},\;\;\;\; \norm{A_{\mu}(u^m_n,\nabla^{\perp} u^m_n)-A_{\mu}(u_n,b_n)}_{L^2_{\mu}}\leqslant \frac{1}{m^2}.
\end{equation}
By the weaker Poincar{\'e} inequality \eqref{weakpoincare} or alternatively by using the classical Poincar{\'e} inequality on each component manifold separately we notice that $\norm{\nabla_{\mu} u^m_n - \nabla_{\mu} u_n}_{L^2_{\mu}} \leqslant \frac{1}{m^2}$ possibly after passing to a subsequence once again.

Before we start the process of modifying the Cosserat vectors $b_n,$ we will show that without any loss on generality, the domain can be 'straightened out' locally. 

From this moment until the end of the proof, we assume that the low-dimensional structure $S$ consists of exactly two component manifolds with a non-trivial intersection. 
At the end of the proof, we will evoke Proposition \ref{proppp} to conclude the global instance. 

We consider two fixed manifolds $E_i,E_j$ such that $E_i \cap E_j \neq \emptyset$ and treat separately each of the following instances: $\dim E_i=\dim E_j=2,$ $\dim E_i=\dim E_j=1,$ and $\dim E_i=2,\;\dim E_j=1.$

Let $\mu \in \mathcal{S},\;$ $\supp \mu = E_i \cup E_j,\;$ $E_i\cap E_j \neq \emptyset,\;$ $\dim E_i = \dim E_j = 2.$ Let us assume that each component $E_k,\; k=i,j$ can be parametrized by a single parametrization -- existence of a partition of unity similar to the one presented in Proposition \ref{proppp} justifies, that such restriction do not narrow the class of considered structures.

Let $\Psi_i$ be a parameterization of $E_i$ (up to the boundary) and $\Psi_i(B^1)=E_i,$ where $B^1$ is the 2-dimensional unit ball in variables $(x,y)$ with the center at zero.  We define a diffeomorphism 
$\Theta_i$ by the formula $\Theta_i(x,y,z):= (0,0,z)+\Psi_i(x,z).$ By a construction of $\Theta_i$ we have \hbox{$\Theta_i^{-1}(E_i \cup E_j) = B^1 \cup F_j,$} where $F_j$ is some 2-dimensional manifold and $B^1 \cap F_j$ is a smooth curve. Let $\Psi_j$ be a parametrisation of $F_j.$ Existence of such mappings is provided by assumptions posed on component manifolds in the definition of the class of low-dimensional structures $\mathcal{S}.$ 

Analogously to what was made in the case of the component $E_i$ we define a diffeomorphism $\Theta_j$ by the formula $\Theta_j(x,y,z):=(0,y,0) + \Psi_j(x,y).$ Now $\Theta_j^{-1}(B^1\cup F_j) = P^1 \cup P^2,$ where $P^1,P^2$ are ``flat'' 2-dimensional manifolds with $P^1 \cap P^2 \subset \left\{(x,y,z)\in \R^3: y=z=0\right\}.$ Let us denote $I:= \Theta_i\circ \Theta_j,$ thus clearly $I^{-1}(E_i \cup E_j) = P^1 \cup P^2.$ Having the "flat" structure $P^1 \cup P^2$ we can easily find a smooth diffeomorphism mapping both components $P^1$ and $P^2$ to $2$-dimensional discs.

If $\mu \in \mathcal{S}$ is the given low-dimensional structure and $\nu := \mathcal{H}^2\lfloor_{P^1}+ \mathcal{H}^2\lfloor_{P^2},$ then obviously $\nu \in \mathcal{S}$ and the $D(A_{\mu})$-convergence is equivalent to the convergence in the sense of $D(A_{\nu}).$ To see this, let us notice that the pushback of the measure $\mu$ by the diffeomorphism $I^{-1}$ is a measure absolutely continuous with respect to $\nu,$ and its $\nu$-related density will be bounded and separated from zero.
A simple computation shows that convergence in the sense of $H^1_{\nu}$ is equivalent to $H^1_{\widetilde{\mu}}$-convergence, $D(A_{\nu})$-convergence is equivalent to $D(A_{\widetilde{\mu}})$-convergence. This also implies that convergence in the sense of $D(L_{\nu})$ is equivalent to convergence in the sense of $D(L_{\widetilde{\mu}}).$ An important consequence of this observation is that in further studies it is enough to focus our attention on the case of the measure $\nu$ being a sum of Hausdorff measures representing each "flattened" component manifold.

$\boldsymbol{\dim E_i=\dim E_j =2}$

We start by examining a case of $\dim E_i=\dim E_j=2.$ Further considerations can be conducted in the coordinate system related to components $E_i,\;E_j.$ This means that without loss of generality we can take $$E_i = \left\{(x,y,0)\in \R^3: x^2+y^2\leqslant 1,\; x,y\in [-1,1]\right\},$$ $$E_j = \left\{(x,0,z)\in \R^3: x^2+z^2\leqslant 1,\; x,z\in [-1,1] \right\}$$ and 
$$\mu = \mathcal{H}^2\lfloor_{E_i} + \mathcal{H}^2\lfloor_{E_j}.$$

As the intersection $E_1\cap E_2$ is the set $\left\{(x,0,0)\in \R^3: x\in [-1,1]\right\}$ we know that for any function $w \in D(L_{\mu})$ a first coordinate of the gradient vector $\nabla_{\mu} w = (w_{x,\mu}, w_{y,\mu}, w_{z,\mu})$ belongs to $H^1_{\mu},$ that is $w_{x,\mu}\in H^1_{\mu}.$ The Cosserat vector field $b$ is needed only to ensure that $(0,w_{y,\mu}, w_{z,\mu}) + b \in (H^1_{\mu})^3.$

For $n \in \mathbb{N},$ define 
$$h_{n,\mu}=\left(0,{h_{n1}}_{\mu},{h_{n2}}_{\mu}\right) := \left(0,{u_{n+1}}_{y,\mu},{u_{n+1}}_{z,\mu}\right) - \left(0,{u_n}_{y,\mu},{u_n}_{z,\mu}\right) \in \left(H^1(E_k)\right)^3,\; k=i,j.$$ 

By a fact that $L_{\mu}u_n,\; n \in \mathbb{N}$ is a Cauchy sequence in $L^2_{\mu}$ we are going to show that, after extracting a suitable subsequence, we have for all $n \in \mathbb{N}$
$$\norm{h_{n,\mu}}_{H^1(E_k)}\leqslant \frac{1}{n^2},\; k\in \{i,j\}.$$ 

The operator $L_{\mu}$ can be decomposed into two classical operators acting on each component $E_k,\; k=i,j$ separately. This observation allows us to apply to restricted operators classical estimates of the elliptic regularity theory (with Neumann boundary conditions). For the detailed formulation of the recalled facts, see the appendix in \cite{Las21}, or \cite{Eva15} Section 6.3.2 for the Dirichlet boundary condition version. 

On each $E_k$ this gives the estimate 
\begin{equation}\label{esty}
\norm{h_{n,\mu}}_{H^1(E_k)}\leqslant C\left(\norm{L_{\mu}u_{n+1} - L_{\mu}u_n}_{L^2(E_k)} + \norm{u_{n+1} - u_n}_{L^2(E_k)}\right).
\end{equation}

By the $L^2_{\mu}$-convergence of the right-hand side, after passing to a subsequence if necessary, we obtain that for $k=i,j\;$ $\norm{h_{n,\mu}}_{H^1(E_k)} \leqslant \frac{1}{n^2}.$

Further, for each $n,m \in \mathbb{N}$ the function $h^m_{n,\mu}$ is introduced as $$h^m_{n,\mu}=\left(0,{h^m_{n1}}_{\mu},{h^m_{n2}}_{\mu}\right) := \left(0,{u^{m+1}_{n+1}}_{y,\mu},{u^{m+1}_{n+1}}_{z,\mu}\right) - \left(0,{u^m_n}_{y,\mu},{u^m_n}_{z,\mu}\right) \in \left(H^1(E_k)\right)^3,\; k=i,j.$$ 

Applying the triangle inequality, estimate \eqref{esty} for the $H^1(E_k)$-norm of $h_{n,\mu},$ estimates \eqref{esty2} for the speed of convergence of sequences of smooth functions $A_{\mu}(u^m_n,\nabla^{\perp}u^m_n)$ and $u^m_n,$ combined with the Poincar{\'e} inequality on each component manifold $E_k,$ the following estimate is obtained (for $k \in \{i,j\}$) 
\begin{equation}
\begin{aligned}\label{hestimate}
\norm{h^m_{n,\mu}}_{H^1(E_k)}\leqslant \norm{u^{m+1}_{n+1} - u^m_{n+1}}_{H^2(E_k)} +\norm{u^m_{n+1} - u_{n+1}}_{H^2(E_k)} \\+ \norm{h_{n,\mu}}_{H^1(E_k)} + \norm{u^m_{n} - u_{n}}_{H^2(E_k)} \leqslant \frac{1}{n^2} +\frac{3}{m^2}.
\end{aligned}
\end{equation}

For each smooth function $u^m_n \in C^{\infty}_c(\R^3)$ we know that $(u^m_n, \nabla^{\perp}u^m_n) \in D(A_{\mu}).$ To ensure good behaviour of the Cosserat sequence our next goal is to modify the sequence  \hbox{$(u^m_n,b^m_n)\in D(A_{\mu}),$} where $b^m_n = \nabla^{\perp}u^m_n.$ Precisely speaking, for each $u^m_n,$ we construct a new normal vector field $\widetilde{b^m_n},$ for which $(u^m_n, \widetilde{b^m_n})$ is also a member of $D(A_{\mu})$ and a sequence of diagonal elements of the presented below infinite lower triangular matrix\\ 
\begin{center}
$\begin{bmatrix}
\nabla_{\mu} u^1_1 + \widetilde{b^1_1} &  &  &  \\
\nabla_{\mu} u^2_1 + \widetilde{b^2_1} & \nabla_{\mu} u^2_2 + \widetilde{b^2_2} &  &  \\
\nabla_{\mu} u^3_1 + \widetilde{b^3_1} & \nabla_{\mu} u^3_2 + \widetilde{b^3_2} & \nabla_{\mu} u^3_3 + \widetilde{b^3_3} &\\
\vdots & \vdots & \vdots & \ddots
\end{bmatrix}.$
\end{center}
converges in the space $(H^1_{\mu})^3.$

We are going to show that the diagonal sequence $\left((u^n_n,\widetilde{b^n_n})\right)_{n \in \mathbb{N}}$ is such that $(u^n_n,\widetilde{b^n_n})\in D(A_{\mu}),\;$ $u^n_n \xrightarrow{L^2_{\mu}} u,\;$ $\widetilde{b^n_n} \xrightarrow{L^2_{\mu}} \widetilde{b}$ for some $u, \widetilde{b} \in L^2_{\mu},$ and $\nabla_{\mu}\left(\nabla_{\mu} u^n_n+\widetilde{b^n_n}\right) \xrightarrow{L^2_{\mu}} \nabla_{\mu}\left(\nabla_{\mu} u+\widetilde{b}\right) \in L^2_{\mu}.$

Let a restriction of $w \in L^2_{\mu}$ to the component $E_k,\; k\in \{i,j\}$ be denoted by $w^k,$ that is $w^k:= w\lfloor_{E_k}.$

For the function $u^m_n$ the new Cosserat vector field $\widetilde{b^m_n}\in L^2_{\mu}(\R^3;T^{\perp}_{\mu})$ is defined as  
$$\widetilde{b^m_n}(x,y,z) := \begin{cases}
\left(0,\; 0,\; {u^m_{n\; z}}^j(x,y)\right) & \text{ on }\; E_i\\
\left(0,\; {u^m_{n\; y}}^i(x,z),\; 0\right) & \text{ on }\; E_j
\end{cases}
.$$

An idea standing behind the sequence $\widetilde{b^m_n}$ is to 'copy' the function ${u^m_{n\; y}}^i$ from the component manifold $E_i$ (in variables $(x,y)$) to the component manifold $E_j$ (in variables $(x,z)$) and analogously to 'copy' ${u^m_{n\; z}}^j$ from $E_j$ (variables $(x,z)$) to $E_i$ (variables $(x,y)$).

To ensure that the constructed pairs belong to the domain $D(A_{\mu})$ we use the fact that in the case of $\mu \in \mathcal{S}$ there exists a characterisation of membership in $D(A_{\mu})$ stated in \hbox{Proposition 3.11} in \cite{Bou02}. The mentioned proposition is based on the idea of constructing a smooth approximating sequence of the class $C^{\infty}(\R^3)$ by extending a sequence initially defined on $\supp \mu$ to the whole space $\R^3$ by the help of the Whitney Extension Theorem. 

To conclude that each pair ${(u^m_n,\widetilde{b^m_n}) \in H^1_{\mu} \times L^2_{\mu}(\R^3;T^{\perp}_{\mu})}$ is an element of the domain $D(A_{\mu})$ we need to verify if the following three conditions are satisfied by
\begin{itemize}
\item[a)] For $k \in \{i,j\},$ we need to check if ${u^m_n}^k \in H^2(E_k)$ and ${b^m_n}^k \in \left(H^1(E_k)\right)^3.$ It turns out to be trivial due to the fact that $u^m_n$ is a smooth function in $\R^3$ and due to the way in which we constructed the Cosserat field $\widetilde{b^m_n}.$
\item[b)] We can use the same arguments as before to verify that the junction set $E_i \cap E_j$ belongs to continuity points of $(u^m_n, \nabla_{\mu}u^m_n + b^m_n).$ This is the second assumption, which needs to be satisfied.
\item[c)] The third condition demands that we check if on each component $E_k$ the corresponding restriction ${u^m_n}^k$ is of class $C^2(E_k),$ the normal vector field ${\widetilde{b^m_n}}^k$ is Lipschitz continuous on $E_k$ and $({u^m_n}^k, \nabla_{E_k}{u^m_n}^k + {\widetilde{b^m_n}}^k)\lfloor_{E_i\cap E_j} = (u^m_n, \nabla_{\mu}u^m_n + \widetilde{b^m_n})\lfloor_{E_i \cap E_j}.$
\end{itemize}

All of the above demands are satisfied immediately, because of the smoothness of $u^m_n$ and the choice of the vector field $\widetilde{b^m_n}.$
As the listed conditions are satisfied, by application of the mentioned characterisation, we have $(u^m_n,\widetilde{b^m_n})\in D(A_{\mu}).$

Now we consider the diagonal sequence $$(u^n_n,\widetilde{b^n_n})_{n \in \mathbb{N}}.$$

By estimate \eqref{hestimate}, the construction of the Cosserat field $\widetilde{b^n_n}$ and a completeness of the space $H^1_{\mu}$ it follows that 
$$\lim_{n \to \infty} \left( \nabla_{\mu}u^n_n +\widetilde{b^n_n} \right) \in \left(H^1_{\mu}\right)^3.$$

This means that 
$$A_{\mu}(u^n_n,\widetilde{b^n_n}) \xrightarrow{\left(L^2_{\mu}\right)^{3\times 3}} D \in \left(L^2_{\mu}\right)^{3\times 3}.$$ 

We know that $u^n_n \xrightarrow{L^2_{\mu}} u \in L^2_{\mu}.$ Moreover, by the way it was constructed, $\widetilde{b^n_n}$ converges in $L^2_{\mu}(\R^3;T^{\perp}_{\mu})$ to some element, from now on called as $\widetilde{b},$ 
 that is  $\widetilde{b^n_n} \xrightarrow{L^2_{\mu}(\R^3;T^{\perp}_{\mu})} \widetilde{b}.$ 
 
A compactness of the operator $A_{\mu}$ (Theorem \ref{main}) implies that
$$A_{\mu}(u^n_n,\widetilde{b^n_n}) \xrightarrow{\left(L^2_{\mu}\right)^{3 \times 3}} A_{\mu}(u,\widetilde{b})$$
and 
$$(u,\widetilde{b}) \in D(A_{\mu}).$$

Finally, under the assumption $L_{\mu}u^n_n \xrightarrow{L^2_{\mu}} B,$ we obtain that 
$$L_{\mu}u^n_n \xrightarrow{L^2_{\mu}} L_{\mu}u,$$ 
thus 
$$L_{\mu}u_n \xrightarrow{L^2_{\mu}} L_{\mu}u$$ 
and $u \in D(L_{\mu}).$\\

On each component manifold $E_k,\; k=1,2,$ we have shown the convergence ${\norm{u^n_n - u_n}_{H^2(E_k)} \to 0,}$ as $n \to \infty.$ This implies that on each boundary $\partial E_k$ we have a convergence of normal traces in $L^2(\partial E_k).$ In this way we conclude that both
$$(u,\widetilde{b}) \in D(A_{\mu})_N\; \text{ and }\; u \in D(L_{\mu})_N.$$

This proves that the operator $L_{\mu}$ is closed in the case of two components of $\dim E_i = \dim E_j =2.$

$\boldsymbol{\dim E_i=\dim E_j =1}$

Finally, we study the remaining case of $\dim E_i=\dim E_j =1.$

The simplicity of a one-dimensional junction gives chances to describe explicitly the domain $D(A_{\mu})$ (see Example 5.2 in \cite{Bou02}). 

Let $P=E_i\cup E_j,$ $\tau$ be a tangent unit vector field and $\nu$ be a normal outer unit vector field. 

As it is stated in paper \cite{Bou02} in Example 5.2 we have the following characterisation of the space $D(A_{\mu})$ in a case of structures consisting only one-dimensional components:
$$D(A_{\mu}) = \{(u,b)\in H^1_{\mu}\times L^2_{\mu}(\R^3;T^{\perp}_{\mu}): \text{ for } k=\{i,j\}, u\lfloor_{E_k}\in H^2(E_k), b\lfloor_{E_k}\in H^1(E_k), u'\tau + b\nu\in C(P)\}.$$

Without loss of generality, we may assume that $E_i\cap E_j = \{(0,0)\} \in \R^2.$ For $k \in \{i,j\},$ let $\widehat{v}:=\eta_{E_k}v$ denote the parallel transport on the component manifold $E_k$ of a vector $v\in \R^2.$ We put $b_{E_i}(0,0) := (\nabla_{\mu}u)\lfloor_{E_j}$ and we define the Cosserat vector field on $E_i$ as $\widehat{b}_{E_i}:=\eta_{E_i}b_{E_i}(0,0).$ 

Analogously we proceed on $E_j.$ Then we put $\widehat{b}:= \widehat{b}_{E_i} + \widehat{b}_{E_j}.$ By the above characterisation of $D(A_{\mu})$ it is easy to see that for any element of the sequence $(u_n,b_n) \in D(A_{\mu})$ there exists a mentioned before modification $\widehat{b}_n$ such that $(u_n,\widehat{b}_n)\in D(A_{\mu}).$ 

 By standard properties of Sobolev functions (e.g., see \cite{Leo09, Thm. 7.13}), if $v \in H^1(E_k),$ then $v \in C(E_k),$ and moreover, if $v_n \in H^1(E_k)$ and $v_n \xrightarrow{H^1(E_k)}v,$ then exists a subsequence converging locally uniformly. Each $\widehat{b}_n$ is well-defined and a vector field $b$ being a $L^2_{\mu}$-limit of the sequence of this Cosserat vectors exists. 

Due to the fact that the $L^2_{\mu}$-convergence of $L_{\mu}u_n$ implies convergence of $u_n$ in the \hbox{$H^2(E_k)$-sense,} and further that $(u,b)\in D(A_{\mu}),$ by closedness of the operator $A_{\mu}$ in $L^2_{\mu}$ we conclude that $u \in D(L_{\mu})$ and $B=L_{\mu}u.$ 

This proves closedness in the sense of $L^2_{\mu}$-norm convergence of the operator $L_{\mu}.$
The same argument as the one used in the previously considered case gives that 
the result is valid also in the subspace $D(L_{\mu})_N.$

\textbf{"From local to global"}

We have examined closedness in the ``local'' variants. To show a similar result on a whole structure, we need to gather all established results.

Let $u_n \in D(L_{\mu})_N,$ $u_n \xrightarrow{L^2_{\mu}}u,$ $L_{\mu}u_n \xrightarrow{L^2_{\mu}} B.$ Denote $u^p_n:=u_n\lfloor_{\supp \mu_p}.$  We have shown that locally for all $p \in I$ exists $u^p \in D(L_{\mu_p})_N$ which satisfies $u^p_n \xrightarrow{L^2_{\mu_p}}u^p$ with  $L_{\mu_p}u^p_n \xrightarrow{L^2_{\mu_p}} L_{\mu_p}u^p.$ 

Using the introduced partition of unity, we write $$u_n = \sum_{p \in I}\alpha_p u^p_n \text{ and } u = \sum_{p \in I}\alpha_p u^p.$$

By the aforementioned observations we have for all $p\in I$ 
\begin{itemize}
\item $\alpha_pu^p_n \in D(L_{\mu_p})_N,$
\item $\alpha_p u^p_n \xrightarrow{L^2_{\mu_p}}\alpha_pu^p,$
\item $L_{\mu_p}(\alpha_pu^p_n) \xrightarrow{L^2_{\mu_p}} L_{\mu_p}(\alpha_pu^p).$
\end{itemize}

Applying Proposition \ref{proppp} gives $u \in D(L_{\mu})_N$ satisfying 
$L_{\mu}u_n \xrightarrow{L^2_{\mu}} L_{\mu}u.$ This proves that the operator $L_{\mu}$ is closed.

\qed
\end{dow2}
 
\begin{kom}\cite{Cho24}
In the above proof, the last two conditions (Section 2.1, points b) and c)) that were given on the measure $\mu \in \mathcal{S}$ played an essential role. However, we believe that these restrictions on the class of accessible measures can be loosened. This is an interesting question that we leave open for further exploration.
\end{kom}

The next propositions show more properties of the operator $L_{\mu}.$ For the need of our future applications, instead of the operator $L_{\mu}$ we consider the operator $Id - \alpha L_{\mu},$ for some fixed positive constant $\alpha.$ 

\begin{prop}\cite{Cho24}
For any $\alpha >0,$ the operator $Id - \alpha L_{\mu}: D(L_{\mu})_N \to L^2_{\mu}$ is closed with respect to the $L^2_{\mu}$-convergence.
\begin{dow}
A sum of closed operators is a closed operator. The identity operator $Id$ is continuous, thus it is closed, and Theorem \ref{main} provides that the operator $L_{\mu}$ is closed. 
\qed
\end{dow}
\end{prop}

\begin{prop}\label{prop4}\cite{Cho24}
For any $\alpha \in (0,+\infty),$ the operator $Id - \alpha L_{\mu}: D(L_{\mu})_N \to L^2_{\mu}$ is injective. 
\begin{dow}
Proceeding by a contradiction, let us assume that there exists a positive eigenvalue $\lambda = \frac{1}{\alpha}$ of the operator $L_{\mu}$ with the corresponding eigenfunction $w \in D(L_{\mu}).$ 

Multiplying the eigenvalue equation $\lambda w = L_{\mu} w$ by the eigenfunction $w,$ using the fact, that as $w \in D(L_{\mu})$ the value of $L_{\mu}$ can be computed componentwise, integrating by parts and by the assumption $(B\nabla_{\mu}u\cdot n)\lfloor_{\partial \Omega} = 0$ we obtain 
$$\lambda\norm{w}^2_{L^2_{\mu}} = \int_{\Omega} L_{\mu} w w d\mu = -\int_{\Omega} B\nabla_{\mu}w \nabla_{\mu} w d\mu.$$ 

By ellipticity of the operator $B$ (or equivalently by parabolicity of $L_{\mu}$) we estimate the integral $\int_{\Omega} B\nabla_{\mu}w \nabla_{\mu} w d\mu$ from above and conclude that
$$-\theta \int_{\Omega} |\nabla_{\mu} u|^2 d\mu \geqslant \lambda\norm{w}^2_{L^2_{\mu}}.$$
This implies, that $\lambda < 0,$ what is a contradiction. 
\qed
\end{dow}
\end{prop}

\begin{prop}\cite{Cho24}
$D(L_{\mu})_N$ is dense in $L^2_{\mu}.$
\begin{dow}
We have the inclusion $\{u\in C^{\infty}(\Omega): (B\nabla_{\mu}u\cdot n)\lfloor_{\partial \Omega}=0\} \subset D(L_{\mu})_N.$ Moreover, the space ${\{u\in C^{\infty}(\Omega): (B\nabla_{\mu}u\cdot n)\lfloor_{\partial \Omega}=0\}}$ is a dense subset of $L^2_{\mu}.$ This two facts imply that $D(L_{\mu})_N$ is dense in $L^2_{\mu}.$ 
\qed
\end{dow}
\end{prop}

Now we can proceed to the construction of a semigroup generated by the operator $\Delta_{\mu}.$

\subsection{Generation of semigroup}{$ $}
\vspace{0.5cm}

Well-known methods of constructing semigroups generated by linear operators are based on the notion of the resolvent operator. In the case considered in our paper \cite{Cho24}, the operator $\lambda Id - \Delta_{\mu}: D(\Delta_{\mu})_N \to L^2_{\mu}$ (or equivalently the operator $Id - \alpha \Delta_{\mu}$) is not surjective, thus its inverse does not exist and it is not possible to introduce the resolvent operator (at least in a standard way). 

To circumvent this problem we apply a technique proposed in \cite{Mag89}, where instead of inverting the given operator, the ``forward'' iterations on a set of "very regular" functions are considered.

Throughout this section, we consider the operator $\Delta_{\mu},$ but it seems possible to generalise the presented method to cover the case of a more general form of the operator $L_{\mu}.$

\begin{prop}\cite{Cho24}\label{ein}
Assume that the operator $G$ generates the semigroup $V,$ $\Delta_{\mu} \subset G$ be a restriction of $A.$ Then for any $u \in A_s(\Delta_{\mu})_N$ we have $\norm{u}_{C^{\infty}(V)}=\norm{u}_{C^{\infty}(\Delta_{\mu})_N}.$
\begin{dow}
An elementary observation. 
\qed
\end{dow} 
\end{prop}

Next proposition shows that spaces $L^2_{\mu,s}$ are independent of the parameter $s>0.$ 

\begin{prop}\cite{Cho24}\label{rownasie}
For any $s>0$ the equality $L^2_{\mu,s} = L^2_{\mu}$ holds.
\begin{dow}
Firstly, let us note that any function $\phi \in C^{\infty}(E_1)$ can be extended to the structure $E=E_1 \cup E_2$ by the formula 
$\widetilde{\phi} =\begin{cases}
\phi, & E_1\\
\phi, & E_2
\end{cases}$ 
is a member of $\bigcap_{n=1}^{\infty}D(\Delta^n_{\mu}).$ This is a result of Proposition 3.11 in \cite{Bou02} characterising membership in the set $D(\Delta_{\mu}).$ 

Choosing a suitable family of smooth functions, we conclude that any function 
$\widetilde{v} =\begin{cases}
v, & E_1\\
v, & E_2
\end{cases},$
where $v \in L^2(E_1),$ satisfies $\widetilde{v} \in \overline{A_s(\Delta_{\mu})_N}^{\norm{\cdot}_{L^2_{\mu}}}.$ To show that the set $\overline{A_s(\Delta_{\mu})_N}^{\norm{\cdot}_{L^2_{\mu}}}$ contains other functions of the class $L^2_{\mu}$ we will follow the following procedure. 

We indicate a family of smooth functions $\phi_n \in L^2(E_1)$ satisfying Neumann boundary conditions on $E_1.$ We also demand that each $\phi_n$ can be extended by zero to the whole $E,$ that is 
\hbox{$\widetilde{\phi_n} =\begin{cases}
\phi_n, & E_1\\
0, & E_2
\end{cases}$} and that $\widetilde{\phi_n}\in A_s(\Delta_{\mu})_N.$ Besides that, the family $\{\phi_n\}$ should span a "large enough" subspace of $L^2_{\mu},$ for instance, the odd subspace of $L^2_{\mu}.$ 

Next, we show that characteristic functions
$\widetilde{\chi_A} =\begin{cases}
\chi_A, & E_1\\
0, & E_2
\end{cases}, A \subset E_1$
can be obtained as a \hbox{$L^2_{\mu}$-limit} of elements of $A_s(\Delta_{\mu})_N.$ This is enough to conclude that $L^2_{\mu}=\overline{A_s(\Delta_{\mu})_N}^{\norm{\cdot}_{L^2_{\mu}}}.$ 

Due to an explicit construction of the sequence $\{\phi_n\}$ we will need to conduct the mentioned procedure separately for the case $\dim E_1=\dim E_2=1$ and the case $\dim E_1=\dim E_2=2.$ 
Let us start with the one-dimensional situation. 

We define the family $\{H_{2k-1}, H_{4k-2}, H_{4k}: k=1,2,3,...\},$ where $H_{4k-2}=
\begin{cases}
\sin(\frac{\pi}{2}nx)+ \frac{\pi}{2}nx, & E_1\\
0, & E_2
\end{cases},$
$H_{4k}=
\begin{cases}
\sin(\frac{\pi}{2}nx)- \frac{\pi}{2}nx, & E_1\\
0, & E_2
\end{cases},$
$H_{2k-1}=
\begin{cases}
\sin(\frac{\pi}{2}nx), & E_1\\
0, & E_2
\end{cases}.$

Let $\{\alpha_n\}\in l^2$ be a sequence of coefficients of the Fourier series of a fixed odd function $f \in L^2(E_1),$ that is $S_N = \sum_{n=1}^N \alpha_n \sin(\frac{\pi}{2}nx)$ and $\norm{S_N - f}_{L^2} \to 0$ as $N \to \infty.$ Let us assume the decay condition $|\alpha_n|\leqslant \frac{C'}{n^4},$ for some positive constant $C'.$ 

Denote $h_n := H_n\lfloor_{E_1}$ and we introduce $\gamma_N:= \sum_{n=1}^N \alpha_n h_n.$ There exists a constant $c \in \R$ for which $\norm{\gamma_N-S_n -cx}_{L^2} \to 0$ as $N \to \infty.$ Thus $\gamma_N \xrightarrow[N\to \infty]{L^2} f + cx.$ 

Our aim is to replace the function $f$ with the function $cx$ and conduct a similar procedure. A series of absolute values of Fourier coefficients of the function $cx$ decays asymptotically to $\frac{1}{n},$ thus the decay condition needed to obtain the above convergence of partial sums is violated, and we cannot proceed directly. 

As odd periodic smooth functions are dense in $L^2(E_1)$ (assuming that the endpoints $-1$ and $1$ of $E_1$ are glued together, thus we consider smoothness in the sense of the torus $\T^1$) let us take a sequence $\{z_m\}$ of such functions converging in the $L^2$-norm to the function $cx.$ It is clear that the Fourier coefficients $\{\widehat{z_m}(k)\}$ of each $z_m$ decay faster then $\frac{C}{k^p}$ for any exponent $p \in \mathbb{N}.$ 

Let $\zeta^m_N:=\sum_{k=1}^N \widehat{z_m}(k) h_k.$ We have $\zeta^m_N \xrightarrow[N \to \infty]{L^2} z_m + c_mx.$ 

By the fact that the $L^2$-convergence of functions implies $l^2$-convergence of corresponding Fourier coefficients we judge that $\zeta_N^N \xrightarrow[N \to \infty]{L^2} cx + \tilde{c}x=(c+\tilde{c})x,$ where $c,\tilde{c}$ are real constants. Now we see that 
$\begin{cases}
(c+\tilde{c})x, & E_1\\
0, & E_2
\end{cases}$ is in $\overline{A_s(\Delta)_N}^{\norm{\cdot}_{L^2_{\mu}}}.$
Without any loss on generality, we might expect that both constants $c$ and $\widetilde{c}$ are non-zero. 

From the above observations we immediately derive that $\gamma_N - \frac{c}{c+\widetilde{c}}\zeta^N_N \xrightarrow[N \to \infty]{L^2} f,$ so the function 
$\begin{cases}
f, & E_1\\
0, & E_2
\end{cases}$ 
is a member of the space $\overline{A_s(\Delta_{\mu})_N}^{\norm{\cdot}_{L^2_{\mu}}}.$ 

Repeating this procedure on the component $E_2$ we prove that $L^2_{odd}(E_1) \times L^2_{odd}(E_2) \subset \overline{A_s(\Delta_{\mu})_N}^{\norm{\cdot}_{L^2_{\mu}}},$ where $L^2_{odd}$ stands for the odd subspace of the space $L^2.$ 

To finish the study of this instance, we notice that both the constant function $1 \in \overline{A_s(\Delta_{\mu})_N}^{\norm{\cdot}_{L^2_{\mu}}}$ and the function 
$\begin{cases}
\begin{cases}
1, & x\in A\\
-1, & x \notin A 
\end{cases}, & E_1\\
0, & E_2
\end{cases} \in \overline{A_s(\Delta_{\mu})_N}^{\norm{\cdot}_{L^2_{\mu}}}$ for any measurable set $A \subset E_1.$ 

As a consequence, we obtain that for any measurable $A \subset E_1$ we have 
$\begin{cases}
\chi_A, & E_1\\
0, & E_2
\end{cases} \in \overline{A_s(\Delta_{\mu})_N}^{\norm{\cdot}_{L^2_{\mu}}}.$ 

This gives that $L^2_{\mu} = \overline{A_s(\Delta_{\mu})_N}^{\norm{\cdot}_{L^2_{\mu}}}$ in the one-dimensional case.

In the two-dimensional case, we will follow an essentially similar way. The crucial difficulty is finding a disc analogue of the functions $h_n.$ 

It turns out that all the necessary requirements are fulfilled by the family \hbox{$\{u_n, n=1,2,3,...\}$} which in polar coordinates can be written as $u_n(\phi,r) = c^{-1}_n \sin(n\phi)J_n(j_n'r),$ where $c^{-1}_n$ is a \hbox{$L^2$-normalizing} constant, $J_n$ is the n-th order Bessel function of the first kind, and $j_n'$ is a root of $J_n'.$ For a more detailed discussion of the family $u_n$ see \cite{Wat62}. 

As $u_n$ is an eigenfunction of the Neumann Laplacian corresponding to the eigenvalue $j_n',$ by the standard facts of the spectral theory we know that $u_n$ are smooth, satisfy Neumann boundary conditions and the functions $u_n$ span the odd subspace of $L^2(E_1).$ Moreover, the eigenvalue equality implies that all derivatives of each $u_n$ have proper growth. 

The polar representation of $u_n$ shows that each $u_n$ and each $\Delta^m u_n$ vanishes on the intersection set $\Sigma = E_1 \cap E_2$ (or this happens up to rotation), thus $u_n$ can be extended by zero to the whole structure $E$ and 
$\widetilde{u}_n=\begin{cases}
u_n, & E_1\\
0, & E_2
\end{cases}$
satisfies $\widetilde{u}_n \in A_s(\Delta_{\mu})_N.$ 
Now, we are able to proceed analogously to the earlier discussed one-dimensional case. 

In this way we conclude $L^2_{\mu} = \overline{A_s(\Delta_{\mu})_N}^{\norm{\cdot}_{L^2_{\mu}}}$ what finishes the proof. 
  
\qed
\end{dow}
\end{prop}	
	
The fact that convergence in the sense of the graph of $\Delta_{\mu}$ implies convergence in the Sobolev norm $H^2$ on each component manifold will be needed in further considerations.

\begin{prop}\cite{Cho24}\label{zwei}
Assume that $u_n \in A_s(\Delta_{\mu})_N,$ both the sequence $u_n$ and the sequence $\Delta_{\mu} u_n$ are convergent in the $L^2_{\mu}$-norm, then the sequence $u_n$ is convergent in $H^2(E_i),\; i=1,2.$
\begin{dow}
The evoked earlier characterisation (see Proposition 3.11 in \cite{Bou02}) of the space $D(\Delta_{\mu})$ shows that membership $u_n \in A_s(\Delta_{\mu})_N$ implies $u_n \in H^2(E_i),$ for $i=1,2.$ 

By considering separately each component $E_i,\; i=1,2$ and using the local regularity estimates (see for instance \cite{Eva10, p.306}) we obtain that convergence in the sense of graph of $\Delta$ implies convergence in the seminorm ${\norm{\nabla^2 \cdot}_{L^2_{\mu}}.}$By the interpolation theorem (see Thm 5.2, p.133 in \cite{Ada03}) we obtain that a joint convergence in seminorms $\norm{\cdot}_{L^2(E_i)}$ and $\norm{\nabla^2\cdot}_{L^2(E_i)}$ is equivalent to a convergence in the standard $H^2(E_i)$-norm. 
\qed
\end{dow}
\end{prop}\\	
	
To provide that a solution constructed by the action of a semigroup is fully valuable, we need to ensure that for each time $t$ it is a member of the domain of the generator. 

\begin{lem}\cite{Cho24}
Let $u \in A_s(\Delta_{\mu})_N,$ let $V_s$ be a family of operators constructed in the proof of Theorem 2 in \cite{Mag89}. Then for each $t \in (0,s)$ we have $V_s(t)u \in D(\Delta_{\mu})_N.$
\begin{dow}
From the proof of Theorem 2 in \cite{Mag89, p.99} it follows that for an arbitrary $u \in A_s(\Delta_{\mu})_N$ and $t \in (0,s)$ we have $V_s(t)u \in \overline{A_s(\Delta_{\mu})_N}^{\norm{\cdot}_{C^{\infty}(V_s)}}.$ 

By the equality of Proposition \ref{ein} it follows $V_s(t)u \in \overline{A_s(\Delta_{\mu})_N}^{\norm{\cdot}_{C^{\infty}(\Delta_{\mu})_N}}.$ Let $w_n \in A_s(\Delta_{\mu})_N$ be a Cauchy sequence in $\norm{\cdot}_{C^{\infty}(\Delta_{\mu})_N}.$ 

As Proposition \ref{zwei} implies convergence of $w_n$ in the sense of $H^2(E_i)$ on each component manifold $E_i,\; i=1,2,$ and as the normal trace $\frac{\partial w_n}{\partial n}\lfloor_{\partial E_i}:H^1(E_i) \to L^2(E_i), i=1,2$ is continuous we conclude that the limit $w:=\lim_{n \to \infty}w_n$ in $\norm{\cdot}_{C^{\infty}(\Delta_{\mu})_N}$ satisfies 
$\frac{\partial w}{\partial n}\lfloor_{\partial E_i}=0,\; i=1,2.$ It is easy to check that $w \in H^2(E_i)$ and $w \in C(E).$ 

Now using the reasoning presented in the proof of closedness of the operator $L_{\mu}$ we deduce that $w \in D(\Delta_{\mu})_N.$ In this way we obtained $V_s(t)u \in D(\Delta_{\mu})_N,$ for all $t \in (0,s).$
\qed
\end{dow}
\end{lem}	

We will also need the following simple observation.

\begin{prop}\cite{Cho24}
$\mathcal{E}(\Delta_{\mu})^T_N$ is dense in $L^2_{\mu}.$
\begin{dow}
It follows directly from the proof of Proposition \ref{rownasie}.
\qed
\end{dow}
\end{prop}	

The following proposition is equivalent to the continuity of the resolvent operator if such an operator is well-defined.
	
\begin{prop}\cite{Cho24}\label{drei}
For any neighbourhood of zero $W \subset L^2_{\mu}$ exists a neighbourhood of zero $U \subset L^2_{\mu}$ such that if for all $k\in \{1,2,...\},$ for all $u \in D(\Delta_{\mu}^k)_N$ and some constants $\alpha,C>0$ we have $(1-\alpha C)^{-k}(Id-\alpha \Delta_{\mu})^k u\in U,$ then $u \in W.$
\begin{dow}
Let us recall that the operator $Id - \alpha \Delta_{\mu}:D(\Delta_{\mu})_N \subset L^2_{\mu} \to L^2_{\mu}$ is closed due to closedness of $\Delta_{\mu}.$ Assume that the statement postulated in the thesis is not true. This means that the following sentence is valid:
$\exists k \in \{1,2,...\}\; \exists u \in D(\Delta_{\mu}^k)_N,$ exists a neighbourhood of zero $W \subset L^2_{\mu}$ such that for any neighbourhood of zero $U \subset L^2_{\mu}$ we have $(1-\alpha C)^{-k}(Id-\alpha \Delta_{\mu})^k u \in U \wedge u \notin W.$ 

This implies existence of a sequence $u_n \in D(\Delta_{\mu})_N$ such that $\norm{(Id-\alpha \Delta_{\mu})u_n}_{L^2_{\mu}} \to 0$ and \hbox{$\norm{u_n}_{L^2_{\mu}}>c'>0.$} By continuity of the classical resolvent operator defined for the Laplace operator $\Delta$ on each component $E_i,$ we see that the convergence $\norm{(Id-\alpha \Delta_{\mu})u_n}_{L^2_{\mu}} \to 0$ implies $\norm{u_n}_{L^2_{\mu}} \to 0$ as $n \to \infty.$
By induction over $k \in \{1,2,...\},$ we stipulate that the condition expressed in the thesis is valid. 
\qed
\end{dow}
\end{prop}	
	
Finally, we can summarise our study over generating semigroups in the following theorem.

\begin{tw}\cite{Cho24}\label{eqisemi}
The operator $\Delta_{\mu}$ generates equicontinuous semigroup $S:[0,T] \times L^2_{\mu} \to L^2_{\mu}.$
\begin{dow}
We verified, that on the interval $[0,T]$ all conditions of Theorem 2 of \cite{Mag89} (see p.95) are satisfied. This implies, that $\Delta_{\mu}$ generates an equicontinuous semigroup \hbox{$S:[0,T] \times L^2_{\mu} \to L^2_{\mu}.$}
\qed
\end{dow}
\end{tw}

We may conclude the results established in this section with the proof of Theorem \ref{existence}.

\begin{dow1}\cite{Cho24}
Let $S$ be the semigroup as in Theorem \ref{eqisemi}. As $g \in A_T(\Delta_{\mu})_N$ and the semigroup $S$ is generated by the operator $\Delta_{\mu}$ we immediately notice that the function $u(t,x):=(S(t)g)(x)$ is a solution in a sense of Definition \ref{parabolic2}. {\color{blue}The uniquness of solution is a standard result and follows from \cite{Paz83, Thm. 2.6}.} 
\qed
\end{dow1}

\begin{przyk}\cite{Cho24}
We present an example of an application of Theorem \ref{existence} to the existence of solutions in a case of two orthogonal discs intersecting each other.

Let $\mu \in \mathcal{S}$ be a low-dimensional structure of the form $$\mu:=\mathcal{H}^2|_{D_1}+\mathcal{H}^2|_{D_2},$$ where $D_1:=\{(x,y,0)\in \R^3: x^2+y^2\leqslant 1\}$ and $D_2:=\{(x,0,z)\in \R^3: x^2+y^2 \leqslant 1\}.$ We denote $D:= D_1 \cup D_2$ and $\partial D := \partial D_1 \cup \partial D_2.$

Assume that the matrix $B=(b_{ij}),\; i,j\in\{1,2,3\}$ from Definition \ref{secondop} consists of constant entries $b_{ii}\equiv 1,\; i \in \{1,2,3\}$ and $b_{ij}\equiv 0,\; i\neq j,\; i,j\in\{1,2,3\}.$ 

The operator $L_{\mu}:D(L_{\mu})_N \to L^2_{\mu}$ related with this matrix is of the simple form $$L_{\mu}u=\sum_{i=1}^3(\nabla^2_{\mu}u)_{ii} = \Delta_{\mu}u.$$

We consider the evolutionary heat equation
\begin{equation}\label{parabolic3}
\begin{aligned}
u_t - \Delta_{\mu}u = 0 \;\;\;(\mathcal{H}^2|_{D_1}+\mathcal{H}^2|_{D_2})\times l^1([0,T])-\text{a.e. in } (D_1 \cup D_2) \times [0,T] \\
(\nabla_{\mu} u,\eta) = 0 \;\text{ on } (\partial D_1 \cup \partial D_2) \times [0,T] \\
u = g \;\text{ on } (D_1 \cup D_2) \times \{0\}.
\end{aligned}
\end{equation}
Here $g \in L^2_{\mu}$ and $\eta$ is a vector field such that $\eta\lfloor_{D_1}$ is the outer normal unit vector field to $D_1$ and $\eta\lfloor_{D_2}$ is the outer normal unit vector field to $D_2.$

By the result of Theorem \ref{existence}, there exists a unique function $u:[0,T]\to L^2_{\mu}$ which is a solution of the heat transfer issue \eqref{parabolic3} in the sense of Definition \ref{parabolic2}.  
\end{przyk}

\newpage

\

\newpage
\newpage
\section{Non-semigroup approach}	
The results of this section are established in our paper \cite{Cho24}.	
 
 A fundamental drawback of the presented semigroup method is that it provides the existence of solutions in a narrow space of functions and under the assumption of a very regular initial input, see Theorem \ref{eqisemi} in the previous section. This is closely related to two aspects: a form of the considered second-order operator and the non-existence of the corresponding resolvent operator. 
 
 The second-order operator $L_{\mu}$ was constructed to be consistent with the variational theory introduced in \cite{Bou02} and the operators used there. The operator $L_{\mu}$ needs restrictive conditions to be posed on function spaces to provide well-definiteness of it. Further restrictions are caused by the fact that the considered operator is non-invertible, implying that the resolvent operator does not exist (at least in a classical sense). Due to this, our method of constructing solutions bases on "forward" iterations of the considered operator. Such a method demands further significant restrictions on the considered space of functions. 
	
 Now our goal is to touch the low-dimensional parabolic issues differently. Firstly ensuring that the class of solutions is wide enough, and later on, we try to deduce additional regularity of obtained solutions.
	
 In this section, we focus on constructing solutions to weak versions of parabolic problems with initial data of low regularity, and in further considerations, we show in what sense the regularity of it can be upgraded.
	We begin by introducing a notion of weak parabolic problems and by examining the existence of solutions in the first-order framework of \cite{Ryb20}. 

 Throughout this section, we can loosen our restrictions on the class of considered low-dimensional structures. We need only to assume that $\mu \in \widetilde{\mathcal{S}}.$ This means that components of low-dimensional structures may have non-fixed dimensions.
 
We adapt the framework exposed, for instance, in book \cite{Sho97} to provide the existence of solutions to the weak parabolic problems on the low-dimensional structures. 

We will need the next proposition to prove the existence of solutions to \eqref{parabolicweak}.

\begin{prop}\cite{Cho24}\label{conde}
For any $\phi \in \mathring{\mathcal{T}}$ we have
\begin{itemize}
\item[a)] $\norm{\phi}_{L^2H^1_{\mu}} \leqslant \norm{\phi}_{{\mathcal{T}}},$
\item[b)] $E(\phi,\phi)\geqslant C\norm{\phi}^2_{{\mathcal{T}}},$ where $C>0$ is a constant independent of a choice of $\phi.$ 
\end{itemize}
\begin{dow}
Point a) follows directly from the definitions of the related norms. This means that the space $\mathring{\mathcal{T}}$ embeds continuously in the space $\mathring{\mathcal{H}}.$ We move to the proof of point b). 

Using integration by parts with respect to the time variable and ellipticity of the matrix operator $B,$ we see that 
$$E(\phi,\phi) = \int_0^T\int_{\Omega} (B\nabla_{\mu} \phi, \nabla_{\mu} \phi) d\mu dt - \int_0^T \int_{\Omega} \phi \phi' d\mu dt \geqslant C'\left(\norm{\nabla_{\mu} \phi}^2_{L^2L^2_{\mu}} + \norm{\phi(0)}^2_{L^2_{\mu}}\right)$$ with a constant $C'>0$ that does not depend on $\phi.$ 

Rewriting the $\mathcal{T}$-norm, we estimate
$$\norm{\phi}^2_{\mathcal{T}} = \norm{\phi}^2_{L^2H^1_{\mu}}+\norm{\phi(0)}^2_{L^2_{\mu}}=\norm{\nabla_{\mu}\phi}^2_{L^2L^2_{\mu}}+\norm{\phi}^2_{L^2L^2_{\mu}}+\norm{\phi(0)}^2_{L^2_{\mu}}\leqslant (1+C_p)\norm{\nabla_{\mu}\phi}^2_{L^2L^2_{\mu}}+\norm{\phi(0)}^2_{L^2_{\mu}},$$ where the constant $C_p$ comes from the generalized Poincar{\'e} inequality (formula \eqref{weakpoincare} in \hbox{Section 2.2}). As a conclusion we derive that $C\norm{\phi}^2_{\mathcal{T}}\leqslant E(\phi,\phi).$
\qed
\end{dow}
\end{prop}	

Let us recall the Lions variant of the Lax-Milgram Lemma (see, for instance, \cite{Sho97}).

\begin{tw}\cite{Sho97}\label{lions}
Let $\mathcal{M}$ be a Hilbert space equipped with the norm $\norm{\cdot}_{\mathcal{M}}$ and $\mathcal{N}$ with the norm $\norm{\cdot}_{\mathcal{N}}$ be a normed space. Let $H:\mathcal{M} \times \mathcal{N} \to \R$ be a bilinear form and assume that for any $\phi \in \mathcal{N}$ we have $H(\cdot,\phi)\in \mathcal{M}^*.$ Then the condition $$\inf_{\norm{\phi}_{\mathcal{N}}=1}\sup_{\norm{u}_{\mathcal{M}}\leqslant 1}|H(u,\phi)|\geqslant c >0$$ is equivalent to the fact that for any $F \in \mathcal{N}^*$ there exists $u \in \mathcal{M}$ such that for all $\phi \in \mathcal{N}$ we have $H(u,\phi)=F(\phi).$
\end{tw}

The following proposition (see \cite{Sho97}) serves as a useful criterion for verifying that one of the implications of the Lions Theorem is valid.

\begin{prop}\cite{Sho97}\label{lionscol}
Let there exists a continuous embedding of $\mathcal{N}$ in $\mathcal{M}.$ If there is some positive constant $A$ such that $H(\phi,\phi)\geqslant A\norm{\phi}^2_{\mathcal{N}}$ for all $\phi \in \mathcal{N},$ then for any $F \in \mathcal{N}^*$ exists $u \in \mathcal{M}$ satisfying $H(u,\phi)=F(\phi)$ for all $\phi \in \mathcal{N}.$
\end{prop} 

Now we are prepared to deal with the well-posedness of problem \eqref{parabolicweak}.

\begin{tw}\cite{Cho24}\label{weakpara}
Let $f \in L^2(0,T;{L^2_{\mu}})$ and $u_0 \in \mathring{L^2_{\mu}}.$ The problem \eqref{parabolicweak} has a unique solution $u \in \mathring{\mathcal{H}}.$ 
\begin{dow}
Proposition \ref{conde} implies that the assumption of Proposition \ref{lionscol} is satisfied. Thus by applying the Lions Theorem (Theorem \ref{lions}) the existential part is done. 

To derive the uniqueness of solutions notice that point b) of Proposition \ref{conde} is stricter than the monotonicity of the operator. Thus the uniqueness is provided by applying Proposition 2.3 formulated on p.112 of \cite{Sho97}.
\qed
\end{dow}
\end{tw}	

As the existence of weak solutions is already discussed, our next goal is to analyse their regularity. 

\subsection{More regular solutions}{$ $}
\vspace{0.5cm}

Let us begin with the following, at this moment only formal, computations:

$\int_{\Omega} \nabla_{\mu} u \cdot \nabla_{\mu}v d\mu = \int_{E_1}\nabla_{\mu} u \cdot \nabla_{\mu}v d\bar{x} + \int_{E_2}\nabla_{\mu} u \cdot \nabla_{\mu}v d\bar{x} = 
-\int_{E_1} \Delta u v d\bar{x} + \int_{\partial E_1} \frac{\partial u}{\partial n} v d\sigma - \int_{E_2} \Delta u v d\bar{x} + \int_{\partial E_2} \frac{\partial u}{\partial n} v d\sigma = - \int_{\Omega} Au v d\mu + \int_{\partial E} \frac{\partial u}{\partial n} v d\sigma,$ 
here we denote $A:= \Delta_{E_1}+ \Delta_{E_2}.$

It should be clear that the above computations make sense only when we additionally assume ${u\in H^2(E_i),\; i=1,2.}$ 

Moreover, introducing the zero Neumann boundary condition on each component $E_i,\; i=1,2$ it follows by continuity of the operator $A:H^2(E_1)\times H^2(E_2) \to L^2_{\mu}$ that $\int_{\Omega} \nabla_{\mu} u \nabla_{\mu}v d\mu \leqslant C \norm{v}_{L^2_{\mu}},$ for all $v \in H^1_{\mu}.$ \\ 

\begin{defi}\cite{Cho24, Sho97}\label{strop}\\
Let $D(\mathcal{A}) \subset \mathring{H^1_{\mu}}$ such that for $u \in D(\mathcal{A})$ we have ${\int_{\Omega} (B\nabla_{\mu} u, \nabla_{\mu} v) d\mu \leqslant C \norm{v}_{L^2_{\mu}}}$ for all $v \in \mathring{H^1_{\mu}}$ and some positive constant $C.$

We define the operator ${\mathcal{A}:D(\mathcal{A})\to L^2_{\mu}}$ by the equality $\int_{\Omega} (B\nabla_{\mu} u, \nabla_{\mu} v) d\mu = \int_{\Omega} (\mathcal{A}u) v d\mu$ satisfied for all $v \in \mathring{H^1_{\mu}}.$ Existence of the element $\mathcal{A}u$ is guaranteed by the Riesz Representation Theorem for functionals on a Hilbert space.
\end{defi}	

\begin{kom}\cite{Cho24}
Let us point out that the operator $\mathcal{A}$ is consistent with the integration by parts formula. We do not assume enough regularity on $u$ to ensure the existence of the normal trace in the classical sense. We must follow the generalized approach by applying the results established in \cite{Che01}. This theory provides the existence of the normal trace and the operator $\mathcal{A}$ in the sense of distributions or measures. 

Further, by a condition from the definition of the domain $D(\mathcal{A})$ (Definition \ref{strop}), we represent the normal trace and other terms of the integration by parts formula (we refer to Theorem 2.2 in \cite{Che01}) as functionals on adequate variants of $L^2$ spaces. 

In this way we obtain $\int_{\Omega} \mathcal{A}u v d\mu + \int_{\partial \Omega}[B\nabla_{\mu}u,\nu]\lfloor_{\partial \Omega} v d\sigma \leqslant C \norm{v}_{L^2_{\mu}}.$ In a case if $[B\nabla_{\mu}u,\nu]\lfloor_{\partial \Omega}$ is non-zero somewhere on $\partial \Omega,$ then we can find a sequence of continuous functions $v_n \in \mathring{H^1_{\mu}}$ such that the second integral on the left side of the inequality diverges and the other terms on both sides are bounded. This indicates that in fact the operator $\mathcal{A}$ satisfies $\int_{\Omega} B\nabla_{\mu} u \nabla_{\mu} v d\mu = \int_{\Omega} \mathcal{A}u v d\mu.$

Finally, let us observe that ${\{u \in \mathring{H^1_{\mu}}: u|_{E_i}\in H^2(E_i),\; [B\nabla_{\mu}u,\nu]\lfloor_{\partial \Omega}=0,\; i=1,2\} \subset D(\mathcal{A}).}$ 
\end{kom}

\begin{kom}\label{self}\cite{Cho24}
The operator $\mathcal{A}:D(\mathcal{A}) \to L^2_{\mu}$ is selfadjoint. This follows directly from the symmetry of the matrix operator $B(\bar{x})$ for $\mu$-a.e. $\bar{x}\in E$ and the symmetry of the bilinear form $(\cdot,\cdot)_{L^2_{\mu}}.$
\end{kom}\\

Using the theorem of Lions - Theorem \ref{lions}, we prove the existence of strong solutions in the sense of the operator $\mathcal{A}.$

\begin{prop}\cite{Cho24}\label{exiw}
Let $f \in L^2(0,T;{L^2_{\mu}}), u_0 \in \mathring{H^1_{\mu}}.$ Then there exists a unique function \hbox{$u \in W^{1,2}(0,T;{L^2_{\mu}})$} satisfying $u' + \mathcal{A}u = f$ in $L^2(0,T;L^2_{\mu})$ with $u(0)=u_0.$ Moreover, for almost every $t \in (0,T)$ it is $u(t) \in D(\mathcal{A}).$
\begin{dow}
By a fact that the operator $\mathcal{A}$ is obtained from the non-negative, symmetric bilinear form, see Definition \ref{strop}, and due to its self-adjointness (see Comment \ref{self} above) we are in a case covered by Proposition 2.5 of \cite{Sho97}. Proposition 2.5 of \cite{Sho97} allows adapting the Lions Theorem to the equation with the operator $\mathcal{A}.$ In this way we obtain existence of unique function $u$ such that equality $u' + \mathcal{A}u = f$ is satisfied in the norm of the space $L^2(0,T;L^2_{\mu}).$
\qed
\end{dow} 
\end{prop}

Now we focus on extending a class of accessible initial data $u_0$ and discuss remarks dealing with the regularity of the obtained solution.

\begin{prop}\cite{Cho24}
A solution determined in Proposition \ref{exiw} is a member of the space of functions $C([0,T];\mathring{H^1_{\mu}}).$
\begin{dow}
This is a result of an equivalence of $\int_{\Omega}(B\nabla_{\mu}u,\nabla_{\mu}v)d\mu$ and the scalar product of the Hilbert space $\mathring{H^1_{\mu}}.$ Such equivalence is valid because we assume that on the low-dimensional structure $\mu$ the generalized Poincar{\'e} inequality is true. Moreover, the mentioned equivalence implies that the assumption $u_0 \in \mathring{H^1_{\mu}}$ cannot be relaxed in the given class of solutions to the considered problem.
\qed
\end{dow}
\end{prop}

Finally, we examine the regularity of solutions in the setting where as an initial data we can take the function of the class $\mathring{L^2_{\mu}}.$ 

\begin{prop}\cite{Cho24}
Let the operator $\mathcal{A}$ and the function $f$ be as in Proposition \ref{exiw}. Assume that ${u_0 \in \mathring{L^2_{\mu}}.}$ Then the unique solution $u \in W^{1,2}(0,T;{L^2_{\mu}})$ of the problem
$u'+\mathcal{A}u=f$ in $L^2(0,T;H^{1}_{\mu})$ with $u(0)=u_0 \in \mathring{L^2_{\mu}}$ satisfies $t^{1/2}u' \in L^2(0,T;L^2_{\mu})$ and for an arbitrary $d\in(0,T)$ satisfies $u \in W^{1,2}(d,T;L^2_{\mu}) \cap C([d,T];\mathring{H^1_{\mu}}).$
\begin{dow}
Under the given assumptions, we are in the regime of Corollary 2.4 from \cite{Sho97}, which gives the proposed statements.
\qed
\end{dow}
\end{prop}

A relation between solutions of the form evoked in Proposition \ref{exiw} and weaker solutions recalled in Theorem \ref{weakpara} is investigated in the below lemma.

\begin{lem}\cite{Cho24}
Let $u\in W^{1,2}(0,T;{L^2_{\mu}}), u_0\in \mathring{{H}^1_{\mu}}$ satisfies the equation $u'+\mathcal{A}u=f$ in $L^2(0,T;L^2_{\mu})$ with $u(0)=u_0.$ Then $u$ is a weak solution in the sense of Definition \ref{weak3}. 
\begin{dow}
After multiplying the equation $u'+\mathcal{A}u=f$ by a test function $\phi \in \mathring{\mathcal{C}^{\infty}_0}$ and integrating over $\Omega$ and time, we obtain
$\int_0^T \int_{\Omega}u' \phi + \mathcal{A}u \phi d\mu \dt = \int_0^T \int_{\Omega}u' \phi + (B\nabla_{\mu}u, \nabla_{\mu}\phi) d\mu \dt.$

Continuing by integrating by parts the term with the time derivative it follows that\\
$\int_0^T \int_{\Omega}u' \phi + (B\nabla_{\mu}u, \nabla_{\mu}\phi) d\mu \dt=\int_0^T \int_{\Omega} (B\nabla_{\mu}u, \nabla_{\mu}\phi) - u\phi' d\mu \dt+ \int_{\Omega} u_0 \phi(0)d\mu.$ From this computations we conclude that $E(u,\phi) = F(\phi),$ thus the function $u$ is a weak solution in a sense of Definition \ref{weak3}.
\qed
\end{dow}
\end{lem}

\subsection{Examples}{$ $}
\vspace{0.5cm}

The following examples are evoked to present that in the low-dimensional setting, even the simplest stationary elliptic equation might possess weak solutions which give a new quality. 

We will show that there are solutions that are not simple gluings of classical solutions of projected problems on component manifolds. 

\begin{przyk}\cite{Cho24}\label{ex2}
	Assume that $\Omega \subset \mathbb{R}^2$ is a 2-dimensional unit ball $B(0,1).$ Define the subsets \hbox{$E_1:=\{(y,0): y\in [-1,1]\},$} $E_2:=\{(0,z):z \in [-1,1]\}$ and the corresponding measure $\mu := \mathcal{H}^1|_{E_1}+\mathcal{H}^1|_{E_2}.$ 
 Let us consider the function
	$f:=\begin{cases}
	y \text{ on } E_1,\\
	0 \text{ on } E_2
	\end{cases}.$  
	
 Clearly we have $f \in \mathring{L^2_{\mu}}.$	Let us consider the stationary heat problem (see \cite{Ryb20} for the existence and uniqueness result)
	\begin{equation}\label{slabe}
	\int_{\Omega} \nabla_{\mu}u\cdot \nabla_{\mu} \phi d\mu = \int_{\Omega} f \phi d\mu
	\end{equation}
	for $\phi \in C^{\infty}_c(\R^2).$
	In this case the projected gradient $\nabla_{\mu}$ has a form of the classical 1-dimensional derivative $\partial$ in variables corresponding to each component $E_i, i=1,2.$
 
 We can write down equation \eqref{slabe} in a form
	\begin{equation*}
	\int_{E_1}\partial_y u \partial_y \phi dy + \int_{E_2}\partial_z u \partial_z \phi dz = \int_{E_1} y \phi dy + \int_{E_2} 0 \phi dz = \int_{E_1} y \phi dy.
	\end{equation*}
	On the component $E_1$ let us take $u_1(y):= -\frac{21}{1080}-\frac{y^4}{12}+\frac{y^3}{6}+\frac{y^2}{6}-\frac{y}{2}.$ 
 
 Computations show that $u_1'(-1)=u_1'(1)=0,$ $u_1(0)=\frac{-21}{1080},$ and $\int_{E_1}u(y)dy = \frac{7}{180}.$ Let us put on the component $E_2$ the constant function $u_2(z):=-\frac{21}{1080}.$ The function $u_2$ has a mean over $E_2$ equal to $-\frac{7}{180}.$ 
 
 It is easy to observe that neither $u_1$ nor $u_2$ satisfies the weak equation separately on components, that is $u_i, i=1,2$ is not the solution of $\int_{E_i} \partial_x u_i \partial_x \phi dx = \int_{E_i} f \phi dx,$ for $x\in \{y,z\},$ due to the fact that $\int_{E_i} u_i dx \neq 0.$ 
 
 On the other hand, the function 
	$u := \begin{cases}
	u_1 \text{ on } E_1,\\
	u_2 \text{ on } E_2
	\end{cases}$
	belongs to $\mathring{H^1_{\mu}}$ and is a solution to weak problem \eqref{slabe}. Moreover, in the considered case, the uniqueness of solutions to \eqref{slabe} is guaranteed by Theorem 1.1 of \cite{Ryb20}. This is equivalent to satisfying that a solution is globally continuous and has a mean zero. 
\end{przyk}

The second example shows that a solution to the weak low-dimensional Poisson problem \eqref{slabe} may differ from a sum of solutions to classical problems considered separately on each component manifold.

\begin{przyk}\cite{Cho24}\label{ex3}
	We are considering Poisson equation \eqref{slabe} taking for $\Omega \subset \R^3$ a 3-dimensional unit ball $B(0,1),$ the manifolds $E_1:=\{(x,y,0):x^2+y^2\leqslant 1\},$ \hbox{$E_2:=\{(x,0,z): x^2+z^2\leqslant 1\}$} and the measure $\mu := \mathcal{H}^2|_{E_1}+\mathcal{H}^2|_{E_2}.$ 
 
 Let us define a function $w:E_1 \to \R,$ $w(x,y):=\cos(\pi(x^2+y^2))$ and for the force term let us put the function
	$f:= \begin{cases}
	\Delta w \text{ on } E_1,\\
	0 \text{ on } E_2
	\end{cases}.$  
 
 After some computations we can check that $\nabla w = -2\pi\sin(\pi(x^2+y^2))(x,y),\;$ $\frac{\partial w}{\partial n}\lfloor_{\partial E_1} = 0,$ and \hbox{$\Delta w = -4\pi \sin(\pi (x^2+y^2)) - 4 \pi^2 (x^2+y^2) \cos(\pi (x^2+y^2)).$} By a change of variables for the polar coordinates, we can easily verify that $\int_{E_1} \Delta w d\bar{x} =0.$
	
 The previous observations show that $f \in \mathring{L^2_{\mu}}.$ Both $w$ and the zero function are the solutions of the classical weak Poisson problem posed on the corresponding component, but the function 
	$u:=\begin{cases}
	w \text{ on } E_1\\
	0 \text{ on } E_2
	\end{cases}$
	does not belong to $\mathring{H^1_{\mu}}$ or even to $H^1_{\mu}.$ It should be noted that even the modified function 
	$\widetilde{u}:=\begin{cases}
	w +c_1 \text{ on } E_1\\
	0 +c_2 \text{ on } E_2
	\end{cases}$
	for any $c_1, c_2 \in \R$ is not a member of $H^1_{\mu}.$ 
 
 The general existence theorem established in \cite{Ryb20} can be applied to this problem and provides the existence of the unique solution $\widehat{u}\in \mathring{H^1_{\mu}}.$ The conducted reasoning shows that the solution $\widehat{u}$ will be different from any function $\widetilde{u}$ and thus will be different from the function $u$ in an essential way, meaning that neither only by some constant nor by two different constants added independently on components. 
\end{przyk}

\subsection{Asymptotic convergence}{$ $}
\vspace{0.5cm}

We focus on studying the asymptotic behaviour of weak solutions to parabolic problems. 

Using the existential theorems established in Section 4.1, the variational methods of \cite{Ryb20} and transferring the results of \cite{Gol08}, we prove that solutions to parabolic problem \eqref{parabolicweak} converge as $t \to \infty$ to a solution of the stationary heat equation.

We will work with a slightly different formulation of the weak equation. Thus we will need an equivalence (see Proposition 2.1 in \cite{Sho97}) phrased by the below proposition. 

\begin{prop}\cite{Cho24}\label{edef}
Weak formulation of the parabolic initial problem \eqref{parabolicweak} is equivalent to a problem of finding $u \in \mathring{\mathcal{H}}$ with $u(0)=u_0 \in L^2_{\mu}$ such that the equation $u'(t)+\mathcal{A}u(t)=f(t)$ is satisfied in the sense of $\mathring{\mathcal{H}}^*$ for a.e. $t\in (0,+\infty).$
\begin{dow}
This can be proved by applying Proposition 2.1 of \cite{Sho97} to equation \eqref{parabolicweak} in the corresponding framework.
\qed
\end{dow}
\end{prop} 

To shorten the notation, we introduce a notion of the energy functional related to the considered problem.

\begin{defi}\cite{Cho24}
Let a functional $E_{\mu}:\mathring{H^1_{\mu}} \to \R$ be defined as $E_{\mu}(u):=\int_{\Omega} (B\nabla_{\mu}u,\nabla_{\mu}u)d\mu.$ The functional $E_{\mu}$ is called the energy functional.
\end{defi} 

To avoid confusion, we formulate the notion of a weak elliptic problem, which we consider further.

\begin{defi}\cite{Cho24}\label{stat}
Let $f \in \mathring{H^1_{\mu}}.$ We say that a function $u^* \in \mathring{H^1_{\mu}}$ is a solution of the stationary heat equation if 
$$\int_{\Omega} (B\nabla_{\mu}u^*,\nabla_{\mu}v)d\mu=\int_{\Omega} fv d\mu$$
is satisfied for all functions $v \in H^1_{\mu}.$
\end{defi}

Adapting the proof of Theorem 1 of \cite{Gol08, p.269} to the considered case, we obtain the following theorem dealing with the asymptotic behaviour of parabolic solutions. 

\begin{tw}\cite{Cho24}
Let $u^* \in \mathring{H^1_{\mu}}$ be a solution to stationary heat equation (Definition \ref{stat}). A solution $u$ to parabolic issue \eqref{parabolicweak} converges as $t \to \infty$ to $u^*$ in the sense of the $H^1_{\mu}$-norm.
\begin{dow}
A main point of the proof is in verifying that $E_{\mu}(u(t)-u^*) \to 0$ as $t \to \infty.$ 

Let us compute: 
\begin{equation}\label{differ}
\begin{aligned}
\frac{d}{dt}\norm{u-u^*}^2_{L^2_{\mu}}=\int_{\Omega}u_t(u-u^*)d\mu=\int_{\Omega}f(u-u^*)d\mu-\int_{\Omega}(B\nabla_{\mu}u,\nabla_{\mu}(u-u^*))d\mu=\\ \int_{\Omega}f(u-u^*)d\mu-\int_{\Omega}(B\nabla_{\mu}u,\nabla_{\mu}(u-u^*))d\mu+\int_{\Omega}(B\nabla_{\mu}u^*,\nabla_{\mu}(u-u^*))d\mu-\int_{\Omega}f(u-u^*)d\mu=\\-\int_{\Omega}(B\nabla_{\mu}u,\nabla_{\mu}(u-u^*))d\mu+\int_{\Omega}(B\nabla_{\mu}u^*,\nabla_{\mu}(u-u^*))d\mu =\\ -\int_{\Omega}(B\nabla_{\mu}(u-u^*),\nabla_{\mu}(u-u^*))d\mu=-E_{\mu}(u-u^*)\leqslant 0.
\end{aligned}
\end{equation}

This shows that the term $\norm{u-u^*}^2_{L^2_{\mu}}$ is non-increasing. Please note that if $E_{\mu}(w)=0$ for some $w\in \mathring{H^1_{\mu}},$ then $w \equiv 0.$ This means if $u(t) = u^*$ for some $t \in (0,\infty),$ then this implies we have $u(t+s)=u_0$ for all $s>0.$ 

Now we see that as $\norm{u-u^*}^2_{L^2_{\mu}}\geqslant 0,$ there exists a sequence of times $t_i \in (0,\infty)$ such that $\frac{d}{dt}\norm{u_i-u^*}^2_{L^2_{\mu}} \to 0$ as $t_i \to \infty.$ Here we denote $u_i:=u(t_i).$

In this way we get $E_{\mu}(u_i-u^*) \to 0$ as $t_i\to \infty.$ 

As \eqref{differ} shows, the term $\norm{u-u^*}^2_{L^2_{\mu}}$ is non-increasing with respect to $t\in (0,\infty)$ thus we can deduce that the convergence occurs for any sequence of times, that is $E_{\mu}(u-u^*) \to 0$ as $t\to \infty.$ 

The ellipticity of the operator $B$ implies that $\nabla_{\mu}u \xrightarrow{t \to \infty} \nabla_{\mu}u^*$ in the $L^2_{\mu}$-norm sense. Making use of the fact that for the low-dimensional structure $\mu$ the generalized Poincar{\'e} inequality is satisfied combining it with $\int_{\Omega}u(t)d\mu=\int_{\Omega}u^*d\mu=0$ we conclude that $u \xrightarrow{t\to \infty} u^*$ in $\mathring{H^1_{\mu}}.$ 
\qed
\end{dow}
\end{tw}

\begin{kom}\cite{Cho24}
It seems possible to adapt the proof of Theorem 1 of \cite{Gol08} to a wider class of differential equations on low-dimensional structures like, for instance, the p-Laplace equation with the force term dependent on a solution. Nevertheless, this kind of differential equations in the framework of low-dimensional structures has not been studied yet, and we leave it for future studies.
\end{kom}

\newpage

\section{Higher regularity of weak elliptic solutions}

This section contains the results obtained in our paper \cite{Cho23}.

One of the standard ways of showing extra Sobolev regularity in the Euclidean setting is the method of difference quotients - if they are uniformly bounded, they justify the existence of a weak gradient. In a classical setting (see for example \cite[Section 6.3]{Eva10}), one obtains uniform bound on difference quotients of weak solutions, thus improving their regularity. 

We aim to proceed analogously on low-dimensional structures. However, the difficulties in such an approach are readily visible - in general, we cannot easily shift the function supported on the structures we work with in the directions orthogonal to its components. We circumvent this issue in the following way:
\begin{itemize}
    \item For a given $u \in H^1_{\mu}$, we take a sequence of smooth functions $ C^\infty_c(\R^3) \ni \phi_n \xrightarrow{H^1_\mu} u$, which we can arbitrarily shift in any direction; 
    \item As we shall see, we need further control of $\phi_n$, and therefore we introduce a special modification of $\phi_n$, which we denote by $\widehat{\phi_n}$. Its key property is that its global behaviour is determined by the restriction to a low-dimensional structure;
    \item Our modification still obeys the $H^1_\mu$ convergence and its shifts are well-defined. The limit $\lim\limits_{n\rightarrow \infty} \widehat{\phi_n}(\cdot + h)$ plays the role of a generalized translation of $u$. 
\end{itemize}

\begin{uw}\cite{Cho23}
The third point above is not as immediate as it may seem, since for an arbitrary approximating sequence $\phi_n \in C^\infty_c(\R^3)$, justifying the membership $u \in H^1_\mu$, it may happen that $\phi_n (\cdot + h)$ no longer converges in $H^1_\mu$. This is one of the fundamental reasons why we introduce a specific form of the extension. 
\end{uw}

The proof of our main theorem -- $H^2$-type regularity on each component manifold -- consists of two parts: showing that the functional space $H^1_{\mu}$ is closed with respect to the aforementioned generalised translation and secondly, establishing uniform bounds on generalized difference quotients. 

\subsection{Geometric approach to componentwise regularity}{$ $}
\vspace{0.5cm}

In what follows, we restrict our attention to the structures consisting only of "straightened out" components. While such structures serve as generic models of various types of intersections, we later show that the general result follows from the one we obtain in the model cases. We address this issue at the end of this section.

We consider the following structure:

Let $S = S_1 \cup S_2$, where 
\begin{equation}\label{dyski}
\begin{aligned}
S_1 = \{ (x,y,0) \in \R^3 \ : \ x^2 + y^2 \leqslant 1 \}, \quad S_2 =  \{ (0,y,z) \in \R^3 \ : \ y^2 + z^2 \leqslant 1 \}.
\end{aligned}
\end{equation}

We introduce two Hausdorff measures associated with $S$: 
\begin{equation}\label{miary}
\begin{aligned}
\mu := \mathcal{H}^2\mres_{S_1} +\mathcal{H}^2\mres_{S_2}, \quad \widetilde{\mu} := \theta_1\mathcal{H}^2\mres_{S_1}+\theta_2\mathcal{H}^2\mres_{S_2},\quad \theta_i\in C^{\infty}(S_i)\text{ for } i=1,2.
\end{aligned}
\end{equation}

Denote the intersection set by 
$$
\Sigma:=S^1\cap S^2 = \{(0,y,0) \ : \ |y|\leqslant 1 \}
$$ and the restriction of $u$ to a single disc as $u_i:=  u\lvert_{S_i}$. We are ready to state the first result.

Let us now consider the low-dimensional elliptic problem for $S$ previously studied in \cite{Ryb20}. 

For a given $f \in \mathring{L^2_{\mu}},$ a solution $u \in \mathring{H^1_{\mu}}$ satisfies the equality 
\begin{equation}\label{cieplo2}
\int_{\Omega}B_{\widetilde{\mu}} \nabla_{\widetilde{\mu}}u \cdot \nabla_{\widetilde{\mu}} \phi d\widetilde{\mu} = \int_{\Omega}f \phi d\widetilde{\mu}
\end{equation}
for any $\phi \in C^{\infty}(\R^3).$ Existence of the unique solution to this problem is the main result of paper \cite{Ryb20}. 

We may include the densities in the operator $B_{\mu}$ and change $\widetilde{\mu}$-related gradients to $\nabla_{\mu},$ 
\begin{equation}\label{cieplo3}
\int_{\Omega}\widetilde{B_{\mu}} \nabla_{\mu}u \cdot \nabla_{\mu} \phi\, d\mu = \int_{\Omega}\widetilde{f} \phi\, d\mu.
\end{equation}

We now apply the strategy outlined before to the solution $u$.

The below reasoning was established in our paper \cite{Cho23}.

\noindent\textbf{Step 1:} $\mathbf{\tr^{\Sigma} u \in H^1_{\loc}(\Sigma)}$ 

We begin by showing extra differentiability in the direction of the variable $y$. Let us define 
$$
u^{h,y}:= u(\cdot + he_y), \text{ where } e_y=(0,1,0).
$$
Such translation is a well-defined function, at least on any low-dimensional structure $S'$ satisfying $\intel S' \subset \subset \intel S$ (in the inherited topology on $S;$ we refer to $S'$ as a substructure of $S$) and for sufficiently small $h>0.$ We abuse the notation a little and by $\mu$ denote the restriction $\mu\mres_{S'}$; it is justified since all of the further reasoning is carried out locally. 

Taking any $\phi_n\in C^{\infty}(\R^3)$ satisfying $\phi_n \rightarrow u$ in $H^1_\mu$, we see that 
$$
\phi_n^{h,y} \xrightarrow[n\to \infty]{H^1_{\mu}} u^{h,y},
$$ 
and thus $u^{h,y} \in H^1_{\mu}.$

We put 
$$
D^h_yu:= \frac{1}{h}(u^{h,y}-u) \in H^1_\mu.
$$ 

Let $\xi \in C^{\infty}_c(\R^3)$ be a smooth function which will be specified later. We have $\xi^2D^h_yu \in H^1_{\mu}$ and further 
$\phi^y:=-D^{-h}_y(\xi^2D^h_yu) \in H^1_{\mu}.$ This means that $\phi^y$ can be used as a test function in \eqref{cieplo3}.

Since $\Sigma$ is $\mu$-negligible, we are able to rewrite 
$$
\int_{\Omega}\widetilde{B_{\mu}}\nabla_{\mu} u \cdot \nabla_{\mu} \phi^y d\mu = \int_{S_1}\widetilde{B_{\mu}}\nabla_{S_1} u \cdot \nabla_{S_1} \phi^y dS_1 + \int_{S_2}\widetilde{B_{\mu}}\nabla_{S_2} u \cdot \nabla_{S_2} \phi^y dS_2,$$ 

and the right-hand side of the equation \eqref{cieplo3} can be presented as 
$$
\int_{\Omega} \widetilde{f} \phi^y d \mu = \int_{S_1}\widetilde{f} \phi^y dS_1+ \int_{S_2}\widetilde{f} \phi^y dS_2.$$

Having the equality in the expanded form as above, we might apply the standard method of showing higher regularity of solutions by establishing a uniform bound on difference quotients (for example, see Section 6.3.1 in \cite{Eva10}).

Recall that we assumed $\mu$ to be a restriction of the measure $\mathcal{H}^1\mres_{S_1}+ \mathcal{H}^1\mres_{S_2}$ to a fixed subset $S' \subset \subset S.$ 

This provides that for small enough $h>0$ we have $\norm{D^h_y \nabla_{\mu}u}_{L^2_{\mu}}\leqslant C,$ which further implies that $\norm{\partial_y \nabla_{\mu} u}_{L^2_{\mu}} \leqslant C,$ where the constant $C>0$ is independent of $h.$ Using the explicit form of $\nabla_\mu$, we have 
$$\norm{\partial_y^2u}_{L^2_{\mu}}\leqslant C \text{ and } \norm{\partial_y \partial_x u}_{L^2_{\mu}}\leqslant C.$$

Now, we make use of the fact that in a region separated from the junction set $\Sigma$, the equation is reduced to the standard case, and the local regularity is known. Define 
$$
S_1^+:= \{(x,y,0)\in S_1: x>0\}, \quad S_1^-:= \{ (x,y,0) \in S_1: x<0\}.
$$ 

Let $V \subset\subset S_1^+;$ in particular, notice that $\inf\{x: \exists y, (x,y,0)\in V\}\geqslant \alpha>0.$ Taking for $\xi$ properly chosen (as in the classical proof) function supported in $V,$ we conclude by the standard elliptic regularity theory that $u \in H^2_{\loc}(V).$ 

Moreover, the $H^2$-regularity implies almost everywhere in $V$ the symmetry of the second derivatives: $\partial_y\partial_xu=\partial_x\partial_yu.$ By the fact that the set $V$ was chosen arbitrary, we deduce that the equality \hbox{$\partial_y\partial_xu=\partial_x\partial_yu$} is valid a.e. in $S_1^+$ with respect to the measure $\mathcal{H}^2\lfloor_{S_1^+}.$  

Proceeding analogously on $S_1^-$ we get $\partial_y\partial_xu=\partial_x\partial_yu$ a.e. on $S_1^-,$ so the equality $\partial_y\partial_xu=\partial_x\partial_yu$ is true a.e. on $S_1.$ 

By the fact that $\partial_y\partial_xu \in L^2_{\loc}(S_1)$ and $\partial_y\partial_xu=\partial_x\partial_yu$ a.e. on $S_1$ we deduce that $\partial_x\partial_y u \in L^2_{\loc}(S_1).$ We will check that $\partial_x\partial_y u$ is actually a weak $\partial_x$-derivative of $\partial_yu$ on $S_1.$ 

Let 
$$
\mathcal{A}:= \{\phi \in C^{\infty}_c(\R^3): \phi\lvert_{S_1}\in C^{\infty}_c(S_1)\}. 
$$

For any $\phi \in \mathcal{A}$ it follows that 
\begin{align*}
\int_{S_1}\partial_x\partial_yu \phi dS_1 &= \int_{S_1}\partial_y\partial_xu \phi dS_1 = - \int_{S_1}\partial_xu \partial_y\phi dS_1 \\
&= \int_{S_1}u \partial_x\partial_y\phi dS_1 = \int_{S_1}u \partial_y\partial_x\phi dS_1 = - \int_{S_1}\partial_yu\partial_x\phi dS_1,
\end{align*}
thus $\partial_x\partial_yu = \partial_x(\partial_yu),$ so indeed $\partial_x\partial_yu$ is a weak derivative of $\partial_yu.$\\ 

The result implies that $\partial_yu \in H^1(S_1)$, and so $\Sigma$-trace of $\partial_yu$ is well-defined. This implies the membership $\tr^{\Sigma}\partial_yu \in L^2(\Sigma).$

Notice that the standard mollification argument shows that smooth functions $C^{\infty}(\R^2)$ are dense in the subspace 
$$
\{u\in H^1(S_1): \partial_yu \in H^1(S_1)\}
$$ inherited with the norm 
$$
\left(\norm{u}_{H^1(S_1)}^2+\norm{\partial_y u}^2_{H^1(S_1)}\right)^{\frac{1}{2}}.
$$ 

This fact provides the existence of a smooth sequence $\phi_n$ approximating $u$ in the above norm, and further, we can derive the commutation $\tr^{\Sigma}\partial_yu=\partial_y\tr^{\Sigma}u.$ As a conclusion this gives \hbox{$\partial_y\tr^{\Sigma}u \in L^2(\Sigma).$} A verification that $\partial_y\tr^{\Sigma}u$ is a weak derivative of $\tr^{\Sigma}u$ is a byproduct of the proof of the above commutation. Finally, we obtain 
$$
\tr^{\Sigma}u \in H^1(\Sigma).
$$

\noindent\textbf{Step 2: Construction of the extension}

Now we introduce the special form of an extension of a function supported on a low-dimensional structure to the whole $\R^3$. 

Before proceeding with a rather technical construction, we propose an informal description: for a given point $(x,y,z)\in \R^3$ (not necessarily $(x,y,z)\in S$), the extension can be expressed as 
$$
\widetilde{u}(x,y,z) := u(0,y,z) - (\tr^\Sigma u)(0,y,0) + u(x,y,0).
$$
 
By the previous step, we know that $\tr^{\Sigma}u \in H^1(\Sigma).$ Let us recall that the classical trace theory provides the trace operator 
$$ \tr^{\Sigma}:H^1(S_1) \to H^{1/2}(\Sigma)$$ 
which is surjective and $\tr^{\Sigma}(H^{3/2}(S_1))= H^1(\Sigma).$ Thus there exists $v \in H^{3/2}(S_1)$ such that $\tr^{\Sigma}v = \tr^{\Sigma}u.$ It follows that $\tr^{\Sigma}(v-u) = \tr^{\Sigma}v - \tr^{\Sigma}u  = 0.$ Moreover, we have $v-u \in H^1(S^1)$.

Observe that proceeding as in  Step 1, we obtain that for any set $W\subset \subset S_1$ separated from the intersection $\Sigma$ it holds that $u \in H^2(W)$. Let us choose small, fixed $h>0$ and assume that $\Sigma-(h,0,0) \subset W$ for some fixed $W$ as above.

Now, let $\alpha_n, \beta_n \in C^{\infty}(\R^2)$ be sequences converging in the $H^1(S_1)$-norm to $v$ and $v-u$, respectively. Besides that, we demand from the sequence $\alpha_n$ to converge to $v$ in $H^{3/2}(S_1)$ and from $\beta_n$ to converge to $v-u$ in $H^{3/2}(W)$; the latter can be constructed using the diagonal argument. Without the loss of generality, we can assume $\beta_n\lvert_{\Sigma}=0.$ 

The extension of such sequences to the $\R^2$ space is a standard fact; first, we judge the existence of sequences converging on $S_1,$ then each sequence term can be smoothly extended to the whole $\R^2$ space. By the continuity of the trace operator, we have
$$
\tr^{\Sigma}\alpha_n \xrightarrow{H^1(\Sigma)}\tr^{\Sigma}v, \quad \tr^{\Sigma}\beta_n \xrightarrow{H^1(\Sigma)}\tr^{\Sigma}(v-u) = 0.
$$

We express the function $u$ as $$u=v-(v-u)$$ and we put $\gamma_n:= \alpha_n - \beta_n \in C^{\infty}(\R^2).$
Immediately, we see that $\tr^{\Sigma}\gamma_n = \tr^{\Sigma}(\alpha_n - \beta_n) = \tr^{\Sigma}\alpha_n,$ and further 
$$
\gamma_n \xrightarrow{H^1(S_1)} v- (v-u) = u \text{ and } \tr^{\Sigma}\gamma_n \xrightarrow{H^1(\Sigma)}\tr^{\Sigma}v = \tr^{\Sigma}u.
$$

Recall that $\alpha_n$ is defined only on $S^1$, which lies in $xy$ plane. We introduce the $\R^3$-extension of $\alpha_n$ by the formula $\widetilde{\alpha_n}(x,y,z):= \alpha_n(x,y).$ 

We collect some readily seen properties of this extension: $\widetilde{\alpha_n} \in C^{\infty}(\R^3)\cap H^1_\mu,$ 
$$
\widetilde{\alpha_n}\lvert_{S} = \begin{cases}
    \alpha_n &\text{ on } S_1,\\
    \tr^{\Sigma}\alpha_n &\text{ on } S_2
\end{cases}
\qquad \text{and}
\qquad 
\widetilde{\alpha_n} \xrightarrow{H^1_\mu} \begin{cases}
    v &\text{ on } S_1,\\
    \tr^{\Sigma}v &\text{ on } S_2
\end{cases}
= 
\begin{cases}
    v &\text{ on } S_1,\\
    \tr^{\Sigma}u &\text{ on } S_2,
\end{cases}
$$
with the last expression belonging to $H^1_\mu$ (by the completeness of this space).

Analogously we extend $\beta_n$ to $\widetilde{\beta_n}$ and we introduce $\widetilde{\gamma_n}:= \widetilde{\alpha_n}-\widetilde{\beta_n},$ which is exactly the extension of $\gamma_n$ one would obtain by following the same process as for $\alpha_n$ and $\beta_n$. 

Similarly as before, the function $\widetilde{\gamma_n}$ has the following properties: ${\widetilde{\gamma_n} \in C^{\infty}(\R^3)\cap H^1_{\mu},}$	 
$$
\widetilde{\gamma_n}\lvert_{S}=\begin{cases}
\gamma_n &\text{ on } S_1,\\
\tr^{\Sigma}\alpha_n &\text{ on } S_2
\end{cases}
\quad \text{and}
\quad
\widetilde{\gamma_n} \xrightarrow{H^1_{\mu}} \begin{cases}
u &\text{ on } S_1,\\
\tr^{\Sigma}u &\text{ on } S_2
\end{cases}\in H^1_{\mu}.
$$  

Repeating analogous construction on the component $S_2$, we obtain
$\widetilde{\delta_n}\in C^{\infty}(\R^3)\cap H^1_\mu$ such that
$$
\widetilde{\delta_n}\lvert_{S}=\begin{cases}
\tr^{\Sigma}\alpha_n' & \text{ on } S_1,\\
\delta_n &\text{ on } S_2
\end{cases}
\quad \text{and}
\quad
\widetilde{\delta_n} \xrightarrow{H^1_{\mu}} \begin{cases}
\tr^{\Sigma}u &\text{ on } S_1,\\
u, &\text{ on } S_2
\end{cases} \in H^1_{\mu},
$$
where $\alpha'_n$ is defined just as $\alpha_n$ was defined on $S_1$.

Having extended the function $u$, we proceed with extending the trace in a similar fashion. As $\tr^{\Sigma}u \in H^1(\Sigma)$ we might find a sequence $\rho_n \in C^{\infty}(\R)$ converging to $\tr^\Sigma u$ in $H^1(\Sigma)$. 

Define $\widetilde{\rho_n}(x,y,z):= \rho_n(y),\; \widetilde{\rho_n} \in C^{\infty}(\R^3)\cap H^1_\mu.$ The function $\widetilde{\rho_n}$ shares the following properties: 
$$
\widetilde{\rho_n}\lvert_{S}=\begin{cases}
\rho_n & \text{ on } S_1,\\
\rho_n & \text{ on } S_2
\end{cases}
\quad \text{and}
\quad
\widetilde{\rho_n} \xrightarrow{H^1_{\mu}}\begin{cases}
\tr^{\Sigma}u & \text{ on } S_1,\\
\tr^{\Sigma}u & \text{ on } S_2
\end{cases}\in H^1_{\mu}.$$

With all the needed extensions in hand, we define the sequence $\Gamma_n \in C^{\infty}(\R^3):$ 
\begin{equation}\label{extension}
\Gamma_n(x,y,z):= \widetilde{\gamma_n}(x,y,z)+\widetilde{\delta_n}(x,y,z)-\widetilde{\rho_n}(x,y,z).
\end{equation}

\noindent We have $\Gamma_n \in H^1_{\mu}$ and 
$$
\Gamma_n\lvert_{S}= \begin{cases}
\gamma_n+ \tr^{\Sigma}\alpha_n'-\rho_n &\text{ on } S_1,\\
\tr^{\Sigma}\alpha_n+\delta_n-\rho_n, &\text{ on } S_2.
\end{cases}
$$

\noindent Moreover, collecting all the limits of $\widetilde{\gamma_n}, \widetilde{\delta_n}$ and $\widetilde{\rho_n}$ we arrive at
$$
\Gamma_n \xrightarrow{H^1_{\mu}} 
\begin{cases}
u+\tr^{\Sigma}u-\tr^{\Sigma}u & \text{ on } S_1,\\
\tr^{\Sigma}u+u-\tr^{\Sigma}u &\text{ on } S_2
\end{cases} 
\;
=
\; 
\begin{cases}
u &\text{ on } S_1,\\
u &\text{ on } S_2
\end{cases}
\;
= 
\;
u \in H^1_{\mu}.
$$\\

\noindent \textbf{Step 3: Generalized translation and difference quotient}

\noindent We use the sequence $\Gamma_n$ to introduce a generalized notion of translation, i.e. we show how to construct a shift of the function $u$ by the vector $(h,0,0),$ where $h\in \R.$ To fix the perspective, let us choose $h>0$ to be as in the Step 2. Otherwise, we would translate the function in the opposite direction, which can be clearly done by the same method. 

Define
\begin{equation}\label{translation}
\Gamma_n^{h,x}(x,y,z):= \Gamma_n(x+h,y,z)=\widetilde{\gamma_n}(x+h,y,z)+\widetilde{\delta_n}(x+h,y,z)-\widetilde{\rho_n}(x+h,y,z).
\end{equation}

Notice that 
$$
\widetilde{\gamma_n}(x+h,y,z)=\widetilde{\alpha_n}(x+h,y,z)-\widetilde{\beta_n}(x+h,y,z)=\alpha_n(x+h,y)-\beta_n(x+h,y),
$$ where 
$$
\alpha_n(x+h,y)\lvert_{S}=
\begin{cases}
\alpha_n(x+h,y) & \text{ on } S_1,\\
\alpha_n(-h,y) &\text{ on } S_2
\end{cases}
\;
=
\;
\begin{cases}
\alpha_n(x+h,y) &\text{ on } S_1,\\
\tr^{\Sigma}[\alpha_n(\cdot_1+h,\cdot_2)] &\text{ on } S_2
\end{cases}
\in H^1_{\mu}.
$$ 

The sequence $\beta_n$ has an analogous representation on the structure $S$. Therefore, the translated sequence $\widetilde{\gamma_n}$ can be expressed as
\begin{align*}
\widetilde{\gamma_n}(x+h,y,z) &=
\begin{cases}
\alpha_n(x+h,y) - \beta_n(x+h,y) &\text{ on } S_1,\\
\alpha_n(-h,y) - \beta_n(-h,y) &\text{ on } S_2 
\end{cases}\\
&=
\begin{cases}
\alpha_n(x+h,y) - \beta_n(x+h,y) &\text{ on } S_1,\\
\tr^{\Sigma}[\alpha_n(\cdot_1+h,\cdot_2)]-\tr^{\Sigma}[\beta_n(\cdot_1+h,\cdot_2)] &\text{ on } S_2.
\end{cases}
\end{align*}

\noindent Passing to the limit in the $H^1_{\mu}$-norm we obtain the following closed-form expression 
\begin{align*}
\widetilde{\gamma_n}(x+h,y,z) \xrightarrow{H^1_{\mu}}
&\begin{cases}
v(x+h,y) - \left(v(x+h,y)-u(x+h,y)\right) &\text{ on } S_1,\\
\tr^{\Sigma}[v(\cdot_1+h,\cdot_2)]-\left(\tr^{\Sigma}[v(\cdot_1+h,\cdot_2)]-\tr^{\Sigma}[u(\cdot_1+h,\cdot_2)]\right) &\text{ on } S_2
\end{cases}\\
=&\begin{cases}
u(x+h,y) &\text{ on } S_1,\\
\tr^{\Sigma}[u(\cdot_1+h,\cdot_2)] &\text{ on } S_2.
\end{cases}
\end{align*}

From now on, to avoid possible confusion, we will use the notation $u_i:=u\lvert_{S_i},$ for $i=1,2$. Thus the last equality can be written as
$$
\widetilde{\gamma_n}(x+h,y,z) \xrightarrow{H^1_{\mu}}
\begin{cases}
u_1(x+h,y) &\text{ on } S_1,\\
\tr^{\Sigma}\left[u(\cdot_1+h,\cdot_2)\right] &\text{ on } S_2.
\end{cases}
$$

Similar calculations allow us to check the convergence of $\widetilde{\gamma_n}$, which was the extension of $u_2 = u\lvert_{S_2}$. Omitting the details, we arrive at 
\begin{align*}
    \widetilde{\delta_n}(x+h,y,z) &= \widetilde{\alpha_n}'(x+h,y,z)-\widetilde{\beta_n}'(x+h,y,z) \\
    &= \alpha_n'(y,z)(x+h)-\beta_n'(y,z)(x+h)\\
    &= \alpha_n'(y,z)(x)-\beta_n'(y,z)(x)
\end{align*}
(recall that $\alpha', \beta'$ are constructed analogously to $\alpha, \beta$, but this time around on $S_2$).

Moreover, we have
$$
\widetilde{\delta_n}\lvert_{S}= 
\begin{cases}
\tr^{\Sigma}\left(\alpha_n' - \beta_n'\right) &\text{ on } S_1,\\
\alpha_n'(y,z)- \beta_n'(y,z) &\text{ on } S_2
\end{cases}
\;
=
\;
\begin{cases}
\tr^{\Sigma}\alpha_n' &\text{ on } S_1,\\
\delta_n(y,z) &\text{ on } S_2,
\end{cases}
$$
thus after passing with $n \to \infty$ we obtain
$$
\widetilde{\delta_n}(x+h,y,z)\xrightarrow{H^1_{\mu}}
\begin{cases}
\tr^{\Sigma}u_2 &\text{ on } S_1,\\
u_2(y,z) &\text{ on } S_2.
\end{cases}
$$

Lastly, we consider the sequence $\rho_n$ associated with the extension of $\Sigma$--trace. 

We have:
$$
\widetilde{\rho_n}(x+h,y,z)=\rho_n(y)(x+h,z)=\rho_n(y)(x,z)=\rho_n(y),
$$

and further
$$
\widetilde{\rho_n}(x+h,y,z) \xrightarrow{H^1_{\mu}}
\begin{cases}
\tr^{\Sigma}u_1 &\text{ on } S_1,\\
\tr^{\Sigma}u_1 &\text{ on } S_2.
\end{cases}
$$

For the convenience of the readers, we summarise all the obtained limits:
\begin{align*}
    \widetilde{\gamma_n}(x+h,y,z) &\xrightarrow{H^1_{\mu}}
    \begin{cases}
        u_1(x+h,y) &\text{ on } S_1,\\
        \tr^{\Sigma}[u_1(\cdot_1+h,\cdot_2)] &\text{ on } S_2,
    \end{cases} \\
    \widetilde{\delta_n}(x+h,y,z) &\xrightarrow{H^1_{\mu}}
    \begin{cases}
        \tr^{\Sigma}u &\text{ on } S_1,\\
        u_2(y,z) &\text{ on } S_2,
    \end{cases} \\
    \widetilde{\rho_n}(x+h,y,z) &\xrightarrow{H^1_{\mu}}
    \begin{cases}
        \tr^{\Sigma}u &\text{ on } S_1,\\
        \tr^{\Sigma}u &\text{ on } S_2.
    \end{cases}
\end{align*}

Recalling the definition of $\Gamma^{h,x}_n$ in \eqref{translation}, we arrive at 
\begin{align*}
\Gamma_n^{h,x} \xrightarrow{H^1_{\mu}} 
&\begin{cases}
u_1(x+h,y)+\tr^{\Sigma}u - \tr^{\Sigma}u &\text{ on } S_1,\\
\tr^{\Sigma}[u_1(\cdot_1+h,\cdot_2)]+u_2(y,z) - \tr^{\Sigma}u &\text{ on } S_2
\end{cases}\\
=
&\begin{cases}
u_1(x+h,y) &\text{ on } S_1,\\
\tr^{\Sigma}[u_1(\cdot_1+h,\cdot_2)]+u_2(y,z) - \tr^{\Sigma}u_1 &\text{ on } S_2
\end{cases}
 \in H^1_{\mu}.
\end{align*}

\noindent This limit will be treated as a generalized translation of $u$. 

From now on, we denote 
$$
u^{h,x}
:=
\begin{cases}
u_1(x+h,y) &\text{ on } S_1,\\
\tr^{\Sigma}[u_1(\cdot_1+h,\cdot_2)]+u_2(y,z) - \tr^{\Sigma}u_1 &\text{ on } S_2.
\end{cases}.
$$

In a completely analogous way, we introduce the z-direction translation: $u^{h,z}.$

With the notion of translation at hand, we define generalized difference quotients as 
$$
D^{h,x}u:=\frac{1}{h}(u^{h,x}-u), \quad D^{h,z}u:=\frac{1}{h}(u^{h,z}-u).
$$ 

We have 
$$
D^{h,x}u,\; D^{h,z}u \in H^1_{\mu}.
$$ 
Note that $u^{h,x}$ and $u^{h,z}$ coincide with classical translations on $S_1$ and $S_2$, respectively.

\textbf{Step 4:} Higher regularity estimates

In this step, we use the notion of generalized difference quotients to establish higher regularity.

Let $\phi:=-\xi^2D^{-h,x}(u)\in H^1_{\mu}$, where $\xi \in C^{\infty}(\R^3),\; \supp (\xi\lvert_{S}) \subset \subset S$ is the properly chosen cut-off function to be specified later. Due to Remark \ref{H1 jako testowe}, it is a suitable test function in weak formulation \eqref{weak} and thus in an equivalent formulation \eqref{cieplo3}. 

Proceeding with the latter one, we have
\begin{align}
\begin{split}
\int_{\Omega}\widetilde{B_{\mu}}\nabla_{\mu}u \cdot \nabla_{\mu}(D^{h,x}\phi) d\mu 
&=\int_{S_1}\widetilde{B_{\mu}}\nabla_{S_1}u \cdot \nabla_{S_1} D^{h,x}\phi dS_1 + \int_{S_2}\widetilde{B_{\mu}}\nabla_{S_2}u \cdot \nabla_{S_2} D^{h,x}\phi dS_2\\
&= \int_{S_1} \widetilde{B_{\mu}}\nabla_{S_1}u \cdot\nabla_{S_1} D^h_x\phi dS_1 + \int_{S_2} \widetilde{B_{\mu}}\nabla_{S_2}u \cdot\nabla_{S_2} D^{h,x}\phi dS_2.
\end{split}
\end{align}
Here $D^h_x$ is a classical difference quotient, which follows from the observation that the generalised translation $\phi^{x,h}$ coincides with a classical one (in $x$ direction still) on $S_1$, as mentioned before. 

Therefore the integral over $S_1$ already is in a proper form, but we need to show that a constant independent of $h$ uniformly bounds the second term on the right-hand side. 

We have 
\begin{align}
\begin{split}
&\int_{S_2}\widetilde{B_{\mu}} \nabla_{S_2}u \cdot \nabla_{S_2}D^{h,x}\phi dS_2 =
\int_{S_2} \widetilde{B_{\mu}}\nabla_{S_2}u \cdot \nabla_{S_2}\left[\frac{1}{h}(\phi^{h,x}-\phi)\right]dS_2 \\
= &\int_{S_2} \widetilde{B_{\mu}}\nabla_{S_2}u \cdot \nabla_{S_2} \left[ \frac{1}{h}(\tr^{\Sigma}[\phi_1(\cdot_1+h,\cdot_2)]+\phi_2(y,z)-\tr^{\Sigma}[\phi_1(\cdot_1,\cdot_2)]-\phi_2(y,z))\right]dS_2 \\
= &\int_{S_2} \widetilde{B_{\mu}}\nabla_{S_2}u \cdot \nabla_{S_2} \left[\frac{1}{h}(\tr^{\Sigma}[\phi_1(\cdot_1+h,\cdot_2)]-\tr^{\Sigma}[\phi(\cdot_1,\cdot_2)])\right]dS_2.
\end{split}
\end{align}

We estimate this last integral. In the below computations, the constant $C>0$ varies from line to line and is independent of the parameter $h.$

\begin{align}\label{szacowanko}
\begin{split}
&\left|\int_{S_2} \widetilde{B_{\mu}}\partial_yu \partial_y \left[\frac{1}{h}\left(\tr^{\Sigma}[\phi_1(\cdot_1+h,\cdot_2)]-\tr^{\Sigma}[\phi_1(\cdot_1,\cdot_2)]\right) \right]dydz\right|\\
\leqslant &\left|\int_{S_2} \widetilde{B_{\mu}}\partial^2_yu \left[\frac{1}{h}\left(\tr^{\Sigma}[\phi_1(\cdot_1+h,\cdot_2)]-\tr^{\Sigma}[\phi_1(\cdot_1,\cdot_2)]\right)\right]dydz\right|\\
+&\left|\int_{S_2}\partial_y\widetilde{B_{\mu}} \partial_yu \left[\frac{1}{h} \left(\tr^{\Sigma}[\phi_1(\cdot_1+h,\cdot_2)]-\tr^{\Sigma}[\phi_1(\cdot_1,\cdot_2)]\right)\right]dydz\right|\\
\leqslant \;&C\norm{\partial_yu}_{L^2(S_2)} \left\Vert\frac{1}{h}(\tr^{\Sigma}[\phi_1(\cdot_1+h,\cdot_2)]-\tr^{\Sigma}[\phi_1(\cdot_1,\cdot_2)])\right\Vert_{L^2(S_2)}\\ 
+ &C \norm{\partial_y^2 u}_{L^2(S_2)} \left\Vert\frac{1}{h}(\tr^{\Sigma}[\phi_1(\cdot_1+h,\cdot_2)]-\tr^{\Sigma}[\phi_1(\cdot_1,\cdot_2)])\right\Vert_{L^2(S_2)}\\ 
= &C(\norm{\partial_yu}_{L^2(S_2)}+\norm{\partial_y^2 u}_{L^2(S_2)})\left\Vert\frac{1}{h}(\tr^{\Sigma}[\phi_1(\cdot_1+h,\cdot_2)]-\tr^{\Sigma}[\phi_1(\cdot_1,\cdot_2)])\right\Vert_{L^2(S_2)}\\
\leqslant &C(\norm{\partial_yu}_{L^2(S_2)}+\norm{\partial_y^2 u}_{L^2(S_2)})\left\Vert\frac{1}{h}(\tr^{\Sigma}[\phi_1(\cdot_1+h,\cdot_2)]-\tr^{\Sigma}[\phi_1(\cdot_1,\cdot_2)])\right\Vert_{L^2(\Sigma)}\\
\leqslant\; &C \left(\norm{\partial_y^2 u}_{L^2(S_2)}+\norm{\partial_yu}_{L^2(S_2)}\right) \left\Vert\frac{1}{h}([\phi_1(\cdot_1+h,\cdot_2)]-[\phi_1(\cdot_1,\cdot_2)])\right\Vert_{L^2(S_1)}.
\end{split}
\end{align}
 
 The second estimate follows from the assumption that $B_{\mu}\lvert_{S_i} \in W^{1,\infty}(S_i),\; i=1,2.$ 

In the second to last inequality, we used the fact that traces are independent of the $z$-variable. This implies that it is possible to estimate it from above by the integral over $\Sigma.$ 

The last inequality results from the linearity and the boundedness of the trace operator. This estimate is meaningful as it allows us to change integration over $\Sigma$ to integration over $S_1$ (where the $x$-variable is 'contained'). 

Using the Young inequality with $\eps,$ we obtain
\begin{align}
&C \left(\norm{\partial_y^2 u}_{L^2(S_2)}+\norm{\partial_yu}_{L^2(S_2)}\right) \left\Vert\frac{1}{h}([\phi_1(\cdot_1+h,\cdot_2)]-[\phi_1(\cdot_1,\cdot_2)])\right\Vert_{L^2(S_1)}\\
\leqslant \;&C\left(\frac{1}{\eps^2}(\norm{\partial_y^2 u}_{L^2(S_2)}+\norm{\partial_yu}_{L^2(S_2)})+\eps\norm{D^h_x\phi_1}_{L^2(S_1)}\right).
\end{align}

As in the classical proof, we choose $\eps>0$ small enough and move the term $\eps\norm{D^h_x\phi_1}_{L^2(S_1)}$ to the left-hand side of \eqref{cieplo3}, obtaining the expression of the type
$$\widetilde{C} \norm{D^h_x\phi_1}_{L^2(S_1)} \leqslant C + \int_{\Omega}\widetilde{f} D^{h,x}\phi d\mu,$$
with positive constants $\widetilde{C}, C$ independent of the parameter $h.$

We deal with the right-hand side of the equation analogously. 
Firstly, we decompose the integral as
$$
\int_{\Omega}\widetilde{f} D^{h,x}\phi\, d\mu = \int_{S_1}\widetilde{f} D^h_x \phi\, dS_1 + \int_{S_2}\widetilde{f} D^{h,x}\phi \,dS_2.
$$

The integral over $S_1$ is in a standard form and can be treated as in the classical proof. To deal with the integral over $S_2$, we use the analogous estimates as in the computations presented in estimation \eqref{szacowanko}. Applying the standard difference quotients method (i.e. also fixing a function $\xi$) we conclude a uniform boundedness of the first integral; this is clearly explained in the mentioned Section 6.3.1 of the book of Evans \cite{Eva10}. 

In this way, we end up with the formula
$$
\norm{D^h_x \phi_1}_{L^2(S_1)} \leqslant C,
$$ 

where $C>0$ is again independent of $h.$ 
This implies that
$$
\partial^2_xu \in L^2(S_1)
$$ and $\partial^2_xu$ is a second weak derivative of the function $u.$ We can invoke the analogous procedure to provide extra regularity with respect to the $z$-variable.

Concluding all the reasoning, we arrive at
\begin{tw}\cite{Cho23}\label{wazne}
Assume $\mu \in \mathcal{S}$ is as in \eqref{miary} and $\supp \mu$ is as in \eqref{dyski}. Let $u \in H^1_{\mu}$ be a solution to Problem \eqref{cieplo3} with the coefficients matrix $B_{\mu}\lvert_{S_i}\in W^{1,\infty}(S_i),\; i=1,2.$ Then for $i=1,2$ we have $u\in H^2_{\loc}(S_i).$ $\qedhere$
\end{tw}

The same reasoning can also be applied to establish the analogous result if $\dim S_1 = \dim S_2 = 1$, which we state in the following

\begin{tw}\cite{Cho23}
Let 
\begin{equation*}
S_1:=\{(x,0,0)\in \R^3: x\in [-1,1]\}, \quad S_2:=\{(0,0,z)\in \R^3: z\in [-1,1]\},
\end{equation*}
\begin{equation*}
S:=S_1 \cup S_2, \quad \mu:= \mathcal{H}^1\lfloor_{S_1}+\mathcal{H}^1\lfloor_{S_2}.
\end{equation*}
 Let $u \in H^1_{\mu}$ be a solution to \eqref{cieplo3} on $S$ and $B_{\mu}\lvert_{S_i}\in W^{1,\infty}(S_i),\; i=1,2.$ Then for $i=1,2$ the solution satisfies $u\in H^2_{\loc}(S_i).$ 
\end{tw}

\begin{dow}
Observe that $S \subset \{(x,0,z)\in \R^3: x,z\in\R\}.$ This allows us to work in a two-dimensional subspace and then trivially extend constructions to $\R^3.$ The proof is based on completely analogous constructions and reasoning as in the two-dimensional regularity theorem proved above. To avoid repetition, we skip the presentation of the proof here.
\qed
\end{dow}\\

Now we can give a short proof of the higher regularity of weak solutions in the instance of components of different dimensions.

\begin{tw}\cite{Cho23}\label{sobreg}
Assume that $\dim S_1 \neq \dim S_2.$ If $u \in H^1_{\mu}$ is a weak solution of \hbox{Problem \eqref{cieplo3}} with $B_{\mu}\lvert_{S_i}\in W^{1,\infty}(S_i),\; i=1,2,$ then $u \in H^2_{\loc}(S_i),\; i=1,2.$
\end{tw}
\begin{dow}
For $i=1,2$ denote $b^i:=-D^{-h}(\xi^2 D^h u^i),$ with a smooth function $\xi$ chosen as in the classical proof. By Proposition \ref{zero_extension}, we know that the functions 
$$
a^1:=
\begin{cases}
b^1, &S_1\\
0, & S_2
\end{cases} 
\quad
\text{and}
\quad
a^2:=
\begin{cases}
0, &S_1\\
b^2, &S_2
\end{cases}
$$ are in $H^1_\mu$ and therefore we can use them as test functions. This implies that the considered low-dimensional problem reduces to the classical one, and the classical proof yields the desired regularity. \qed
\end{dow}

As a last effort in this section, we address two assumptions previously made to simplify the proof of Theorem \ref{wazne}. Firstly, we assumed that without the loss of generality, we can restrict our attention to structures composed of two sub-manifolds. Secondly, we considered only "straightened-out" domains. 

In what follows, we show that our construction easily carries on to the general setting.

Let us begin with discussing how to pass from the case of two intersecting manifolds, both being subsets of either $\R$ or $\R^2$, to a general case of a low-dimensional structure consisting of multiple components (which we assume to be "straightened-out" still). 

Let 
$$
S=\bigcup_{i=1}^m S_i \quad \text{and} \quad \mu:= \sum_{i=1}^m\mathcal{H}^{\dim S_i}\mres_{S_i}.
$$
Keeping in line with previous notation, we will denote for $1 \leqslant i\neq j \leqslant m$ 
$$
\Sigma_{ij}:= S_i\cap S_j.
$$ 

Recall that the intersections $\Sigma_{ij}$ are mutually isolated, i.e. for each two $\Sigma_{ij} \neq \Sigma_{i'j'}$ there exist fixed open (in $\R^3$) sets $O_{ij}$ and $O_{i'j'}$ satisfying 
\begin{equation}\label{pokrywka}
\begin{aligned}
\Sigma_{ij} \subset O_{ij}, \quad \Sigma_{i'j'} \subset O_{i'j'} \quad \text{and} \quad O_{ij} \cap O_{i'j'} = \emptyset.
\end{aligned}
\end{equation} 

Let us put 
$$
C^{\infty}_{ij}:= \{\phi \in C^{\infty}(\R^3): \supp\phi\subset O_{ij}\} \quad \text{and} \quad \mu_{ij}:= \mu\mres_{O_{ij}} \text{ for } 1\leqslant i,j \leqslant m.
$$

Then if $u \in H^1_{\mu}$ solves \eqref{cieplo3} it also satisfies 
\begin{equation}
\int_{O_{ij}}\widetilde{B_{\mu}}\nabla_{\mu_{ij}}u \cdot \nabla_{\mu_{ij}} \phi d\mu_{ij} = \int_{O_{ij}}\widetilde{f} \phi d\mu_{ij}
\end{equation}
for all $\phi \in C^{\infty}_{ij}.$ Now we can conclude that if $S_{ij}:= O_{ij} \cap S$ and we know that $u \in H^2_{\loc}(S_{ij}),$ then $u \in H^2_{\loc}(S_i) $ for all $1\leqslant i \leqslant m$. 

Secondly, we address the "flatten-out" assumption. It was justified since after decomposing the general structure $S$ into substructures $S_{ij},\; 1\leqslant i,j \leqslant m$ it is possible to use the proper composition of diffeomorphisms to obtain the demanded "flat" structures. As we work with closed manifolds, the diffeomorphic changes of manifolds produce some smooth densities that are bounded and isolated from zero. This means that such changes do not impact the convergence in the norms of the spaces $H^1_{\mu},$ $H^2(S_i)$ or $D(A_{\mu})$. For the construction and detailed discussion of the related diffeomorphisms, see the beginning of the proof of Theorem \ref{main} in Subsection 4.1.

Collecting all the elements of our reasoning, we arrive at the
\begin{tw}\cite{Cho23}\label{globreg}
Let 
$$
S = \bigcup_{i=1}^m S_i \quad \text{and} \quad \mu = \sum_{i=1}^m\mathcal{H}^{\dim S_i}\mres_{S_i}.
$$
If $u \in H^1_{\mu}$ is a weak solution to  \eqref{cieplo3} on $S$ with $B_{\mu}\lvert_{S_i}\in W^{1,\infty}(S_i),\; 1 \leqslant i \leqslant m.$ Then ${u\in H^2_{\loc}(S_i)}$ for $1 \leqslant i \leqslant m.$ $\qedhere$
\end{tw}

\subsection{Continuity of solutions}{$ $}
\vspace{0.5cm}

In this section, we show how we can apply the main regularity result to obtain continuity of solutions in the case of structures with a constant dimension of components. 

In other words, we prove that if a measure $\mu$ belongs to the class $\mathcal{S}$ and moreover
\begin{equation}\label{const_dim}
\supp\mu = S=\bigcup_{i=1}^mS_i, \quad \dim S_1 = \ldots = \dim S_m = k = 1, 2,
\end{equation} then weak solutions to elliptic Problem \eqref{cieplo3} are continuous.
 
\begin{tw}(Continuity of solutions)\cite{Cho23}\label{ciag}\\
Let $S=S_1\cup S_2,\; \dim S_1 = \dim S_2.$ Let $u \in H^1_{\mu}$ be a weak solution of Problem \eqref{weak} and ${B_{\mu}\lvert_{S_i}\in W^{1,\infty}(S_i),}$ for  $i = 1, 2.$ Then $u \in C(S).$
\end{tw}
\begin{dow}
By regularity Theorem \ref{globreg}, we know that $u \in H^2_{\loc}(S_i),\; i=1,2.$ By Proposition \ref{slady}
$$
\tr^{\Sigma}u_1=\tr^{\Sigma}u_2 \quad \text{a.e. on } \Sigma.
$$  As $u\in H^2_{\loc}(S_i)$ implies $u \in C(S_i),$ we have $\tr^{\Sigma}u_i=u_i\lvert_{\Sigma}.$ In this way we conclude that $u_1\lvert_{\Sigma}=u_2\lvert_{\Sigma}$ and thus $u \in C(S).$  \qed    
\end{dow}

\begin{uw}\cite{Cho23}
We make the following observations regarding the continuity of $u$:
\begin{enumerate}
\item[a)] In the case $d=1,$ we do not use the fact that $u$ is a weak solution to \eqref{weak}. This shows that continuity is a general property of the space $H^1_{\mu}$ considered on structures built with one-dimensional components.
\item[b)] In the case $d=2$ there exist examples of functions which belong to $H^1_{\mu}$ and are discontinuous. In light of Theorem \ref{ciag}, any possible discontinuities might come only from the lack of continuity of $u|_{S_i} \in H^1(S_i)$ for some component $S_i.$
\item[c)] For a general $\mu \in \mathcal{S}$, not necessarily satisfying $\dim S_1 = \ldots = \dim S_m$, we cannot ensure that $u\in H^1_\mu$ solving \eqref{weak} is continuous. Intuitively, this is due to a fact that intersection of some components will be a point, and a single point has zero $W^{1,2}$-capacity in a plane.
\end{enumerate}
\end{uw}

\noindent Note that Theorem \ref{ciag} addresses the case of structures consisting of two components only. As it turns out, we can easily extend this result.

\begin{tw}\cite{Cho23}\label{ciaglosc_ogolne}
Assume that $S$ and $\mu \in \mathcal{S}$ satisfy \eqref{const_dim} and let $u\in H^1_{\mu}$ satisfy \eqref{weak} with ${B_{\mu}\lvert_{S_i}\in W^{1,\infty}(S_i),}$ $1 \leqslant i \leqslant m.$ Then $u \in C(S)$. 
\begin{dow}
By Theorem \ref{ciag} we have that $u \in C(O_{ij})$ for $1 \leqslant i,j \leqslant m,$ where the sets $O_{ij}$ are elements of the covering as in \eqref{pokrywka}. Regularity Theorem \ref{sobreg} implies that $u \in H^2_{\loc}(S_i),\; 1 \leqslant i \leqslant m,$ thus $u \in C(S_i).$ This two results immediately yield $u \in C(S).$ \qed
\end{dow}
\end{tw}

\subsection{Membership in the domain of the $\mu$-related second-order operator}{$ $}
\vspace{0.5cm}

Having established the higher Sobolev-type regularity, it is only natural to consider whether it is possible to obtain a strong form of the original equation. To this end, it is necessary to introduce second-order differential operators defined on structures in $\mathcal{S}$. It turns out to be quite an involved and delicate matter. 

We focus on showing that thanks to the regularity result Subsection 6.1, it is possible to show that a weak solution $u$ of the elliptic problem belongs to the domain of the low-dimensional second-order differential operator.

Firstly, we need the basic notions of the low-dimensional second-order framework.
The definition of the second-order operator $A_{\mu},$ its domain $D(A_{\mu}),$ the operator $\nabla^2_{\mu},$ the Cosserat vector field $b$ and discussion of their basic properties can be found in Subsection 2.3 and Subsection 3.2, respectively.

The regularity result of Section 6.1 is local -- we have shown $H^2$-regularity on subsets that are compactly embedded in component manifolds of a low-dimensional structure. To avoid extensive technicalities, we do not aim for up-to-the-boundary regularity and choose to keep working in a local setting. This calls for a further modification. 

Let $u \in H^1_\mu$ being a solution to \eqref{cieplo3}. Let us define a low-dimensional structure $\overline{S}$ satisfying $ \overline{S} \subset\subset \intel S$ and consider $\overline{\mu}:= \mu\mres_{\overline{S}}$. Abusing notation a bit, we will denote $\overline{\mu}$ as $\mu$. This allows us to write $u \in H^2(S_i),$ where $S_i$ is a component manifold of $\supp \overline{\mu}.$ Let us clearly state that with this procedure, we do not bother with behaviour on the boundary.

As an intermediate step, we need to establish an auxiliary result showing that the smoothness of the force term is propagated to the solution.

\begin{prop}\cite{Cho23}\label{gladkiereg}
Let $u \in H^1_{\mu}$ be a solution to \eqref{weak} with $B_{\mu} = \text{Id}_{\mu}$ and $f\in C^\infty(\R^3)$. Then $u \in C^{\infty}(\R^3).$
\end{prop}
\begin{dow}
Take the covering of the structure $S$ with $\R^3$-open sets $O_{ij}, \; i,j\in \{1,...,m\}$ such that $\Sigma_{ij}=S_i\cap S_j \subset O_{ij}$ and $O_{ij}$ is isolated from $\Sigma_{i'j'}$ if $i'\neq i$ or $j' \neq j.$ 

Let $\phi \in C^{\infty}(\R^3),$  ${\phi_i := \phi\lvert_{S_i} \in C^{\infty}_c(S_i),\; i=1,...,m,}$ and - for simplicity - let us assume that \hbox{$\supp \phi \cap S \subset O_{12}$.} The same argument works for an arbitrary covering element $O_{ij}, \, i\neq j$. 

As the support of $\phi$ touches only $S_1$ and $S_2,$ we decompose the left-hand side of \eqref{cieplo3} as
$$
\int_{\Omega}\nabla_{\mu}u \cdot \nabla_{\mu} \phi d\mu = \int_{S_1} \nabla_{S_1}u_1 \cdot \nabla_{S_1} \phi_1dS_1 + \int_{S_2} \nabla_{S_2} u_2 \cdot \nabla_{S_2}\phi_2dS_2.
$$ 

After integrating by parts we obtain, for $i=1,2$,
$$
\int_{S_i}\nabla_{S_i} u_i \cdot \nabla_{S_i}\phi_i dS_i = - \int_{S_i}(\Delta_{S_i}u_i)\phi_i dS_i.
$$

Plugging this into equation \eqref{cieplo3} and moving the $S_2$-integral to the right-hand side, we arrive at
\begin{equation}\label{hauhau}
-\int_{S_1}(\Delta_{S_1}u_1)\phi_1dS_1 = \int_{S_2}(\Delta_{S_2}u_2)\phi_2 dS_2 + \int_{S_1}f_1 \phi_1 dS_1 + \int_{S_2}f_2 \phi_2 dS_2.
\end{equation}

\noindent Now, let $\phi_2^\epsilon \in C^\infty(\R^3)$ satisfy
$$
\supp\phi_2^{\eps} \subset \{(x,y,z)\in \R^3: |z| < \eps\} 
\quad 
\text{and} 
\quad 
\phi_2^{\eps}=\phi_2 \text{ on } \left\{(x,y,z)\in \R^3: |z| < \frac{\eps}{2}\right\}
$$
and consider 
$$
\phi^\epsilon := \begin{cases}
    \phi_1 &\text{ on } S_1,\\
    \phi_2^\epsilon &\text{ on } S_2.
\end{cases}
$$
By the construction, we have $\phi^{\eps} \in H^1_{\mu}$, thus it is an admissible test function. 

Since in equation \eqref{hauhau} no derivative acts on either $\phi_1$ or $\phi_2$, we can easily pass with $\eps$ to zero, obtaining 
$$
\int_{S_2} (\Delta_{S_2}u_2)\phi_2^{\eps}dS_2 \xrightarrow{\eps \to 0}0 
\text{ and } 
\int_{S_2}f_2 \phi_2 dS_2 \xrightarrow{\eps \to 0}0.
$$

As a consequence, this gives 
$$
-\int_{S_1}(\Delta_{S_1}u_1)\phi_1dS_1 = \int_{S_1}f_1\phi_1dS_1.
$$

Since $\phi_1 \in C^{\infty}_c(O_{12} \cap S_1)$ can be chosen arbitrary, we conclude 
$$
-\Delta_{S_1}u_1 = f_1 \text{ a.e. in } O_{12} \cap S_1.
$$ 
This implies that $u_1 \in C^{\infty}(O_{12}\cap S_1).$ We proceed analogously on the second component manifold $S_2$ and obtain $u_2 \in C^{\infty}(O_{12}\cap S_2).$ Substituting $O_{12}$ with other elements of the covering, we conclude $u_i \in C^{\infty}(S_i)$ for all $1 \leqslant i \leqslant m.$ Applying the Whitney Extension Theorem to $u$, we obtain $u \in C^{\infty}(\R^3)$.    
\qed
\end{dow}

With the above result at hand, we show that $u$ belongs to the domain of $A_\mu$. 

For clarity in the presentation of the next result, let us abandon the identification of the structure $S = \supp \mu$ with its compactly embedded substructure $\overline{S} = \supp \overline{\mu}.$ 

\begin{tw}\cite{Cho23}\label{nalezenie}
Let $u \in H^1_{\mu}$ satisfy equation \eqref{weak} with $B_{\mu} = \text{Id}_{\mu}.$ Assume that $\overline\mu,\;$ \hbox{$\supp{\overline\mu}=\overline{S}$} is a low-dimensional measure such that $\intel \overline{S} \subset \subset \intel S$ (in the inherited \hbox{topology on $S$).} Then there exists $b \in L^2_{\mu}(\R^3; T_{\mu}^{\perp})$ such that $(u,b) \in D(A_{\overline\mu}).$
\end{tw}
\begin{dow}
Firstly, we deal with the case of a regular right-hand side. Let $f \in \mathring{L^2_{\mu}}$ and additionally let us assume that $f \in C^{\infty}(\R^3).$ 

By Proposition \ref{gladkiereg} we know that $u\in \mathring{H^1_{\mu}}$ being a solution of 
$$
\int_{\Omega}\nabla_{\mu}u \cdot \nabla_{\mu}\phi d\mu = \int_{\Omega} f\phi d\mu \quad \forall \phi \in C^\infty(\R^3)
$$ has the property that $u\lvert_{\overline{S}}$ can be extended to $u\lvert_{\overline{S}} \;\in C^{\infty}(\R^3),$ where $\overline{S}$ is as before a compactly embedded substructure.   

Now, for a given $h \in \mathring{L^2_{\mu}}$, let $u^h\in \mathring{H^1_{\mu}}$ be a solution to \eqref{weak} with $h$ as the right hand side. 

Consider a sequence $g_n$ such that
$$
g_n \in C^\infty(\R^3) \cap \mathring{L^2_{\mu}}, \quad g_n \xrightarrow{L^2_\mu} h
$$
and denote solution corresponding to the right-hand side $g_n$ as $v_n \in \mathring{H^1_\mu}.$ 

In other words, we have 
\begin{align*}
    \int_{\Omega} \nabla_\mu v_n \cdot \nabla_\mu \phi \, d \mu &= \int_{\Omega} g_n \phi \, d\mu, \\
    \int_{\Omega} \nabla_\mu u^h \cdot \nabla_\mu \phi \, d \mu &= \int_{\Omega} h   \phi \, d\mu.
\end{align*}

Subtracting both sides, we obtain 
$$
\int_{\Omega}\nabla_{\mu}(v_n-u^h) \cdot \nabla_{\mu}\phi \,d\mu = \int_{\Omega}(g_n-h)\phi \,d\mu
$$
which yields the estimate
\begin{equation*}\label{nierownosc}
\left|\int_{\Omega}\nabla_{\mu}(v_n-u^h)\cdot \nabla_{\mu}\phi \, d\mu \right| 
\leqslant \int_{\Omega}|g_n-h||\phi| \, d\mu
\end{equation*}

Choosing $\phi:= v_n -u^h \in H^1_\mu$ as a test function, we get
\begin{equation*}\label{nierownosc2}
\int_{\Omega}|\nabla_{\mu}(v_n-u^h)|^2 d\mu\leqslant \int_{\Omega}|g_n-h||v_n-u^h|d\mu
\end{equation*}
and after applying the Young inequality with the $\eps$ we conclude
\begin{equation*}\label{nierownosc3}
\begin{aligned}
C\int_{\Omega}|\nabla_{\mu}(v_n-u^h)|^2 d\mu\leqslant \norm{g_n-h}^2_{L^2_{\mu}}, 
\end{aligned}
\end{equation*}
for some positive constant $C.$ 

Passing with $n$ to infinity, we see that 
$$
\nabla_{\mu} v_n \xrightarrow{L^2_{\mu}}\nabla_{\mu}u^h.
$$

Notice that $\int_{\Omega}v_nd\mu=0$ and $\int_{\Omega}u^hd\mu=0.$ Now, the weak Poincar{\'e} inequality \eqref{weakpoincare} implies that 
$$
v_n \xrightarrow{H^1_\mu} u^h.
$$ 

Now, we restrict further considerations on a compactly embedded $\overline{S} \subset S$ with the corresponding measure $\overline{\mu}.$ 
Theorem \ref{main} provides that the operator $\Delta_{\overline\mu}: D(A_{\overline\mu}) \to L^2_{\mu}$ is closed. 

From a definition of the domain $D(A_{\overline\mu})$ it follows that $v_n \in C^{\infty}(\R^3)$ implies $v_n \in D(A_{\overline\mu}).$ Since we already know  
$$
v_n\xrightarrow{L^2_{\overline\mu}}u^h, 
\quad 
\Delta_{\overline\mu}v_n = g_n\xrightarrow{L^2_{\overline\mu}}h,
$$ 
we derive 
$$
u^h \in D(\Delta_{\overline\mu}) \text{\; and \;} \Delta_{\overline\mu}u^h=h.
$$ 

By the definition of the domain $D(\Delta_{\overline\mu})$, there exists 
$b_{u^h} \in L^2_{\overline\mu}(\R^3; T_{\overline\mu}^{\perp})$ such that $$(u^h,b_{u^h})\in D(A_{\overline\mu}).$$

$\qedhere$
\end{dow}

We decided to introduce the assumption $B_{\mu} = \text{Id}_{\mu}$ in order to simplify the proofs of \hbox{Proposition \ref{gladkiereg}} and Theorem \ref{nalezenie}. Without any qualitative changes, we can follow the methods of those proofs and generalise the results to an arbitrary matrix $B_{\mu}$ satisfying Proposition \ref{relax}. For the sake of readability, we omit the details.\\

\end{document}